\documentclass[11pt,a4paper]{article}

\usepackage[T1]{fontenc}
\usepackage{microtype}

\usepackage{amsmath}								
\usepackage{amssymb}
\usepackage{amsthm}
\usepackage{amscd}
\usepackage{amsfonts}
\usepackage{stmaryrd}

\usepackage{euler}

\usepackage{extarrows}

\usepackage[colorlinks, linktocpage, citecolor = blue, linkcolor = blue]{hyperref}
\usepackage{color}
\usepackage{textgreek}
\usepackage{tikz}									
\usetikzlibrary{matrix}
\usetikzlibrary{patterns}
\usetikzlibrary{matrix}
\usetikzlibrary{positioning}
\usetikzlibrary{decorations.pathmorphing}
\usetikzlibrary{cd}

\usepackage{fullpage}
\usepackage{enumerate}

\newtheorem*{theorem*}{Theorem}

\newtheorem{maintheorem}{Theorem}

\newtheorem{theorem}{Theorem}[section]
\newtheorem{lemma}[theorem]{Lemma}
\newtheorem{proposition}[theorem]{Proposition}
\newtheorem{corollary}[theorem]{Corollary} 
\newtheorem{conjecture}[theorem]{Conjecture}

\theoremstyle{definition}
\newtheorem{definition}[theorem]{Definition}
\newtheorem{example}[theorem]{Example}

\newtheorem{remark}[theorem]{Remark}

\newcommand{\R}{\mathbb{R}}

\newcommand{\Z}{\mathbb{Z}}
\newcommand{\ZZ}{\mathbb{Z}}

\newcommand{\PP}{\mathbb{P}}

\newcommand{\RR}{\mathbb{R}}

\newcommand{\Mbar}{\overline{M}}

\newcommand{\calMbar}{\overline{\mathcal{M}}}

\newcommand{\calD}{\mathcal{D}}
\newcommand{\calDR}{\mathcal{DR}}

\newcommand{\calH}{\mathcal{H}}
\newcommand{\calHbar}{\overline{\mathcal{H}}}
\newcommand{\calJ}{\mathcal{J}}
\newcommand{\calM}{\mathcal{M}}
\newcommand{\calO}{\mathcal{O}}

\newcommand{\calX}{\mathcal{X}}
\newcommand{\calXbar}{\overline{\mathcal{X}}}

\newcommand{\calDiv}{\mathcal{D}iv}
\newcommand{\calDivbar}{\overline{\mathcal{D}iv}}

\DeclareMathOperator{\Jac}{Jac}

\DeclareMathOperator{\Hom}{Hom}

\DeclareMathOperator{\Aut}{Aut}

\DeclareMathOperator{\val}{val}

\DeclareMathOperator{\Span}{Span}

\DeclareMathOperator{\trop}{trop}
\DeclareMathOperator{\Div}{Div}
\DeclareMathOperator{\Ram}{Ram}

\DeclareMathOperator{\mdeg}{mdeg}
\DeclareMathOperator{\DR}{DR}
\DeclareMathOperator{\PD}{PD}
\DeclareMathOperator{\src}{src}
\DeclareMathOperator{\br}{br}
\DeclareMathOperator{\Rat}{Rat}
\DeclareMathOperator{\ord}{ord}
\DeclareMathOperator{\slope}{slope}

\let\div\relax
\DeclareMathOperator{\div}{div}

\let\sp\relax
\DeclareMathOperator{\sp}{sp}

\usepackage{color}

\newcommand{\ph}{\varphi}
\newcommand{\si}{\sigma}
\newcommand{\ga}{\gamma}
\newcommand{\Ga}{\Gamma}
\newcommand{\De}{\Delta}
\newcommand{\tF}{\widetilde{F}}
\newcommand{\tG}{\widetilde{G}}
\newcommand{\tph}{\widetilde{\ph}}
\newcommand{\ttau}{\widetilde{\tau}}
\newcommand{\tGa}{\widetilde{\Ga}}

\newcommand{\tDe}{\widetilde{\Delta}}
\newcommand{\tT}{\widetilde{T}}
\newcommand{\ta}{\widetilde{a}}
\newcommand{\tDR}{\widetilde{DR}}

\title{Tropical double ramification loci}

\date{}

\author{Martin Ulirsch and Dmitry Zakharov}

\newcommand{\addresses}{{
  \bigskip
  \footnotesize
  
  \noindent Martin~Ulirsch,\newline \textsc{Institut f\"ur Mathematik, Goethe-Universit\"at Frankfurt, 60325 Frankfurt am Main, Germany }\newline\nopagebreak
  \textit{E-mail address:} \texttt{ulirsch@math.uni-frankfurt.de}

  \medskip
  
  \noindent Dmitry~Zakharov, \newline \textsc{Department of Mathematics, Central Michigan University, Mount Pleasant, MI 48859, USA}\newline\nopagebreak
  \textit{E-mail address:} \texttt{zakha1d@cmich.edu}

}}

\begin{document}

\maketitle

\begin{abstract}
Motivated by the realizability problem for principal tropical divisors with a fixed ramification profile, we explore the tropical geometry of the double ramification locus in $\calM_{g,n}$.There are two ways to define a tropical analogue of the double ramification locus: one as a locus of principal divisors, the other as a locus of finite effective ramified covers of a tree. We show that both loci admit a structure of a generalized cone complex in $M_{g,n}^{trop}$, with the latter contained in the former. We prove that the locus of principal divisors has cones of codimension zero in $M_{g,n}^{trop}$, while the locus of ramified covers has the expected codimension $g$. This solves the deformation-theoretic part of the realizability problem for principal divisors, reducing it to the so-called Hurwitz existence problem for covers of a fixed ramification type. 
\end{abstract}


\section{Introduction}

\subsection{Realizability of tropical principal divisors}

Let $\Ga$ be a stable tropical curve and $D=\div(f)=\sum_{i=1}^n a_i x_i$ a principal divisor on $\Ga$, where $a=(a_1,\ldots,a_n)$ are non-zero integers summing to zero. We say that the pair $(\Ga, D)$ is \emph{realizable} if there exists a smooth curve $X$ over a non-Archimedean field $K$ together with a stable degeneration $\calX$, as well as a principal divisor $\widetilde{D}$ on $X$ with the same multiplicity profile $a$ as $D$, such that 
\begin{itemize}
\item $\Ga$ is the dual tropical curve of $\calX$; and
\item the \emph{specialization} of $\widetilde{D}$, i.e. the multidegree of the special fiber of the closure of $\widetilde{D}$ in a suitably chosen semistable model of $\calX$, is equal to $D$. 
\end{itemize}

The {\it realizability problem} consists in describing all realizable pairs $(\Ga,D)$.

We begin by noting that if we relax the condition that $\widetilde{D}$ have the same multiplicity profile $a$ as $D$, then the realizability problem is always solvable. Indeed, let $\Ga$ be a tropical curve of genus $g$ that arises as the dual tropical curve of a fixed semistable degeneration $\calX$, and let $D$ be a principal divisor on $\Ga$. Write $D=D_1-D_2$ for effective divisors $D_1,D_2\in\Div_d^+(\Ga)$ of some degree $d$. In \cite[Theorem 1.1]{BakerRabinoff_skelJac=Jacskel}, Baker and Rabinoff show that there exist effective divisors $\widetilde{D}_1,\widetilde{D}_2\in \Div_{d+g}^+(X)$ of degree $d+g$ such that $\widetilde{D}=\widetilde{D}_1-\widetilde{D}_2$ is principal and the specialization of $\widetilde{D}$ to $\Ga$ is equal to $D$. In other words, a principal divisor $D$ on $\Ga$ can be lifted to a principal divisor $\widetilde{D}$ by adding at most $g$ additional zeroes and poles which cancel under tropicalization. Thus, loosely speaking, realizability imposes $g$ independent conditions on the pair $(\Ga,D)$. 

The realizability problem is a precise incarnation of the heuristic that there are "more rational functions on tropical curves than tropicalizations of rational functions". As such, it explains why tropical linear systems are typically much larger than the tropicalizations of algebraic linear systems. In particular, it functions as a moral reason for Baker's specialization lemma in \cite{Baker_specialization} saying that the rank of a linear system on an algebraic curve can only go up, when passing to its tropicalization. 

Our approach to this problem begins with \cite{Caporaso_gonality} and is motivated by \cite{ABBRI, ABBRII}. In \cite{Caporaso_gonality} Caporaso studies whether a $d$-gonal tropical curve is necessarily the tropicalization of a $d$-gonal algebraic curve. In \cite{ABBRI,ABBRII}, the authors study the more general problem of lifting finite harmonic morphisms of tropical curves to finite morphisms of algebraic curves. Both \cite{Caporaso_gonality} and \cite{ABBRI, ABBRII} show that the obstruction to this lifting problem is not deformation-theoretic, but combinatorial, as it depends on the non-vanishing of certain local Hurwitz numbers. Moreover, the authors of \cite{ABBRII} observe in \cite[Proposition 4.2]{ABBRII} that, given a finite harmonic morphism from a tropical curve to a tree, any two fibers are linearly equivalent divisors. 

The following Theorem \ref{thm_realizabilityofprincipaldivisors} uses these observations to sum up and refine what is known about the realizability problem for tropical principal divisors. 

\begin{maintheorem}\label{thm_realizabilityofprincipaldivisors}
Let $D=\div(f)=\sum_{i=1}^n a_i x_i$ be a principal divisor on a tropical curve $\Ga$. Then $(\Ga, D)$ is realizable if and only there is a tropical modification $\widetilde{\Ga}$ of $\Ga$ as well as a finite effective harmonic morphism $f\colon \widetilde{\Ga}\rightarrow \De$ to a metric tree $\De$ with two legs $0$ and $\infty$, such that the local Hurwitz numbers of $f$ are all non-zero and the difference of the preimage of $0$ and $\infty$ is a divisor that stabilizes to $D$. 
\end{maintheorem}

In this article, we argue that a principal divisor is realizable if and only if it lies in a natural codimension $g$ locus of $M_{g,n}^{trop}$, called the \emph{tropical double ramification locus}, and satisfies combinatorial conditions determined by the non-vanishing of certain Hurwitz numbers. The tropical double ramification locus contains the tropicalization of the algebraic double ramification locus. In general it is strictly greater, however, it has the expected codimension $g$ in $M_{g,n}^{trop}$, hence exactly encodes the $g$ conditions observed in \cite{BakerRabinoff_skelJac=Jacskel}. 

\subsection{A tropical double ramification locus}

Let $g\geq 1$, $n\geq 2$, and let $a=(a_1,\ldots, a_n)\in \Z^n$ be non-zero integers such that $a_1+\cdots+a_n=0$. The \emph{double ramification locus} $\calDR_{g,a}$ in $\calM_{g,n}$ is the codimension $g$ locus of marked smooth curves $(X,p_1,\ldots, p_n)$ that fulfill the following two equivalent conditions:
\begin{enumerate}[(i)]
\item The divisor $\sum_{i=1}^n a_i p_i$ is principal, in other words
\begin{equation*}
\calO_X\Big(\sum_{i=1}^na_ip_i\Big)\simeq \calO_X \ . 
\end{equation*}
\item There is a map $f\colon X\rightarrow \PP^1$ with ramification profiles $\displaystyle\sum_{i:a_i>0}a_ip_i$ and $\displaystyle\sum_{i:a_i<0}(-a_i)p_i$ over $0\in \PP^1$ and $\infty\in\PP^1$, respectively.
\end{enumerate}

In order to define a tropical double ramification locus, we may use a tropical analogue of either of these two conditions; each of them leads to a different locus in  $M_{g,n}^{trop}$.

\begin{definition}\label{def_introPDDR} Denote by $M_{g,n}^{trop}$ the moduli space of stable tropical curves $(\Ga, p_1,\ldots, p_n)$ of genus $g$ with $n$ marked legs, as defined in \cite{ACP} and Section \ref{section_Mgntrop} below.
\begin{enumerate}[(i)]  
\item The \emph{locus of principal divisors} $\PD_{g,a}\subseteq M_{g,n}^{trop}$ is the set of points $(\Ga, p_1, \ldots, p_n)\in M_{g,n}^{trop}$ such that the divisor $\sum_{i=0}^n a_i p_i$ is principal (see Definition \ref{def_PDharmonic} below for details). 

\item The \emph{double ramification locus} $\DR_{g,a}\subseteq M_{g,n}^{trop}$ is the set of marked curves $(\Ga, p_1, \dots, p_n)\in M_{g,n}^{trop}$ for which there exists a tropical modification $\widetilde{\Ga}$ of $\Ga$ as well as a finite effective harmonic morphism $f\colon \widetilde{\Ga}\rightarrow \De$ to a tree $\De$ with two legs $0$ and $\infty$, such that the preimage of $0$ is the divisor $\displaystyle\sum_{i:a_i>0}a_i p_i$ and the preimage of $\infty$ is the divisor $\displaystyle\sum_{i:a_i<0}(-a_i) p_i$. 
\end{enumerate}\end{definition}

The following Theorem \ref{thm_DRlocus} describes the geometric structure of these loci and constitutes the main result of this paper (see Section~\ref{sec:conecomplexes} for definitions):

\begin{maintheorem}\label{thm_DRlocus}
\begin{enumerate}[(i)]
\item The locus $\PD_{g,a}^{trop}$ of tropical principal divisors is a linear subset of $M_{g,n}^{trop}$ that has maximal cones of every codimension between zero and $g$.

\item The tropical double ramification locus $\DR_{g,a}^{trop}$ is a semilinear subset of $\PD_{g,a}^{trop}$ whose maximal-dimensional cones have codimension $g$ in $M_{g,n}^{trop}$. 
\end{enumerate} 
\end{maintheorem}

We prove Part (i) of Theorem \ref{thm_DRlocus} using an argument similar to the one used to prove the structure theorem for the tropical Hodge bundle \cite[Theorem 1.2]{LinUlirsch}. Part (ii) is proved using a refinement of the algorithm developed in \cite{CoolsDraisma} to describe the $d$-gonal locus in $M_g^{trop}$. 

In the hyperelliptic case, when the zero (or pole) degree of $a$ is equal to 2, we have a stronger result about the topological and combinatorial properties of the tropical double ramification cycle. 

\begin{maintheorem} \label{thm_hyperellipticDC}
Let $a=(2,-2)$, $(2,-1,-1)$, or $(1,1,-1,-1)$. The hyperelliptic double ramification locus $\DR_{g,a}^{trop}$ is a linear subset of respectively $M_{g,2}^{trop}$, $M_{g,3}^{trop}$, and $M_{g,4}^{trop}$ that is connected in codimension one. Its projection to $M_g^{trop}$ is the realizable hyperelliptic locus $H_g$, and the fibers of the projection have dimensions $0$, $1$, and $2$, respectively. 
\end{maintheorem}

Denote by $\calM_{g,n}^{an}$ the non-Archimedean analytic stack associated to $\calM_{g,n}$. As introduced in \cite{ACP}, there is a natural continuous, proper and surjective tropicalization map $\trop_{g,n}\colon \calM_{g,n}^{an}\rightarrow M_{g,n}^{trop}$ that associates to a stable one-parameter degeneration of a marked smooth curve its dual tropical curve. We show in Prop.~\ref{prop_factorization} below that the tropicalization map $\trop_{g,n}$ restricts to a tropicalization map $\trop_{g,a}\colon \calDR_{g,a}^{an}\rightarrow \DR_{g,a}^{trop}$, which is not surjective in general. The situation may be summarized in the following diagram:
\begin{center}\begin{tikzcd}
\calDR_{g,a}^{an} \arrow[d,"\trop_{g,a}"'] \arrow[rr,"\subseteq"]& &\calM_{g,n}^{an} \arrow[d, twoheadrightarrow, "\trop_{g,n}"]\\
    \DR_{g,a}^{trop}\arrow[r,"\subseteq"] & \PD_{g,a}^{trop}\arrow[r,"\subseteq"] & M_{g,n}^{trop}
\end{tikzcd}\end{center}

The following Theorem \ref{thm_realizabilityDR} is a refinement of Theorem \ref{thm_realizabilityofprincipaldivisors}, which in our terminology describes the realizability locus $\trop_{g,a}(\calDR_{g,a}^{an})$ in $\DR_{g,a}^{trop}$. Our moduli-theoretic proof is based on the tropicalization of the moduli space of admissible covers developed in \cite{CavalieriMarkwigRanganathan_tropadmissiblecovers}. An alternative proof could have made use of the lifting results in \cite{ABBRI, ABBRII}.

\begin{maintheorem}\label{thm_realizabilityDR}
The realizability locus $\trop_{g,a}(\calDR_{g,a}^{an})$ is the subcomplex of $\DR_{g,a}^{trop}$ consisting of cones of Hurwitz type. 
\end{maintheorem}

We may reinterpret Theorem \ref{thm_realizabilityDR} as saying that the realizability problem for tropical principal divisors consists of two parts, a continuous and a discrete (actually group-theoretic) obstruction: 
\begin{itemize}
\item By the Bieri--Groves theorem for subvarieties of toroidal embeddings (see \cite[Theorem 1.1]{Ulirsch_tropcomplogreg}) the realizability locus $\trop_{g,a}(\calDR_{g,a}^{an})$ has codimension at least $g$. The principal divisor locus $\PD_{g,a}^{trop}$ has codimension zero, so replacing $\PD_{g,a}^{trop}$ to $\DR_{g,n}^{trop}$ precisely cuts away these superfluous dimensions. 
\item An element of $\DR_{g,a}^{trop}$ lies in $\trop_{g,a}(\calDR_{g,a}^{an})$. In other words it represents a realizable divisor, if and only if the local Hurwitz numbers of the corresponding unramified cover are all non-zero.
\end{itemize}

 In general, the problem of deciding whether, given discrete branch data fulfilling the Riemann--Hurwitz formula, there is a branched cover of Riemann surfaces of that type, goes by the name of the \emph{Hurwitz existence problem}. For genus zero targets this problem is by far and large unsolved (see \cite{PervovaPetronio_HurwitzexistenceI} for the state of the art).


\subsection{Related work}

\subsubsection{Hyperelliptic and $d$-gonal tropical curves}

One line of motivation for our work is the definition of gonality for tropical curves. Given a divisor $D$ on a tropical curve $\Ga$, the Baker--Norine rank of $D$ is the largest integer $k$ such that $D-E$ has a nonempty linear system for any effective divisor $E$ of degree $k$. We say that a tropical curve is divisorially $d$-gonal if it carries a $g^1_d$, i.e.~a divisor of degree $d$ and rank at least one.

This definition is problematic from a moduli-theoretic point of view, for the reason explained above: tropical linear systems are typically larger than tropicalizations of algebraic linear systems. As a result, loci in $M_{g,n}^{trop}$ defined using Brill--Noether conditions such as divisorial gonality inevitably have unexpectedly high dimension. For example, the moduli space of tropical hyperelliptic curves of genus $g$ (i.e. curves admitting a $g^1_2$) has dimension $3g-3$ (see~\cite{LimPaynePotashnik}), instead of the expected $2g-1$, which implies that not every hyperelliptic tropical curve is the tropicalization of a hyperelliptic curve. The behavior is similar for higher gonality. A related result of~\cite{LimPaynePotashnik} is that the function $\dim W^r_d(\Ga)$ is not upper semi-continuous on the moduli space $M_g^{trop}$, implying that loci in $M_g^{trop}$ defined by divisorial Brill--Noether conditions are not even cone complexes in general. 

Hyperelliptic tropical curves have been studied extensively, and the realizability problem for them is fully understood. In \cite{Chan_hyperelliptic}, expanding on earlier work of Baker and Norine for finite graphs \cite{BakerNorine_harmonic}, Chan studies the locus of hyperelliptic curves in $M_g^{trop}$. She proves that a tropical curve $\Ga$ is hyperelliptic if and only if it admits a (necessarily unique) harmonic morphism $\Ga\to \De$ of degree 2 to a metric tree (equivalently, a unique involution $\iota:\Ga\to \Ga$ such that the quotient $\Ga/\iota$ is a tree). Theorem 4.13 in~\cite{ABBRII}, which generalizes Theorem 4.8~\cite{Caporaso_gonality} to metric trees, gives a complete answer to the realizability problem for hyperelliptic curves. Specifically, the theorem states that a hyperelliptic tropical curve $\Ga$ is the tropicalization of an algebraic hyperelliptic curve if and only if for every point $p\in \Ga$ fixed by the hyperelliptic involution $\iota$, the number of tangent directions at $p$ fixed by $\iota$ is less than or equal to $2g(p)+2$. In the language of this paper, this condition is equivalent to requiring that the hyperelliptic morphism $\Ga\to \De$ be {\it effective}.

Hyperelliptic tropical curves can be viewed as ramified Galois covers of metric trees with Galois group $\ZZ/2\ZZ$, and a number of papers have generalized this construction. Len~\cite{Len_hyperelliptic} studies hyperelliptic metrized complexes, while Bolognese, Brandt, and Chua~\cite{BBC} solve the problem of finding the abstract tropicalization of an algebraic hyperelliptic curve. Jensen and Len~\cite{JensenLen} consider unramified $\ZZ/2\ZZ$-covers of arbitrary tropical curves, while Brandt and Helminck~\cite{BrandtHelminck_superelliptic} study Galois covers of metric trees with arbitrary cyclic Galois group and study their locus in $M_g^{trop}$. In~\cite{LUZI}, Len and the authors generalize all these constructions and develop a theory of Galois covers of tropical curves with arbitrary abelian Galois group. The realizability problem is also central to tropical Hurwitz theory. We refer the reader to \cite{CavalieriJohnsonMarkwig, BertrandBrugalleMikhalkin, CavalieriMarkwigRanganathan_tropadmissiblecovers, BBBM, GoujardMoeller, MandelRuddat} and the references therein for more details on this extensive topic. More details on the realizability problem in characteristics other than zero can also be found in \cite{BreznerTemkin, CohenTemkinTrushin}.

From a moduli-theoretic perspective, two papers that are especially important for us are~\cite{CavalieriMarkwigRanganathan_tropadmissiblecovers} and ~\cite{CoolsDraisma}. In \cite{CavalieriMarkwigRanganathan_tropadmissiblecovers}, Cavalieri, Markwig, and Ranganathan provide us a new moduli-theoretic perspective on the tropicalization of ramified covers, expanding on the classical work of Abramovich, Caporaso, and Payne \cite{ACP} for the moduli space of algebraic curves. In \cite{CoolsDraisma}, Cools and Draisma study the $d$-gonal locus in $M_g^{trop}$ from a purely combinatorial perspective. Reformulated in our language, their main result is the following: the locus of curves in $M_g^{trop}$ admitting a finite effective degree $d$ harmonic morphism $\Ga\to \De$ to a metric tree has the expected dimension $\min(2g+2g-5,3g-3)$. It follows from the results of \cite{ABBRII} that any such curve admits a $g^1_d$, hence is divisorially $d$-gonal. In other words, requiring that the $g^1_d$ is represented by a finite effective map to a tree produces a locus of the expected dimension, as in the case of hyperelliptic curves. We note here that the realizability problem for $d$-gonal curves, unlike hyperelliptic curves, is highly non-trivial (and not addressed in~\cite{CoolsDraisma}), since Hurwitz numbers of degree $d\geq 4$ can generally vanish (see \cite[Section 2.2]{Caporaso_gonality}). Our results on the tropical double ramification locus can be seen as a generalization of the results of~\cite{CoolsDraisma}. In \cite{DraismaVargasI} Draisma and Vargas continue their study of tropical gonality by showing that every metric graph has gonality at most $\lceil \frac{g}{2}\rceil+1$.

\subsubsection{Other instances of the realizability problem}

The realizability problem for tropical divisors in its most general form (known to us) can be stated as follows. Let $a=(a_1,\ldots, a_n)\in\Z^n$ be integers such that $a_1+\cdots+a_n=d$ for an integer $d\geq 0$, let $r\geq -1$, and let $K$ be a non-Archimedean field. Given a tropical curve $\Gamma$ and a divisor $D$ on $\Gamma$ of multiplicity profile $a$ (and thus of degree $d$) and of rank $r$, does there exist a smooth projective curve over $K$ and a divisor $\widetilde{D}$ with multiplicity profile $a$ (and thus of degree $d$) and of rank $r$ such that $\Gamma$ is the dual tropical curve of $X$ and $D$ is the specialization of $\widetilde{D}$? We remark that, according Baker's specialization lemma \cite{Baker_specialization}, the rank of a divisor can only go up when specializing to a tropical curve.

The case $r=1$ and $d=2$ is the realizability problem for hyperelliptic curves, which has been completely solved in \cite{ABBRI, ABBRII}. The case $d>2$ (while still $r=1$) has been treated in \cite{LuoManjunath_smoothingg1d} in the framework of limit linear series on metrized curve complexes. A special case of this realizability problem also appears in \cite{BrandtHelminck_superelliptic}, where Brandt and Helminck consider superelliptic curves, i.e.~cyclic Galois covers of $\PP^1$.

In the case $r\geq 2$, Cartwright \cite{Cartwright_Murphyslaw} shows a version of Murphy's law for the realizability problem: given a matroid, there is a tropical curve $\Ga$ and a divisor $D$ of rank $r\geq 2$ on it, such that deciding realizability for $(\Ga,D)$ is equivalent to deciding whether the matroid is realizable over a field. This suggests that the realizability problem is far beyond reach in its most general form. 

Nevertheless, there are many special cases when solutions can be given. For one, the realizability problem for divisor classes admits a full solution on tropical curves of a certain generic shape, known as a \emph{chain of loops} (see \cite{CartwrightJensenPayne}). In \cite{JensenRanganathan}, the authors provide us with a different perspective on the realizability problem on chains of loops using techniques from logarithmic Gromov--Witten theory. This solution allows the authors to prove a new Brill-Noether theorem for curves of a fixed gonality that significantly generalizes the classical Brill-Noether theorem to special curves.

There are also other situations where the realizability problem is solvable uniformly on all tropical curves. In \cite{MoellerUlirschWerner_realizability} Moeller, the first author, and Werner solve the realizability problem for effective canonical divisors, i.e. the case $r=g-1$ and $d=g-1$ (with all $a_i>0$). In this case, Cartwright's restriction does not apply, since effective canonical divisors correspond to sections of twists of the relative dualizing sheaf.

In this article we are treating the case $r=0$ and $d=0$, i.e. the case of principal divisors. As explained above, this has also been the subject of \cite{BakerRabinoff_skelJac=Jacskel}, albeit without reference to the specific multiplicity profiles.

\subsubsection{The double ramification cycle, the tautological ring, and the Abel--Jacobi section}

Our work is also motivated by the significant attention that the algebraic double ramification locus $\calDR_{g,a}\subset \calM_{g,n}$ has attracted in the last decade. A natural direction of inquiry has been to describe a meaningful compactification of $\calDR_{g,a}$ in the Deligne-Knudsen-Mumford compactification $\calMbar_{g,n}$, and to study its (virtual) class in the tautological ring $R^*(\calMbar_{g,n})$. This is also known as \emph{Eliashberg's problem}, who in the early 2000's posed this question in the context of symplectic field theory. In this subsection, we review a number of recent results concerning the double ramification cycle; see~\cite{Pandharipande} for an excellent survey.

There are two strategies, quite different in nature, to construct natural compactifications of the double ramification locus $\calDR_{g,a}$.
\begin{enumerate}[(i)]
\item The double ramification locus $\calDR_{g,a}$ is the pullback of the zero section of the universal Jacobian $\calJ_{g,n}$ along the natural Abel-Jacobi section 
\begin{equation*}\begin{split}
\calM_{g,n}&\longrightarrow \calJ_{g,n},\\ 
(X,p_1,\ldots, p_n)&\longmapsto \big(X,p_1,\ldots, p_n,\calO_X(a_1p_1+\cdots+a_np_n)\big),
\end{split}\end{equation*} 
so we can compactify $\calDR_{g,a}$ by finding a suitable compactification of $\calJ_{g,n}$ and pulling back the zero section along the resolution of the Abel--Jacobi map.
\item The locus $\calDR_{g,a}$ is the set of marked curves admitting a map $f\colon X\rightarrow \PP^1$ with ramification profiles at $0$ and $\infty$ prescribed by $a$, so we define the compactification of $\calDR_{g,a}$ as the image in $\calMbar_{g,n}$ of a suitably compactified moduli space of maps $f\colon X\rightarrow \PP^1$.

\end{enumerate}

The first results on Eliashberg's problem used the universal Jacobian approach. The moduli space $\calM_{g,n}$ admits a partial compactifaction $\calM_{g,n}^{ct}$ parametrizing curves of compact type, whose Jacobians are abelian varieties. The Abel--Jacobi map naturally extends to $\calM_{g,n}^{ct}\longrightarrow\calJ_{g,n}$, and we can define the class $\big[\calDR_{g,a}^{ct}\big]$ in the tautological ring $R^g\big(\calM_{g,n}^{ct}\big)$ as the pullback of the zero section of $\Jac_{g,n}$. Using this approach, Hain~\cite{Hain_normalfunctions} computed the class $\big[\calDR_{g,a}^{ct}\big]$ in homology using Hodge theory, and Grushevsky and Zakharov~\cite{GrushevskyZakharov_thetadivisor} computed the same class in the Chow ring using test curves. In~\cite{GrushevskyZakharov_zerosection}, the authors extended this result to compute the double ramification cycle on a slightly larger space $\calM_{g,n}^1\supset\calM_{g,n}^{ct}$, parametrizing curves having at most one non-separating node. 

There are two natural ways to compactify the moduli space of maps $X\to \PP^1$ with ramification specified by $a\in \ZZ^n$ over two points $0,\infty\in \PP^1$. In both cases we allow both the source curve $X$ and the target $\PP^1$ to degenerate. The first is the space of admissible covers $\calH_{g,a}$, defined by Harris and Mumford in~\cite{HarrisMumford}. We define the {\it admissible} double ramification locus $\calDR^{adm}_{g,a}$ as the image of $\calH_{g,a}$ in $\calMbar_{g,n}$, it coincides with the closure of $\calDR_{g,a}$, and its class is in the tautological ring $R^g(\calMbar_{g,n})$ by~\cite{FP05}. The second, motivated by Gromov--Witten theory, is known as the moduli space of rubber relative stable maps to $\PP^1$. This space is not equidimensional, but admits a virtual fundamental class, whose pushforward $R_{g,a}\in R^g\big(\calMbar_{g,n}\big)$ is called the {\it relative} double ramification cycle (see~\cite{GraberVakil} and~\cite{FP05}).

It may appear at first that $\calDR^{adm}_{g,a}$ is the natural object to study. However, almost all results concerning the double ramification cycle have used the class $R_{g,a}$, while $\calDR^{adm}_{g,a}$ remains somewhat mysterious. Cavalieri, Marcus, and Wise showed in~\cite{2012CavalieriMarcusWise} and~\cite{2013MarcusWise} that the restriction of $R_{g,a}$ to the tautological ring $R^g\big(\calM_{g,n}^{ct}\big)$ coincides with $\big[\calDR_{g,a}^{ct}\big]$. In 2014, based on the known formula for $\big[\calDR_{g,a}^{ct}\big]$ and the results of~\cite{GrushevskyZakharov_zerosection}, Pixton conjectured a formula for $R_{g,a}$ on all of $\calMbar_{g,n}$, which was proved in~\cite{JPPZ}. Together with a formula for $R_{g,a}\in R^g\big(\calMbar_{g,n}\big)$, Pixton conjectured a set of relations in $R^k\big(\calMbar_{g,n}\big)$ for $k\geq g+1$. These relations were proved in~\cite{2016CladerJanda}. In~\cite{CGJZ}, it was shown that these relations imply the known vanishing results for the tautological ring $R^*\big(\calM_{g,n}\big)$, and can be used to determine explicit boundary formulas on $\calMbar_{g,n}$ for tautological classes that vanish on $\calM_{g,n}$.

As we noted above, in order to define a full compactification of $\calDR_{g,a}$ using the universal Jacobian $\calJ_{g,n}$, it is necessary to find a suitable compactification of $\calJ_{g,n}$ and extend the Abel--Jacobi section. In~\cite{Holmes_extendingDRcycle}, Holmes considered the compactification of $\calJ_{g,n}$ to the multidegree zero universal Jacobian $\calJ^{\underline{0}}_{g,n}$ over $\calMbar_{g,n}$. In this case, the Abel--Jacobi section only extends to a rational map $\calMbar_{g,n}\dashrightarrow\calJ^{\underline{0}}_{g,n}$. The scheme-theoretic pullback of the zero section of $\calJ^{\underline{0}}$ to the resolution is proper over $\calMbar_{g,n}$, and its pushforward to $\calMbar_{g,n}$ turns out to coincide with $R_{g,a}$.

Alternatively, Kass and Pagani~\cite{KassPagani} resolved the Abel--Jacobi map by replacing $\calJ_{g,n}$ with a space $\overline{\calJ}_{g,n}(\phi)$ depending on a stability parameter $\phi$, and showed in~\cite{HolmesKassPagani} that the pullback of the zero section along the Abel--Jacobi map is again equal to $R_{g,a}$. Finally, a treatment of this problem using logarithmic geometry was made in~\cite{MarcusWise_logAbelJacobi}, and yet again produces the relative double ramification cycle $R_{g,a}$.

In \cite{AbreuPacini_universalAbelJacobi} Abreu and Pacini show that the tropical Abel-Jacobi section, as a map from $M_{g,n}^{trop}$ to the universal tropical Jacobian they construct in \cite{AbreuPacini_universalJacobian}, tells us exactly which toroidal resolution to apply for the algebraic Abel-Jacobi section to compactify. In \cite[Section 7.1]{AbreuPacini_universalAbelJacobi}, they in particular also consider a tropical analogue of the double ramification locus. In our language, this is exactly the principal divisor locus $\PD_{g,a}^{trop}$.

\subsubsection{Multiplicativity of the double ramification cycle}

Over the locus $\calM_{g,n}^{ct}$ of stable curves of compact type, the double ramification cycle is known to fulfill the multiplicity relation
\begin{equation}\label{eq_multiplicativity}
\big[\calDR^{ct}_{g,a}\big]\cdot \big[\calDR^{ct}_{g,b}\big] = \big[\calDR^{ct}_{g,a}\big]\cdot \big[\calDR^{ct}_{g,a+b}\big] 
\end{equation}
for two vectors $a,b\in\Z^n$ of integers summing to zero. Its extension to all of $\calMbar_{g,n}$, however, does not fulfill the multiplicity relation \eqref{eq_multiplicativity}. 

In \cite{HolmesPixtonSchmitt} the authors find that an analogue of the multiplicity relation \eqref{eq_multiplicativity} remains valid, when considering an extension of the double ramification cycle in a suitable blow-up of $\calMbar_{g,n}$. They, in particular, notice that for \eqref{eq_multiplicativity} to be valid it is enough to consider a toroidal blow-up along the Deligne-Mumford boundary. It would be highly interesting to investigate the relationship between these toroidal blow-ups and the (non-proper) toroidal modification induced by a subdivision of $M_{g,n}^{trop}$ supported on $\DR_{g,a}^{trop}$ constructed in this article. This would very much be in the spirit of Tevelev's theory of tropical compactifications (see \cite{Tevelev_tropcomp}), extended to subvarieties of $\calM_{g,n}$. 

We also refer the reader to \cite{Herr_logproductformula} and \cite{Ranganathan_logproductformula} for an analogous development in the context of a product formula in logarithmic Gromov-Witten theory. 

\subsection{Acknowledgments}
During the preparation of this project we profited from conversations with many of our colleagues. We, in particular, thank Alex Abreu, Renzo Cavalieri, Felix Janda, Yoav Len, Martin M\"oller, Marco Pacini, Dhruv Ranganathan,  Matthew Satriano, and Jonathan Wise.

This project  has  received  funding  from  the  European Union's Horizon 2020 research and innovation programme  under the Marie-Sk\l odowska-Curie Grant Agreement No. 793039. \includegraphics[height=1.7ex]{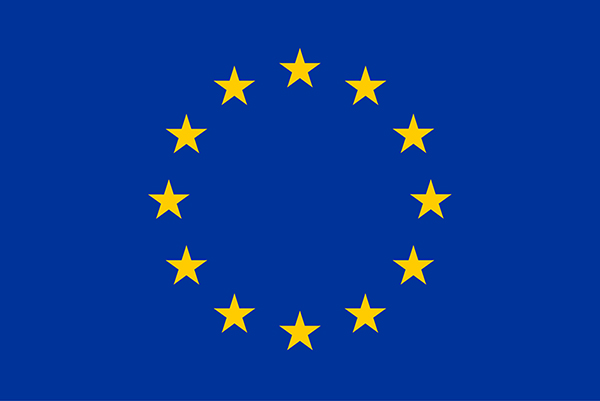}
We also acknowledge support from the LOEWE-Schwerpunkt ``Uniformisierte Strukturen in Arithmetik und Geometrie''.

\setcounter{tocdepth}{2}
\tableofcontents

\section{Preliminaries}

We begin with a plethora of definitions. Most of these are standard and are listed without numbering. However, several definitions that are either of particular importance or that are, to the best of the authors' knowledge, new or unfamiliar, are typeset separately.

\subsection{Graphs}

A {\it graph with legs} $G$, or simply a {\it graph}, consists of the following data:
\begin{itemize}

\item A finite set $X(G)$.

\item An idempotent {\it root map} $r_{G}:X(G)\to X(G)$.

\item An involution $\iota_{G}:X(G)\to X(G)$ whose fixed set contains the image of $r_{G}$. 

\end{itemize}
The image of $r_{G}$, on which it acts trivially, is the set of {\it vertices} of $G$ and is denoted $V(G)$, and the complement $H(G)=X(G)\backslash V(G)$ is the set of {\it half-edges} of $G$. The involution $\iota_{G}$ preserves $H(G)$ and partitions it into orbits of size 1 and 2, called respectively the {\it edges} and {\it legs} of $G$. We denote the sets of edges and legs by $E(G)$ and $L(G)$, respectively. The root map $r_{G}$ assigns {\it root vertices}, two to each edge and one to each leg (every vertex has itself as its endpoint). An edge whose endpoints are equal is called a {\it loop}.

The {\it geometric realization} $|G|$ of a graph $G$ is the one-dimensional CW complex (zero-dimensional if $H(G)$ is empty) defined as follows. As a set, $|G|$ is obtained by taking the union of a copy $I_e$ of the closed interval $[0,1]$ for each edge $e\in H(G)$, a copy $I_p$ of the half-axis $[0,\infty)$ for each leg $p\in L(G)$, and a point for each vertex $v\in V(G)$, which we also denote $v$. For each edge $e=\{h_1,h_2\}\in H(G)$ we identify the $0$ and $1$ of $I_e$ with the vertices $r_{G}(h_1)$ and $r_{G}(h_2)$ in any order. For each $p\in L(G)$, we identify the $0$ in $I_p$ with $h_{G}(p)$. In this way, $V(G)$ is naturally a subset of $|G|$. We say that a graph $G$ is {\it connected} if $|G|$ is connected. Unless otherwise specified, we only consider connected graphs. 

\begin{remark} It may appear that there is no need to strictly distinguish legs, which we can think of as edges of infinite length, from finite extremal edges. However, from a tropical and moduli-theoretic viewpoint, these objects play entirely different roles: the former correspond to marked points (hence our use of $p$ to denote legs), and the latter to unstable rational tails. For this reason, there is no point at the end of a leg, and we do not employ the compactified tropical moduli spaces of curves introduced in \cite{ACP}.

\end{remark}

Given a vertex $v\in V(G)$, the set of {\it tangent directions} $T_vG$ at $v$ enumerates the edges and legs attached to $v$:
$$
T_vG=\{h\in H(G)|r_{G}(h)=v\},
$$
with each loop at $v$ counted twice. The {\it valency} $\val(v)$ of a vertex $v$ is equal to $\#T_vG$. 

A {\it vertex weighting} of a graph $G$ is a map $g:V(G)\to \ZZ_{\geq 0}$, where $g(v)$ is called the {\it genus} of $v$ and represents $g(v)$ infinitesimal loops attached at $v$. We refer to the pair $(G,g)$ as a {\it weighted graph} and usually omit $g$. The {\it genus} $g(G)$ of a connected weighted graph is given by
\begin{equation*}
g(G)=b_1(|G|)+\sum_{v\in V(G)}g(v)
\end{equation*}
where 
\begin{equation*}
b_1(\vert G\vert)=\#E(G)-\#V(G)+1
\end{equation*}
is the first Betti number of the graph $G$. A weighted graph of genus zero is called a {\it tree} (in other words, a tree cannot have nontrivial vertex weights).

\begin{definition}\label{def_Eulercharacteristics} Let $G$ be a weighted graph. We define the {\it Euler characteristic} $\chi(v)$ of a vertex $v\in V(G)$ by
$$
\chi(v)=2-2g(v)-\val(v).
$$
Similarly, we define the {\it Euler characteristic} $\chi(G)$ of a connected weighted graph $G$ by
$$
\chi(G)=2-2g(G)-\#L(G).
$$
\end{definition}
It is easy to check that
$$
\sum_{v\in V(G)} \chi(v)=\chi(G).
$$

\begin{remark} The Euler characteristic $\chi(G)$ is not to be confused with the topological Euler characteristic of the CW-complex $|G|$. 
\end{remark}

We say that a vertex $v\in V(G)$ is {\it stable} if $\chi(v)<0$, {\it semistable} if $\chi(v)\leq 0$, and {\it unstable} if $\chi(v)>0$. An unstable vertex has genus 0 and is either isolated or has valency one (note that there is no unstable vertex at the end of a leg). A semistable vertex that is not stable is either an isolated vertex of genus 1 or a vertex of genus 0 and valency 2, in which case we call it {\it simple}. We say that a weighted graph $G$ is {\it semistable} if all of its vertices are semistable and {\it stable} if all of its vertices are stable. We observe that if $G$ is semistable then $\chi(G)\leq 0$ and if $G$ is stable then $\chi(G)<0$. 

Let $G$ be a connected weighted graph with $\chi(G)<0$. We form the {\it stabilization} $G_{st}$ of $G$ as follows (see \cite{ACP}). First, if $G$ has an unstable vertex $u$ with $\chi(u)=1$, remove $u$ and the edge rooted at it. Proceeding in this way, we remove all unstable trees of edges and obtain the {\it semistabilization} $G_{sst}$ of $G$. Now let $v\in V(G_{sst})$ be a simple vertex $v$. If $v$ is the root vertex of two edges $e_1$ and $e_2$, replace them with a single edge $e$ rooted at the remaining root vertices of $e_1$ and $e_2$. If $v$ is the root vertex of an edge $e$ and a leg $p$, remove $e$ and root $p$ at the remaining root vertex of $e$. Proceeding in this way, we remove all semistable vertices and obtain a stable graph $G_{st}$ with $\chi(G_{st})=\chi(G_{sst})=\chi(G)$. 

Let $G'$ and $G$ be graphs, and let $\ph:X(G')\to X(G)$ be a map of sets commuting with the involution and root maps. Parsing this, we see the following:

\begin{itemize} 

\item A vertex $v'\in V(G')$ maps to a vertex $\ph(v')\in V(G)$.

\item A leg $p'\in L(G')$ with endpoint $v'$ either maps to a leg $\ph(p')\in L(G)$ with endpoint $\ph(v')$, or to a vertex $v=\ph(p')\in V(G)$, in which case $\ph(v')=v$ and we say that $\ph$ {\it contracts} $p'$.

\item For an edge $e'=\{h'_1,h'_2\}\in E(G')$ with endpoints $v'_1$ and $v'_2$, there are three possibilities. First, $\ph(e')=\{\ph(h'_1),\ph(h'_2)\}\in E(G)$ may be an edge with endpoints $\ph(v'_1)$ and $\ph(v'_2)$. Second, $v=\ph(h'_1)=\ph(h'_2)\in V(G)$ may be a vertex, in which case $\ph(v'_1)=\ph(v'_2)=v$ and we say that $\ph$ {\it contracts} $e'$. Finally, it is possible that $\ph(h'_1)=\ph(h'_2)\in L(G)$ is a leg with endpoint $\ph(v'_1)=\ph(v'_2)$. 

\end{itemize}
We say that $\ph:G'\to G$ is a {\it morphism of graphs} if it maps edges to either edges or vertices. We say that $\ph$ is {\it finite} if it does not contract any edges or legs. 

A {\it subgraph} $F$ of a graph $G$ is a subset of $X(G)$ that is preserved by the maps $r_{G}$ and $\iota_{G}$. If $G$ is weighted, we weight $F$ by restriction. Given a subset $S\subset E(G)\cup L(G)$ of edges and legs of $G$, we denote $[S]$ the {\it subgraph generated by} $S$, which consists of $S$ together with all root vertices of elements of $S$.

\begin{definition}\label{def_neighborhoodofsubgraph} Let $F$ be a subgraph of a graph $G$. We define a graph $O(F)$, called the {\it neighborhood} of $F$ in $G$, by attaching to $F$, as a leg, each half-edge of $G$ that is rooted at a vertex of $F$ but is itself not in $F$ (so an edge between vertices of $F$ that is not itself in $F$ produces two legs in $O(F)$). Specifically, as a set $X(O(F))=r_{G}^{-1}(X(F))$, and we define $r_{O(F)}$ as the restriction of $r_{G}$, while $\iota_{O(F)}$ is defined to be the restriction of $\iota_{G}$ on $X(F)\subset r_{G}^{-1}(X(F))$ and the identity map on the complement $r_{G}^{-1}(X(F))\backslash X(F)$. We define the {\it neighborhood} $O(S)$ of a subset $S\subset E(G)\cup L(G)$ to be the neighborhood of $[S]$ in $G$.

We observe that for any $v\in V(F)$ there is a natural identification $T_vO(F)=T_vG$, while in general it is only true that $T_vF\subset T_vG$. In particular, we can identify $\{v\}\cup T_vG$ with the graph $O(\{v\})$, and, if $G$ is weighted, then for a vertex $v\in F$ its Euler characteristics as a vertex in $G$ and as a vertex in $O(F)$ are equal, and are both equal to the Euler characteristic of the graph $O(\{v\})$:
$$
\chi_{G}(v)=\chi_{O(F)}(v)=\chi\big(O(\{v\})\big).
$$

\label{def:neighborhood}
\end{definition}


\subsection{Metric graphs and tropical curves}

Let $G$ be a graph. A {\it metric} on $G$ is a function $\ell:E(G)\to \RR_{>0}$. A {\it metric graph} $(G,\ell)$ is a graph $G$ with a choice of metric $\ell$. We call the underlying graph $G$ the {\it combinatorial type} of the metric graph. A weighted metric graph $(G,g,\ell)$ is a metric graph $(G,\ell)$ together with a weighting $g$ on $G$.

Let $(G,\ell)$ be a metric graph, and let $|G|$ be the geometric realization of its combinatorial type. We equip $|G|$ with the structure of a metric space, called the {\it metric realization} of $(G,\ell)$, as follows. For each edge $e\in E(G)$, we identify the corresponding interval $I_e\subset |G|$ with a finite interval of length $l(e)$. For each leg $p\in |G|$, we identify the interval $I_p\subset |G|$ with a copy of $[0,\infty)$. We then give $|G|$ the path metric. 

Let $(G',\ell')$ and $(G,\ell)$ be metric graphs. A {\it morphism} $\ph:(G',\ell')\to (G,\ell)$ of metric graphs is a pair consisting of a morphism of graphs $\ph:G'\to G$ and a weighting $d_{\ph}:E(G')\cup L(G')\to \ZZ_{\geq 0}$, called the {\it degree} of $\ph$, such that the following properties are satisfied. For an edge $e'\in E(G')$ or a leg $p'\in L(G')$, $\ph$ contracts $e'$ or $p'$ if and only if $d_{\ph}(e')=0$ or $d_{\ph}(p')=0$, respectively. If $d_{\ph}(e')>0$, then we require that
\begin{equation}
\ell\left(\ph(e')\right)=d_{\ph}(e')\ell'(e').
\label{eq:length}
\end{equation}
A morphism $\ph:(G',\ell')\to (G,\ell)$ of metric graphs induces a continuous map $|\ph|:|G'|\to |G|$ of their metric realizations, where, for a pair of edges $e=\ph(e')$, the map is given by dilation by a factor of $d_{\ph}(e')$, and similarly for a pair of legs $p=\ph(p')$. This map is piecewise-linear with integer slope with respect to the metric structure.

A basic inconvenience when dealing with metric graphs is that different graphs may have the same metric realizations. This motivates the following definition.

\begin{definition} A {\it tropical curve} is a pair consisting of a connected metric space $\Ga$ and a weight function $g_{\Ga}:\Ga\to \ZZ_{\geq 0}$ such that there exists a weighted metric graph $(G,g,\ell)$ and an isometry $m:|G|\to \Ga$ of its metric realization with $\Ga$, with respect to which the weight functions agree:
$$
g_{\Ga}(x)=\begin{cases}g(v)&\textrm{ if } x=m(v),\,v\in V(G), \\ 0 & \textrm{ otherwise. }\end{cases}
$$
We call a quadruple $(G,g,\ell,m)$ satisfying these properties a {\it model} for $\Ga$. 
\end{definition}

The {\it genus} of a tropical curve $\Ga$ is given by
$$
g(\Ga)=b_1(\Ga)+\sum_{x\in \Ga} g_{\Ga}(x) 
$$
and is equal to the genus of any model of $\Ga$. A {\it tree} is a tropical curve of genus zero. 

For a point $x\in \Ga$ on a tropical curve $\Ga$ with model $(G,g,\ell,m)$, we define its {\it valence} $\val(x)$ to be $\val(v)$ if $x=m(v)$ for some $v\in V(G)$ and $2$ otherwise. We similarly define the Euler characteristic as $\chi(x)=2-2g_{\Ga}(x)-\val(x)$, these numbers do not depend on the choice of model. We define the Euler characteristic of a tropical curve as $\chi(\Ga)=\chi(G)$ for any model $G$ of $\Ga$. We have
$$
\chi(\Ga)=\sum_{x\in \Ga}\chi(x),
$$
where on the right hand side $\chi(x)=0$ for all but finitely many $x\in \Ga$.

\begin{remark} Our definition of tropical curve differs from Def.~2.14 in \cite{ABBRII}, where a tropical curve is defined as an equivalence class of metric graphs up to tropical modifications (see below).
\end{remark}

Let $\Ga$ be a tropical curve with model $G$. We can form another model $G'$ by replacing an edge $e\in E(G)$ by two new edges joined at a new simple vertex, and arbitrarily splitting the length of $e$. Similarly, we can split a leg $p\in L(G)$ into a leg and an edge of arbitrary length. Conversely, any tropical curve $\Ga$ has a unique {\it minimal model} $G_{min}$ having no simple vertices (unless $\Ga$ is the real line or a circle).  We say that $\Ga$ is {\it stable} if its minimal model is a stable graph.

We observe that every connected tropical curve $\Ga$ except for the real line $(-\infty,\infty)$ has a well-defined set of legs, corresponding to the legs of its minimal model. We define the legs of $(-\infty,\infty)$ to be $(-\infty,0]$ and $[0,\infty)$. We denote by $\Ga^{\circ}$ the tropical curve obtained by removing the legs of $\Ga$ (but retaining the attachment points), and by $c:\Ga\to \Ga^{\circ}$ the natural retraction map.

\begin{definition} A {\it morphism} $\tau:\Ga'\to \Ga$ of tropical curves is a continuous, piecewise-linear map with integer slopes that is eventually linear on every leg. Specifically, let $p'\subset \Ga'$ be a leg, identified with $[0,\infty)$. Then there exists a constant $c>0$ such that either $\tau(x)$ is constant on $p'$ for all $x>c$, or there is a leg $p\subset \Ga$ and numbers $a\in \ZZ_{>0}$ and $b\in \RR$ such that, identifying $p$ with $[0,\infty)$, we have $\tau(x)=ax+b\in p$ for all $x>c$. In the latter case, we say that {\it $\tau$ maps $p'$ to $p$ with degree $a$} and write $\tau(p')=p$ and $d_{\tau}(p')=a$ (note that a finite section of $p'$ may map to $\Ga\backslash p$).

Let $\tau:\Ga'\to \Ga$ be a morphism of tropical curves. A {\it model} for $\tau$ is a pair of models $(G',g',\ell',m')$ and $(G,g,\ell,m)$ for $\Ga'$ and $\Ga$, respectively, and a morphism $\ph:G'\to G$ of metric graphs such that $m\circ|\ph|=\tau\circ m'$. 
\end{definition}

Given a morphism $\tau:\Ga'\to \Ga$ of tropical curves, we construct a model for it as follows. Let $(G',g',\ell',m')$ and $(G,g,\ell,m)$ be models for $\Ga'$ and $\Ga$, respectively. Subdivide $G'$ so that $m'(V(G'))$ contains every point of $\Ga'$ at which $\tau$ changes slope; by definition there are finitely many such points. Then, further subdivide $G'$ and $G$ so that $\tau(m'(V(G')))=m(V(G))$ and $\tau^{-1}(m(V(G)))=m'(V(G'))$, in other words so that vertices map to vertices. We then have a well-defined morphism of graphs $\ph:G'\to G$. Furthermore, on the geometric realizations of the resulting models the map $\tau$ is linear with integer slope on each edge and each leg, and we set the degree of $\ph$ to be equal to this slope.

Given a model $\ph:G'\to G$ of a morphism $\tau:\Ga'\to \Ga$ of tropical curves, we can produce another model by adding a simple vertex to $G$ to split an edge or a leg, and correspondingly splitting all the preimages. Conversely, any morphism $\tau:\Ga'\to \Ga$ to a tropical curve $\Ga$ with $\chi(\Ga)<0$ has a unique {\it minimal model} $\ph_{min}:G'_{min,\tau}\to G_{min,\tau}$ with the property that every simple vertex $v\in V(G_{min,\tau})$ has at least one preimage that is not simple.

\begin{definition} Let $\Ga$ be a tropical curve. An {\it elementary tropical modification} $\Ga$ is a tropical curve obtained by attaching an edge of arbitrary length to a point of $\Ga$, and setting $g(x)=0$ for all points of the new edge. A {\it tropical modification} $\Ga'$ of $\Ga$ is a tropical curve obtained by a finite sequence of elementary tropical modifications of $\Ga$, in other words by attaching finitely many trees of finite size. 

Given a tropical curve $\Ga$ with $\chi(\Ga)<0$, there is a unique stable tropical curve $\Ga_{st}$, called the {\it stabilization} of $\Ga$, such that $\Ga$ is a tropical modification of $\Ga_{st}$. In terms of the underlying graphs, a model for $\Ga_{st}$ can be obtained by taking the stabilization of any model of $\Ga$. We note that $\Ga_{st}$ is naturally a subset of $\Ga$ (the inclusion map is a morphism of tropical curves, but not a harmonic morphism), and the legs of $\Ga$ are in natural bijection with the legs of $\Ga_{st}$.

\end{definition}

\begin{remark} Our definition of tropical modification agrees in spirit, though not in fact, with Def.~2.12 in \cite{ABBRII}, where an elementary tropical modification consists in adding an infinite edge, i.e. a leg, to a metric graph. We distinguish between legs and extremal edges because we interpret the former as marked points on the curve. This distinction will prove to be important when we talk about ramification. Specifically, the absence of a vertex at the end of a leg means that tropical curve with legs has more unramified covers (see below) than a combinatorially identical curve with extremal edges instead of legs. 
\end{remark}


\subsection{Harmonic morphisms and ramification}

\begin{definition} Let $G'$ and $G$ be graphs. A {\it harmonic morphism} $\ph:G'\to G$ is a morphism of graphs together with a weighting $d_{\ph}:X(G')\to \ZZ_{\geq 0}$, called the {\it degree} of $\ph$, such that the following properties are satisfied:
\begin{itemize}

\item If $e'=\{h'_1,h'_2\}$ is an edge then $d_{\ph}(h'_1)=d_{\ph}(h'_2)$. We call this number the {\it degree} of $\ph$ along $e'$ and denote it $d_{\ph}(e')$.

\item If $h'\in H(G')$, then $d_{\ph}(h')=0$ if and only if $\ph(h')$ is a vertex.

\item For every vertex $v'\in V(G')$ and for every tangent direction $h\in T_{\ph(v)}G$ we have
\begin{equation}
d_{\ph}(v')=\sum_{\substack{h'\in T_{v'}G', \\ \ph(h')=h}}d_{\ph}(h').
\label{eq:harmonic}
\end{equation}
In particular, this sum does not depend on the choice of $h$.

\end{itemize}

A harmonic morphism of weighted graphs is a harmonic morphism of the underlying graphs, and a harmonic morphism of metric graphs is a harmonic morphism of the underlying graphs satisfying condition \eqref{eq:length} at every edge of the source. Given a morphism $\tau:\Ga'\to \Ga$ of tropical curves, the slope of $\tau$ locally defines the degree, and we say that $\tau$ is harmonic if it has a harmonic model.

\end{definition}

\begin{remark} Unless $\ph(v')$ is an isolated vertex, the degree of a harmonic morphism $\ph:G'\to G$ at a vertex $v'\in V(G')$ can be recovered from the degrees along the edges and legs.
\end{remark}

Given a harmonic morphism $\ph:G'\to G$ of graphs and a vertex $v\in V(G)$, the sum
$$
d_{\ph}(v)=\sum_{\substack{v'\in V(G'),\\ \ph(v')=v}}d_{\ph}(v')
$$
is called the {\it degree} of $\ph$ at $v$. It is easy to check that $d_{\ph}(v)$ is constant along any connected component of $G$. In particular, if $G$ is connected, we call this number the {\it degree} of $\ph$ and denote it $\deg(\ph)$, note that for any $v\in V(G)$, $e\in E(G)$ or $p\in L(G)$ we have 
$$
\deg(\ph)=\sum_{\substack{v'\in V(G'), \\ \ph(v')=v}}d_{\ph}(v')=\sum_{\substack{e'\in E(G'), \\ \ph(e')=e}}d_{\ph}(e')=
\sum_{\substack{p'\in L(G'),\\ \ph(p')=p}}d_{\ph}(p').
$$
We also observe that if $G$ is connected, then any harmonic morphism $\ph:G'\to G$ of positive degree is necessarily surjective, and in particular $\val(v')\geq \val(\ph(v'))$ for any $v'\in V(G')$.

We say that a harmonic morphism $\ph:G'\to G$ is {\it finite} if $d_{\ph}(h')>0$ for each $h'\in H(G')$, in this case $\ph$ does not contract any edges or legs and is surjective if $G$ is connected. A finite harmonic morphism $\ph:G'\to G$ induces maps on the tangent spaces $d\ph_{v'}:T_{v'}G'\to T_{\ph(v')}G$ for all $v'\in V(G')$.

\begin{remark} Let $\ph:G'\to G$ be a finite harmonic morphism of graphs. Given a length function $\ell$ on $G$, there is a unique length function $\ell'$ on $G'$ such that $\ph$ is a harmonic morphism of the corresponding metric graphs. Indeed, given an edge $e'\in E(G')$, we set $\ell'(e')=\ell\left(\ph(e')\right)/d_{\ph}(e')$. 
\label{rem:edgelengths}
\end{remark}

\begin{definition} Let $\ph:G'\to G$ be a finite harmonic morphism of weighted graphs. The {\it ramification degree} $\Ram_{\ph}(v')$ of $\ph$ at a vertex $v'\in V(G')$ is equal to
\begin{equation}
\Ram_{\ph}(v')=d_{\ph}(v')\chi(\ph(v'))-\chi(v').
\label{eq:Ram}
\end{equation}
We say that $\ph$ is {\it effective} if $\Ram_{\ph}(v')\geq 0$ for all $v'\in V(G')$ and {\it unramified} if $\Ram_{\ph}(v')=0$ for all $v'\in V(G')$, and we remark again that there is no vertex at the end of a leg. The condition $\Ram_{\ph}(v')=0$ is called the {\it Riemann--Hurwitz condition} at $v'$. The {\it ramification degree} $\Ram \ph$ is the sum of the ramification degrees at the vertices of $G'$:
\begin{equation}
\Ram (\ph)=\sum_{v'\in V(G')}\Ram_{\ph}(v').
\label{eq:Ramglobal}
\end{equation}
A finite harmonic morphism of tropical curves is called {\it effective} or {\it unramified} if it has an effective or unramified model, respectively.
\end{definition}
If $\ph:G'\to G$ is a finite harmonic morphism of connected graphs, adding the local ramification degrees we obtain the global Riemann--Hurwitz formula
\begin{equation}
\chi(G')=\deg (\ph)\chi(G)-\Ram(\ph).
\label{eq:RHformula}
\end{equation}
For an unramified morphism $\ph$, the global Riemann--Hurwitz formula is
\begin{equation}
\chi(G')=\deg(\ph)\chi(G).
\end{equation}

\begin{remark} A simple calculation shows that our definition agrees with the standard one in the literature (see, for example, Section 2.2 in \cite{ABBRII}):
\begin{equation}
\Ram_{\ph}(v')=d_{\ph}(v')\left(2-2g(\ph(v'))\right)-\left(2-2g(v')\right)-\sum_{h'\in T_{v'}G'} \left(d_{\ph}(h')-1\right).
\label{eq:Ram2}
\end{equation}
We also note that unramified harmonic morphisms are called {\it tropical admissible covers} in~\cite{CavalieriMarkwigRanganathan_tropadmissiblecovers}.
\end{remark}

\begin{remark} Let $G'$ be a graph, let $G$ be a weighted graph, and let $\ph:G'\to G$ be a finite harmonic morphism. Then there exists at most one weighting of $G'$ that makes $\ph$ an unramified harmonic morphism. Indeed, for a vertex $v'\in V(G')$, solving $\Ram_{\ph}(v')=0$ for $g(v')$ gives us
\begin{equation}
g(v')=1-d_{\ph}(v')+ d_{\ph}(v') g(\ph(v'))+ \frac{d_{\ph}(v')\val(\ph(v'))-\val(v')}{2}.
\label{eq:admissiblegenus}
\end{equation}
If this number is a non-negative integer for all $v'\in V(G')$, then it defines a weighting of $G'$ with respect to which $\ph$ is unramified, otherwise there is no such weighting. 
\label{rem:uniqueweighting}
\end{remark}

\begin{remark} From a moduli-theoretic perspective, we are primarily interested in unramified morphisms of metric graphs. However, Remark~\ref{rem:edgelengths} shows that the metric structure does not play an interesting role in the combinatorics of unramified morphisms. For this reason, to simplify the exposition we formulate our results on unramified morphisms for graphs without a metric, and reintroduce metric graphs only when will consider moduli spaces of tropical curves. 

\end{remark}

If $\ph:G'\to G$ is a harmonic morphism and $F'$ is a subgraph of $G'$, then the restriction $\ph|_{F'}:F'\to \ph(F')$ is not necessarily harmonic. It is, however, harmonic if $F'=\ph^{-1}(F)$ for some subgraph $F$ of $G$. We can extend this morphism to a harmonic morphism $\ph|_{O(F')}:O(F')\to O(F)$ of the corresponding neighborhoods (see Def.~\ref{def:neighborhood}). Note that $\deg \ph|_{O(F')}=\deg \ph|_{F'}=\deg \ph$. Furthermore, if $F'_0$ is a connected component of $F'$ that maps surjectively onto $F$, then we can further restrict $\ph$ to harmonic morphisms $\ph|_{F'_0}:F'_0\to F$ and $\ph|_{O(F'_0)}:O(F'_0)\to O(F)$.

\begin{remark} We observe that restricting a harmonic morphism increases the ramification degree at the remaining vertices, and hence the restriction of an effective morphism is effective. Indeed, let $\ph:G'\to G$ be a harmonic morphism, let $v\in V(G)$, let $h\in T_v(G)$, and let $G_0$ be the subgraph of $G$ obtained by removing $h$, if it is a leg, or the edge containing $h$. Let $G'_0=\ph^{-1}(G_0)$, let $\ph|_{G'_0}:G'_0\to G_0$ be the restriction, let $v'\in \ph^{-1}(v)$, and let $d=d_{\ph}(v')$. If $h$ is not one half of a loop at $v$, then $\chi_{G_0}(v)=\chi_G(v)+1$, while $\chi_{G'_0}(v')\leq \chi_{G'}(v')+d$, because $h$ has at most $d$ preimages at $v'$. Similarly, if $h$ is half of a loop at $v$, then $\chi_{G_0}(v)=\chi_G(v)+2$ and $\chi_{G'_0}(v')\leq \chi_{G'}(v')+2d$, and in either case we have $\Ram_{\ph|_{G'_0}}(v')\geq \Ram_{\ph}(v')$. 
\label{rem:restriction}
\end{remark}

If $\ph:G'\to G$ is unramified, $F$ a connected subgraph of $G$, and $F'_0$ be a connected component of $\ph^{-1}(F)$, then the finite effective harmonic morphism $\ph|_{F'_0}:F'_0\to F'$ is not necessarily unramified. However, the restriction to the neighborhoods $\ph|_{O(F'_0)}:O(F'_0)\to O(F)$ is unramified, since the morphisms $\ph$ and $\ph|_{O(F'_0)}$ are identical in the neighborhood of any vertex $v'\in V(F'_0)$. 

\begin{remark} We also observe that a finite effective harmonic morphism $\ph:G'\to G$ can be promoted to an unramified morphism by adding legs to $G'$ and $G$. Indeed, let $v'\in V(G')$ be a vertex with $\Ram_{\ph}(v')\geq 1$, let $v=\ph(v')$, and let $\ph^{-1}(v)=\{v'_1,\ldots,v'_m\}$ with $v_1=v'$. We modify $\ph:G'\to G$ as follows: attach a leg $p$ to $v$, attach $d_{\ph}(v')-1$ legs to $v'$, and attach $d_{\ph}(v'_i)$ legs to $v'_i$ for $i=2,\ldots,m$. Map all these new legs of $G'$ to $p$. Set the degree of $\ph$ to be 2 at one of the legs attached to $v'$, and 1 at all other legs. The resulting graph morphism is a finite effective harmonic morphism having ramification index one less at $v'$, and the same at all other vertices. Performing this procedure $\Ram(\ph)$ times, we obtain an unramified morphism.
\label{rem:addramification}
\end{remark}

We now show that an unramified harmonic morphism of graphs induces an unramified morphism of their stabilizations. Indeed, let $\ph:G'\to G$ be unramified. For any $v'\in V(G')$, the condition $\Ram_{\ph}(v')=0$ implies that $\chi(v')$ and $\chi(\ph(v'))$ have the same sign, so $G'$ is semistable or stable if and only if $G$ is. Furthermore, by~\eqref{eq:RHformula} $\chi(G')<0$ if and only if $\chi(G)<0$.

If $G$ is unstable, pick a vertex $v\in V(G)$ with $\chi(v)=1$, and let $e\in E(G)$ be the unique edge rooted at $v$. Then $\chi(v')=1$ and $d_{\ph}(v')=1$ for all $v'\in \ph^{-1}(v)$. Hence $v$ has $\deg(\ph)$ preimages $v'_i$, each of which is a root vertex of a unique edge $e'_i$ mapping to $e$, and $d_{\ph}(e'_i)=1$ for all $i$. Let $u\in V(G)$ be the other root vertex of $e$. For any $u'\in \ph^{-1}(u)$, $d_{\ph}(u')$ of the $e'_i$ are rooted at $u'$. Therefore, removing $v$, $e$, $v'_i$, and $e'_i$ increases $\chi(u)$ by $1$ and increases $\chi(u')$ by $d_{\ph}(u')$ for each $u'\in \ph^{-1}(u)$. This does not change the ramification degree at $u'$. Proceeding in this way, we simultaneously remove the unstable trees of $G'$ and $G$ and obtain an unramified morphism $\ph_{sst}:G'_{sst}\to G_{sst}$.

Similarly, it is clear that the image and preimage of a simple vertex is simple. Let $v\in V(G)$ be a simple vertex with tangent directions $h_1$ and $h_2$, and let $v'\in \ph^{-1}(v)$, then $v'$ is a simple vertex with tangent directions $h'_1$ and $h'_2$ mapping with the same degree to $h_1$ and $h_2$, respectively. We now remove $v$ and join $h_1$ to $h_2$, and at the same time remove $v'$ and join $h'_1$ and $h'_2$ for each $v'\in \ph^{-1}(v)$. We extend $\ph$ to the edges or legs obtained by this joining in the obvious manner. Proceeding in this way, we remove the simple vertices of $G'$ and $G$ and obtain an unramified morphism $\ph_{st}:G'_{st}\to G_{st}$.

\begin{definition} Let $\ph:G'\to G$ be an unramified harmonic morphism of connected weighted graphs, and suppose that $\chi(G')<0$ (equivalently, suppose that $\chi(G')<0$). The unramified harmonic morphism $\ph_{st}:G'_{st}\to G_{st}$ constructed above is called the {\it stabilization} of $\ph$.
\end{definition}

We give a corresponding definition for unramified harmonic morphisms of tropical curves. The stabilization of a tropical curve is naturally a subset, so we obtain the stabilization of an unramified harmonic morphism by restriction. 

\begin{definition} Let $\tau:\Ga'\to \Ga$ be an unramified harmonic morphism of tropical curves, and suppose that $\chi(\Ga')<0$ (equivalenly, suppose that $\chi(\Ga')<0$). The {\it stabilization} $\tau_{st}:\Ga'_{st}\to \Ga_{st}$ is the unramified harmonic morphism obtained by restricting $\tau$ to $\Ga'_{st}$. \label{def:stabilization}

\end{definition}

The following simple observation will prove to be important.

\begin{proposition} Let $G$ be a connected weighted graph. There are finitely many unramified harmonic morphisms $\ph:G'\to G$ of a given degree $d$.

\label{prop:finitelymany}
\end{proposition}

\begin{proof} We give an explicit algorithm enumerating all such covers. 
\begin{enumerate} \item Pick a partition $d_v$ of $d$ for each vertex $v\in V(G)$, and similarly partitions of $d$ for each $e\in E(G)$ and $p\in L(G)$.
\item For each $v\in V(G)$ we let $\ph^{-1}(v)$ be a set indexed by the partition $d_v$, and for $v'\in \ph^{-1}(v)$ define $d_{\ph}(v')$ to be the corresponding element of the partition. We similarly define the edges and legs of $G'$ and the degree of $\ph$ on them.
\item For each edge $e\in E(G)$ rooted at a vertex $v\in V(G)$, there are finitely many ways to attach the elements of $\ph^{-1}(e)$ to the elements of $\ph^{-1}(v)$. We similarly attach the legs $L(G')$ to the vertices $V(G')$.
\item We have constructed a graph $G'$ and a morphism $\ph:G'\to G$. We check whether $G'$ is connected. 
\item If the harmonicity condition~\eqref{eq:harmonic} holds at each $v'\in V(G')$, then we have defined a harmonic morphism $\ph:G'\to G$ of degree $d$. 
\item By Rem.~\ref{rem:uniqueweighting}, there is at most one way to assign genera $g(v')$ for all $v'\in V(G')$ so that $\ph$ is unramified. 
\end{enumerate}\end{proof}

This algorithm is of course computationally intractable for large $d$. However, in Sec.~\ref{sec:hyperelliptic} we will give an explicit implementation of this algorithm in the hyperelliptic case $d=2$. 


\subsection{Edge contraction}

In this section, we recall the operation of edge contraction introduced (see Sec.~3.2.5 of~\cite{CavalieriMarkwigRanganathan_tropadmissiblecovers} or Def.~2.22 in \cite{Caporaso_handbook}). For the convenience of the reader, we spell out the operation in detail.

\begin{definition} \label{def:contractionS} Let $(G,g_G)$ be a weighted graph, and let $S\subset E(G)$ be a set of edges of $G$. The {\it contraction} of $G$ along $S$ is a weighted graph $(G/S,g_{G/S})$ together with a non-finite morphism $c_S:G\to G/S$ defined as follows.

Let $F=[S]$ be the subgraph of $G$ generated by $S$, and let $F_1,\ldots,F_k$ be the connected components of $F$. The graph $G/S$ is obtained from $G$ by contracting each $F_i$ to a vertex $v_i$ of genus $g(F_i)$, and $c_S$ is the natural surjective map. Specifically:

\begin{itemize}
\item The set of vertices of $G/S$ is $V(G/S)=V(G)\backslash V(F)\cup \{v_1,\ldots,v_m\}$. We define the contraction map $c_S$ on $V(S)$ and the weighting $g_{G/S}$ by
$$
c_s(v)=\begin{cases} v_i & \textrm{ if } v\in V(F_i), \\ v & \textrm{ if } v\in V(G\backslash F), \end{cases}
$$
and
$$
g_{G/S}(v)=\begin{cases} g(F_i) & \textrm{ if } v=v_i, \\ g_{G}(v) & \textrm{  if } v\in V(G)\backslash V(F).\end{cases}
$$

\item The set of half-edges of $G/S$ is $H(G/S)=H(G)\backslash H(F)$. We define the contraction map $c_S$ on $H(G)$ by
$$
c(h)=\begin{cases} v_i & \textrm{ if } h\in H(F_i), \\ h & \textrm{ if } h\in H(G)\backslash S. \end{cases}
$$

\item We define the root map on $G/S$ as $r_{G/S}=c_S\circ r_G$, and the involution $\iota_{G/S}$  by restricting $\iota_G$ to $H(G/S)$.

\end{itemize}
The contraction $c_S$ establishes a bijection between the half-edges of $G/S$ and the half-edges of $G$ that are not in $S$. Hence there are bijections between $E(G)\backslash S$ and $E(G/S)$, and between $L(G)$ and $L(G/S)$. 
\end{definition}

If $S=S_1\sqcup S_2$ is a subset of edges of $E(G)$, then $S_2$ is naturally a subset of edges of $E(G/S_1)$, and the contraction map $G\to G/S$ factors as $G\to G/S_1\to (G/S_1)/S_2=G/S$. Hence any contraction is a composition of {\it edge contractions}, which we describe explicitly. Let $e\in E(G)$ be an edge. The contraction $G/e$ of $G$ along $e$ is obtained from $G$ as follows:
\begin{itemize}

\item If $e$ is a loop at $v\in V(G)$, then delete $e$ and increase the genus of $v$ by 1.

\item If $e$ has distinct endpoints $u$ and $v$, then delete $e$ and identify $u$ and $v$ to a new vertex of genus $g(u)+g(v)$. 

\end{itemize}

We can also define edge contractions in an invariant way as follows.

\begin{definition} Let $(G',g')$ and $(G,g)$ be weighted graphs. A {\it contraction} $c:(G',g')\to (G,g)$ is a graph morphism satisfying the following properties:
\begin{itemize}
\item For each half-edge $h\in H(G)$, the preimage $c^{-1}(h)\in H(G')$ consists of one element.

\item For each $p'\in L(G')$ we have $c(p')\in L(G)$, in other words $c$ does not contract any legs. 
\item For each vertex $v\in V(G)$, the weighted graph $c^{-1}(v)$ is connected with genus $g(v)$.
\end{itemize}\label{def:contraction}
\end{definition}

It is an exercise to check that a contraction according to Def.~\ref{def:contraction} is a composition of a contraction along a subset of edges (using Def.~\ref{def:contractionS}) and an isomorphism of weighted graphs.

We can also simultaneously contract the target and source of an unramified harmonic morphism along a subset of edges of the target and its preimage in the source. This definition is implicitly used in Proposition 19 in~\cite{CavalieriMarkwigRanganathan_tropadmissiblecovers}, and the lemma that follows is a restatement of this result. We provide a proof for the sake of completeness, to illustrate the convenience of using the Euler characteristic, and because we will later require a converse result (see Lem.~\ref{lem:edgecontraction}).

\begin{definition} \label{def_contractionofunramifiedharmonic}
Let $\ph:G'\to G$ be an unramified harmonic morphism of weighted graphs, and let $S\subset E(G)$ be a subset of edges of $G$. Let $S'=\ph^{-1}(S)$. We define a harmonic morphism, called the {\it contraction} $\ph_S:G'/S'\to G/S$ of $\ph$ along $S$, as follows.

Let $F'=[S']$ and $F=[S]$ be the subgraphs of $G'$ and $G$ generated by $S'$ and $S$, respectively. To simplify notation we assume that $F$ is connected, and denote the connected components of $F'$ by $F'_1,\ldots,F'_m$.

\begin{itemize}

\item The set of half-edges of $G'/S'$ is $H(G'/S')=H(G')\backslash H(F')$, and for $h'\in H(G'/S')$ we define $\ph_S(h')=c_S(\ph(h'))$ with degree $d_{\ph_S}(h')=d_{\ph}(h')$. 

\item The set of vertices of $G'/S'$ is $V(G'/S')=V(G')\backslash V(F')\cup \{v'_1,\ldots,v'_m\}$, where each $v'_i$ corresponds to a connected component $F'_i$ of $F'$. For $v'\in V(G')\backslash V(F')$ we define $\ph_{S}(v')=c_S(\ph(v'))$ with degree $d_{\ph_S}(v')=d_{\ph}(v')$. Since $F$ is connected, $\ph$ maps $F'_i$ onto $F$, and moreover the restriction $\ph|_{F'_i}:F'_i\to F$ is a harmonic morphism (not necessarily unramified). We define $\ph_{S}(v'_i)=v$, where $v$ is the vertex of $G/S$ corresponding to $F$, with degree $d_{\ph_{S}}(v'_i)=\deg \ph|_{F'_i}$.

\end{itemize}
For arbitrary $S$ we define the contraction of $\ph$ by composing the contractions along the connected components of $[S]$. As with contractions of graphs, contractions of covers can be factored as compositions of contractions along disjoint subsets of edges.
\label{def:unramifiedcontraction}
\end{definition}

The following is an example of an edge contraction of an unramified harmonic morphism $\ph:G'\to G$ along a single edge $e\in E(G)$. The morphism $\ph:G'\to G$ is the right column (only the neighborhoods of the contracted subgraphs are shown), the contracted edge and its preimages in $G'$ are bold, and the contracted morphism is the left column. The edges and half-edges of the sources are marked by degree.

\begin{center}
\begin{tikzpicture}

\begin{scope}[shift={(0,0)}]

\draw [thick] (0,0) -- (-0.8,0.4);
\draw [thick] (0,0) -- (-0.8,-0.4);
\draw [thick] (0,0) -- (0.8,0.4);
\draw [thick] (0,0) -- (0.8,-0.4);
\draw[fill](0,0) circle(.8mm);

\end{scope}

\begin{scope}[shift={(3,0)}]

\draw [thick] (0,0) -- (-0.8,0.4);
\draw [thick] (0,0) -- (-0.8,-0.4);
\draw[fill](0,0) circle(.8mm);
\draw [ultra thick] (0,0) -- (2,0) node[midway,above]{$e$};

\begin{scope}[shift={(2,0)}]
\draw [thick] (0,0) -- (0.8,0.4);
\draw [thick] (0,0) -- (0.8,-0.4);
\draw[fill](0,0) circle(.8mm);

\end{scope}

\end{scope}

\begin{scope}[shift={(0,2)}]

\draw [thick] (0,0) -- (-0.8,0.6) node[left]{$2$};
\draw [thick] (0,0) -- (-0.8,0.2) node[left]{$3$};
\draw [thick] (0,0) -- (-0.8,-0.2) node[left]{$2$};
\draw [thick] (0,0) -- (-0.8,-0.6) node[left]{$3$};
\draw [thick] (0,0) -- (0.8,0.6) node[right]{$4$};
\draw [thick] (0,0) -- (0.8,0.2) node[right]{$1$};
\draw [thick] (0,0) -- (0.8,-0.2) node[right]{$3$};
\draw [thick] (0,0) -- (0.8,-0.6) node[right]{$2$};
\draw[fill,white](0,0) circle(2mm);
\draw[thick](0,0) circle(2mm);
\draw (0,0) node{$2$};

\end{scope}

\begin{scope}[shift={(3,2)}]

\draw [ultra thick] (0,0.7) .. controls (0.5,0.7) and (1.5,0.5) .. (2,0) node[midway,above] {$1$} ;
\draw [ultra thick] (0,0.7) .. controls (0.5,0.2) and (1.5,0) .. (2,0);
\draw (0.7,0) node{$1$};

\draw [ultra thick] (0,-0.7) -- (2,0) node[midway,below] {$3$} ;

\begin{scope}[shift={(0,0.7)}]
\draw [thick] (0,0) -- (-0.8,0.4) node[left]{$2$};
\draw [thick] (0,0) -- (-0.8,-0.4) node[left]{$2$};
\draw[fill](0,0) circle(.8mm);
\end{scope}

\begin{scope}[shift={(0,-0.7)}]
\draw [thick] (0,0) -- (-0.8,0.4) node[left]{$3$};
\draw [thick] (0,0) -- (-0.8,-0.4) node[left]{$3$};
\draw[fill,white](0,0) circle(2mm);
\draw[thick](0,0) circle(2mm);
\draw (0,0) node{$1$};
\end{scope}

\begin{scope}[shift={(2,0)}]
\draw [thick] (0,0) -- (0.8,0.6) node[right]{$4$};
\draw [thick] (0,0) -- (0.8,0.2) node[right]{$1$};
\draw [thick] (0,0) -- (0.8,-0.2) node[right]{$3$};
\draw [thick] (0,0) -- (0.8,-0.6) node[right]{$2$};

\draw[fill](0,0) circle(0.8mm);
\end{scope}

\end{scope}

\end{tikzpicture}
\end{center}

\begin{lemma} \label{lemma_contractionofunramifiedharmonic} Let $\ph:G'\to G$ be an unramified harmonic morphism of weighted graphs, and let $S\subset E(G)$. Then the contraction $\ph_S:G'/S'\to G/S$ of $\ph$ along $S$ is an unramified harmonic morphism.
\end{lemma}

\begin{proof} We again assume that $S$ is connected, and use the notation of Def.~\ref{def:unramifiedcontraction}. The map $\ph_S$ does not contract any half-edges, so it is sufficient to check that it satisfies the harmonicity condition \eqref{eq:harmonic} and $\Ram_{\ph_S}(v')=0$ at every vertex $v'$ of $G'/S'$. In the neighborhood of a vertex $v'\in V(G')\backslash V(F')$, the map $\ph_S$ is identical to $\ph$, hence it is harmonic and unramified, so it remains to check the vertices $v'_i$.

We consider the neighborhoods $O(F)$ and $O(F'_i)$ and the restriction $\ph_i: O(F'_i)\to O(F)$ of $\ph$ to $O(F'_i)$. We observe that $\chi(v)=\chi(O(F))$ and $\chi(v'_i)=\chi(O(F'_i))$. Now let $h\in H(G/S)$ be a half-edge rooted at $v$, it corresponds to a leg $h\in L(O(F))$ rooted at some vertex $u\in V(O(F))$. Similarly, a half-edge $h'\in H(G'/S')$ rooted at $v'_i$ and mapping to $h$ corresponds to a leg $h'\in H(O(F'_i))$ rooted at some vertex $u'\in V(O(F'_i))$ and mapping to $u$. Therefore,
\begin{equation*}\begin{split}
\sum_{\substack{ h'\in T_{v'_i}(G'/S'),\\ \ph_S(h')=h}} d_{\ph_S}(h')&= \sum_{\substack{u'\in V(O(F'_i)), \\ \ph(u')=u}}\;\sum_{\substack{ h'\in T_{u'}\tF'_i, \\ \ph(h')=h}} d_{\ph}(h')\\&=\sum_{\substack{u'\in V(O(F'_i)), \\ \ph(u')=u}}d_{\ph}(u')
=\deg \ph_i=d_{\ph_S}(v'_i).
\end{split}\end{equation*}
Hence $\ph_S$ is harmonic at $v'_i$. Similarly
\begin{equation*}\begin{split}
\chi(v'_i)&=\chi(O(F'_i))\\ &=\sum_{u'\in V(O(F'_i))}\chi(u')\\ &=\sum_{u'\in V(O(F'_i))}d_{\ph}(u')\chi(\ph(u')) 
\\ &=\sum_{u\in V(F)}\chi(u)\sum_{\substack{u'\in V(O(F'_i)), \\ \ph(u')=u}}d_{\ph}(u')\\ &=\sum_{u\in V(O(F))}\chi(u)\deg \ph_i \\ &=\chi(\tF)\deg \ph|_{F'_i}=d_{\ph_S}(v'_i)\chi(v'_i),
\end{split}\end{equation*}
therefore $\ph_S$ is unramified.
\end{proof}

We will later see in Lem.~\ref{lem:edgecontraction} that edge contractions of unramified harmonic morphisms can be (non-uniquely) reversed: given $G$ and $S$ as above, and given an unramified harmonic morphism $\ph$ with target $G/S$, there exists an unramified harmonic morphism to $G$ whose contraction is $\ph$.


\subsection{Divisors and rational functions}

Let $\Ga$ be a tropical curve. A {\it divisor} $D$ on $\Ga$ is a finite formal sum $D=\sum_{i=1}^na_ix_i$ over points $x_i$ of $\Ga$, with integer coefficients. We write $\Div(\Ga)$ for the group of divisors on $\Ga$. The {\it degree} function $\deg\colon\Div(\Ga)\rightarrow \Z$ given by $\deg (D)=\sum_{i=1}^n a_i$ defines a natural grading on $\Div (\Ga)$, and we denote the set of degree $d$ divisors by $\Div_d(\Ga)$. 

\begin{definition} Let $\tau:\Ga'\to \Ga$ be a finite harmonic morphism of tropical curves. The {\it ramification divisor} of $\tau$ is the divisor
$$
\sum_{x'\in \Ga'} \Ram_{\tau}(x') x'.
$$
\end{definition}

It is clear that a finite harmonic morphism $\tau:\Ga'\to \Ga$ of tropical curves is effective if and only if the ramification divisor is effective, and unramified if and only if the ramification divisor is zero.

Let $\Ga$ be a tropical curve. A {\it rational function} $f$ on $\Ga$ is a continuous, piecewise-linear function $f:\Ga\to \RR$ with integer slopes. Given a point $x\in \Ga$ and a tangent direction $h\in T_x\Ga$, there is a well-defined {\it outgoing slope} $\slope_f(h)\in \ZZ$. We require rational functions to be eventually linear on every leg. In other words, if $p\subset \Ga$ is a leg, identified with $[0,\infty)$, then there exist $a\in \ZZ$, $b\in \RR$, and $c\in \RR_{\geq 0}$ such that for $x\in p$ we have $f(x)=ax+b$ for $x\geq c$; the number $a$ is called the {\it slope} of $f$ along $p$ and is denoted $\slope_f(p)$. We write $\Rat(\Ga)$ for the abelian group of rational functions on $\Ga$. The divisor $\div (f)$ associated to a rational function $f$ is given by 
\begin{equation*}
\div(f)=\sum_{x\in\Ga}\ord_f(x) x \ ,
\end{equation*}
where 
\begin{equation*}
\ord_f(x)=\sum_{h\in T_x \Ga}\slope_f(h)
\end{equation*}
is the sum of the outgoing slopes of $f$ at $x$. We note that this sum is finite, because $f$ is eventually linear on each leg and because a graph with no legs is compact. The association $f\mapsto \div(f)$ defines a natural homomorphism $\div\colon \Rat(\Ga)\rightarrow \Div(\Ga)$ whose image is easily checked to lie in $\Div_0(\Ga)$. 

\begin{definition}
We say that a rational function $F\in\Rat(\Ga)$ is \emph{harmonic} if $\div(F)=0$. 
\end{definition}

\begin{remark} It is easy to see that a rational function $f$ on a tropical curve $\Ga$ is the same thing as a morphism from $\Ga$ to $\RR$, where we view the latter as a tropical curve. This morphism is harmonic if and only if the function is harmonic, and finite if and only if the function has nonzero slope everywhere. 

\end{remark}

Recall that for a tropical curve $\Ga$ we denote $\Ga^{\circ}$ the tropical curve obtained from $\Ga$ by removing the legs, and $c:\Ga\to \Ga^{\circ}$ the natural retraction map. 

\begin{proposition}\label{prop_harmonicfunction=rationalfunction}
Let $(\Ga,p_1,\ldots, p_n)$ be a tropical curve with $n$ marked legs $p_1,\ldots, p_n$, and let $a_1,\ldots, a_n\in\Z$ be nonzero integers such that $a_1+\cdots+a_n=0$. There is a one-to-one correspondence between harmonic functions $f\colon\Ga\rightarrow \R$ having outgoing slope $-a_i$ on the leg $p_i$ for $i=1,\ldots,n$, and rational functions $f^{\circ}$ on $\Ga^{\circ}$ such that $\div(f^{\circ})=\sum_{i=1}^n a_i c(p_i)$.
\end{proposition}

\begin{proof}
Let $f\in \Rat(\Ga)$ be a harmonic function whose slope at the leg $p_i$ is $-a_i$ for all $i=1,\ldots n$, and let $f^{\circ}$ be the restriction of $f$ to $\Ga^{\circ}$. It is clear that $\div(f^{\circ})=\sum_{i=1}^n a_i c(p_i)$, since, by the harmonicity of $f$, at any point $x\in \Ga^0$ the sum of the outgoing slopes of $f$ along the finite edges of $\Ga$ is equal to $\sum_{c(p_i)=x} a_i$. Conversely, given a rational function $f^{\circ}\in \Rat(\Ga^{\circ})$ with $\div(f^{\circ})=\sum_{i=1}^n a_i c(p_i)$, we uniquely extend $f^{\circ}$ to a harmonic function $f$ on $\Ga$ with outgoing slope $-a_i$ along the leg $p_i$. 
\end{proof}

\begin{proposition}\label{prop_pullbackharmonic}
Let $\tau \colon \Ga'\rightarrow\Ga$ be a harmonic morphism of tropical curves, and let $f\colon \Ga\rightarrow \R$ a harmonic function. Then the pullback $\tau^\ast f\colon \Ga'\rightarrow \R$, given by $\tau^\ast f(x')=f(\tau(x'))$ for $x'\in \Ga'$, is also harmonic. Moreover, for a leg $p'_i$ of $\Ga'$ mapping to the leg $p_i$ of $\Ga$ we have
\begin{equation*}
\slope_{\tau^\ast f}(p'_i)=d_{\tau}(p_i')\cdot \slope_{f}(p_i).
\end{equation*}
\end{proposition}

\begin{proof}
Let $f\in\Rat(\Ga)$. For a half-edge $h'$ of $\Ga'$ rooted at a vertex $v'$, the outgoing slope of $\tau^\ast f$ is given by the chain rule:
\begin{equation*}
\slope_{\tau^\ast f}(h')=d_\tau(h')\cdot \slope_f (h),
\end{equation*}
where $h=\tau(h')$. Therefore, at every point $x'$ of $\Ga'$ we have 
\begin{equation*}
\sum_{\substack{h'\in T_{x'}\Ga', \\ \tau(h')=h}}\slope_{\tau^\ast f} (h')=\sum_{\substack{h'\in T_{x'}\Ga', \\ \tau(h')=h}}d_{\tau}(h')\cdot \slope_f (h)=d_{\tau}(x')\cdot \slope_f (h).
\end{equation*}
Adding this over all tangent directions $h\in T_x\Ga$, where $x=\tau(x')$, we have
\begin{equation*}
\ord_{\tau^\ast f}(x')= d_{\tau}(x') \cdot\ord_f(x). 
\end{equation*}
If $f$ is harmonic, then $\ord_f(x)=0$, hence $\ord_{\tau^{\ast}f}(x')=0$ and therefore $\tau^\ast f$ is also harmonic.
\end{proof}

Given a tropical curve $\Ga$ with a leg $p$, it is easy to see that any point $x\in p$ is linearly equivalent to the root point $c(p)$, via the function with slope $1$ on the segment connecting $c(p)$ and $x$ and slope zero everywhere else. Hence we can talk about divisors containing legs of $\Ga$, and linear equivalence between them.

\begin{corollary}[\cite{ABBRII} Prop.~4.2]\label{cor_tree->linearequivalence} Let $\Ga$ be a tropical curve, and let $\tau \colon\Ga\rightarrow \De$ be a harmonic morphism to a metric tree. Then the preimages of two legs are linearly equivalent.
\end{corollary}

\begin{proof} Pick two legs $0$ and $\infty$ on $\De$. It is clear that there exists a harmonic function $f$ on $\De$ (unique up to an additive constant) having slopes $+1$ and $-1$ on $0$ and $\infty$, respectively, and slope $0$ along all other legs. By Prop.~\ref{prop_pullbackharmonic}, the pullback $\tau^{\ast}f$ is harmonic on $\Ga$, having slope $d_{\tau}(p)$ along any leg $p\subset \Ga$ mapping to $0$, slope $-d_{\tau}$ along any leg mapping to $\infty$, and slope zero along all other legs. Hence, by Prop.~\ref{prop_harmonicfunction=rationalfunction}, the restriction of $- \tau^\ast F$ to $\Ga^{\circ}$ defines an equivalence between the divisors $\sum_{p:\tau(p)=0} d_{\tau}(p)c(p)$ and $\sum_{p:\tau(p)=\infty} d_{\tau}(p) c(p)$.
\end{proof}


\section{Tropical moduli spaces}

In this section we recall and refine the definition of \emph{generalized cone complexes}, which provide us with a natural language to study tropical moduli spaces. We review two well-known examples: the moduli space of tropical curves $M_{g,n}^{trop}$ and the moduli space of tropical admissible covers $H^{\trop}_{g\rightarrow h,d}(\mu)$. We also introduce a moduli space $\Div^{trop}_{g,d_+,d_-}$ parametrizing triples consisting of a stable tropical curve $\Ga$ of genus $g$ and a pair of effective divisors $D_+$ and $D_-$ on $\Ga$ of degrees $d_+$ and $d_-$, respectively.


\subsection{Generalized cone complexes} \label{sec:conecomplexes}

In \cite[Definition II.1.5]{KKMSD} the authors introduce the notion of a \emph{(rational polyhedral) cone complex} in order to describe the combinatorial structure of toroidal embeddings without self-intersecting strata. In \cite{ACP} (see Section 2.6) this notion is generalized to \emph{generalized (rational polyhedral) cone complexes} that can also deal with self-intersecting toroidal strata. In this section, we recall this definition, informed by the stack-theoretic framework developed in \cite{CCUW}. 

A {\it rational polyhedral cone with integral structure} $(\si,M)$, or simply a {\it cone}, is a topological space $\si$ together with a finitely generated group $M$ of continuous real-valued functions on $\si$, such that the induced map $\si\to \Hom(M,\RR)$ is a homeomorphism onto a strictly convex polyhedral cone in $N_\R:=\Hom(M,\RR)$ which is rational with respect to the lattice $N:=\Hom(M,\ZZ)$. A {\it morphism} $f:(M',\si')\to(M,\si)$ of cones is a continuous map $f:\si'\to \si$ such that the pullback of any function in $M$ is in $M'$.

Given a cone $(\si,M)$, the {\it dual monoid} $S_\si$ consists of those functions in $M$ that are non-negative on $\si$. We recover $\si$ from $S_\si$ as the space of monoid homomorphisms
$$
\si=\Hom(S_\si,\RR_{\geq 0}).
$$
A morphism $f\colon (\sigma', M')\rightarrow(\sigma,M)$ of cones induces a homomorphism $f^\#\colon M\rightarrow M'$ that sends $S_{\sigma}$ to $S_{\sigma'}$. A {\it face} of $\si$ is a subset $\tau$ along which some function $u\in S_{\si}$ vanishes. The dual monoid $S_\tau$ is naturally  given by the localization 
\begin{equation*}
(S_\sigma)_u=\big\{s-ku\in M\big\vert s\in S_\sigma, k\in\Z\big\}
\end{equation*}
of $S_\sigma$ along the submonoid generated by $u$. 

We say that  a cone is \emph{sharp} if the monoid $S_\sigma$ has no non-trivial units or, in other words, if $\sigma$ spans $N_\R$. Write $S_\sigma^\ast$ for the subgroup of units in $S_\sigma$. To any cone $(\sigma,M)$ we may associate a sharp cone $(\sigma, \Mbar)$ by setting $\Mbar=M/S_\sigma^\ast$; a morphism $f\colon (\sigma',M')\rightarrow(\sigma,M)$ induces a morphism $\overline{f}\colon (\sigma',\Mbar')\rightarrow (\sigma,\Mbar)$, since $f^\#(S_{\sigma'}^\ast)\subseteq S_{\sigma}^\ast$. A morphism $(\sigma',M')\rightarrow (\sigma,M)$ is said to be a \emph{face map} if it induces an isomorphism $(\sigma',\Mbar')\xrightarrow{\simeq}(\tau,\Mbar)$ onto a face $\tau$ of $\sigma$. Notice that, in particular, all automorphisms of a cone $(\sigma, M)$ are face morphisms. 

From now on we work only with sharp cones and the term \emph{cone} will uniformly refer to a sharp cone. We denote the category of sharp cones by $\mathbf{RPC}$ and write $\mathbf{RPC}^f$ for the subcategory of face maps. 

\begin{definition}\label{def_conecomplex}
A \emph{generalized rational polyhedral cone complex} $\Sigma$, or simply a \emph{generalized cone complex}, is a topological space $\vert \Sigma\vert$ together with a presentation as a colimit of a diagram $\sigma\colon I\longrightarrow \mathbf{RPC}^f$ of face maps parameterized by an index category $I$, subject to the following axioms:
\begin{enumerate}[(i)]
\item For every face $\tau$ of $\sigma_i$ in the diagram, there is a morphism $j\rightarrow i$ in $I$ such that $\sigma_j=\tau$ and the induced map $\sigma_j\rightarrow \sigma_i$ is the sharp face inclusion of $\tau$ into $\sigma$.
\item If $u\colon j\rightarrow i$ and $v\colon: k\rightarrow i$ are morphisms in $I$ such that $\sigma_v\colon\sigma_k\rightarrow\sigma_i$ factors through a face map $\sigma_k\rightarrow \sigma_j$, then there is a unique morphism $w\colon k\rightarrow j$ such that $u\circ w=v$ and $\sigma_w\colon \sigma_k\rightarrow \sigma_j$ is the given face map.
\end{enumerate}
\end{definition}

The avid reader will notice that the two axioms in Definition \ref{def_conecomplex} precisely say that the functor $\sigma\colon I\rightarrow \mathbf{RPC}^f$ defines a category fibered in groupoids. If we take the colimit not in the category of topological spaces but rather in a $2$-categorical sense, we obtain the notion of a \emph{cone stack}, as introduced in \cite[Section 2]{CCUW}, an analogue Deligne-Mumford stacks over $\mathbf{RPC}$. If $I$ is the category associated to a poset, there is no $2$-categorical structure and we obtain a so-called \emph{cone space}, an analogue of an algebraic spaces over $\mathbf{RPC}$. Finally, if $I$ is a poset and, in (i), every face $\tau$ of $\sigma_i$ is the image of exactly one morphism $\sigma_j\rightarrow \sigma_i$, then the colimit of $\Sigma$ is a \emph{(rational polyhedral) cone complex} in the sense of \cite{KKMSD}.

In \cite[Section 2.6]{ACP} a generalized cone complex is defined to be a topological space $\vert \Sigma\vert$ together with a presentation as an arbitrary diagram of face maps; our definition is equivalent to theirs, since every diagram of face maps generates a unique minimal diagram fulfilling axioms (i) and (ii) so that its colimit does not change. 
 
A morphism $f\colon\Sigma\rightarrow\Sigma'$ of generalized cone complexes is a continuous map $f\colon \vert \Sigma\vert\rightarrow\vert\Sigma'\vert$ such that for every cone $\sigma_{i'}$ in $\Sigma'$ there is a cone $\sigma_i$ in $\Sigma$ such that $f\vert \sigma_i$ factors through $\sigma_{i'}\rightarrow\Sigma'$. Being a morphism of generalized cone complexes is a weaker condition than being a morphism of categories fibered in groupoids over $\mathbf{RPC}^f$: for example, given a finite group $G$ acting trivially on a cone $\sigma$, the quotient $\sigma/G$ is naturally isomorphic to $\sigma$. 

The moduli spaces $M_{g,n}^{trop}$, $\Div_{g,d_+,d_-}^{trop}$, and $H^{\trop}_{g\rightarrow h,d}(\mu)$, which we introduce below, are generalized cone complexes. We require an analogue of the notion of a subvariety for cone complexes. The image of a morphism of generalized cone complexes is not necessarily a subcomplex, since cones are not required to map isomorphically to cones. For example, the image of the map $\RR_{\geq 0}\to \RR_{\geq 0}^2$ defined by $x\mapsto (x,x)$ is not a subcomplex of $\RR_{\geq 0}^2$. This motivates the following definition.

\begin{definition} Let $\Sigma$ be a generalized cone complex. A subset $X\subset \vert \Sigma\vert$ is called {\it semilinear} if, for every cone $\sigma\rightarrow\Sigma$ in $\Sigma$ the preimage $X\cap\sigma$ of $X$ in $\sigma$ is a union of finitely many subsets, each of which is given by finitely many homogeneous $\Z$-linear non-strict inequalities in the coordinates on $\sigma$.
\end{definition}

Given a semilinear subset $X\subset |\Sigma|$, we can (non-canonically) choose a suitable subdivision of $\Sigma'$ of $\Sigma$ such that $X$ is a subcomplex of $\Sigma'$. In other words, a semilinear subset of $\Sigma$ is a subcomplex of a subdivision of $\Sigma$. In a slight abuse of notation, we simply write $X\subseteq \Sigma$ for a semilinear subset of a generalized cone complex $\Sigma$.

\begin{proposition} 
Let $f\colon\Sigma\rightarrow\Sigma$ be a locally finite morphism of generalized cone complexes and let $X\subset \Sigma$ be a semilinear subset. Then $f(X)\subset \Sigma$ is also a semilinear subset.
\label{prop:linear}
\end{proposition}

\begin{proof}
This follows using standard linear algebra. Notice, however, that the condition of $f$ being locally finite is necessary to ensure that the intersection of $f(X)$ with every cone is given by finitely many inequalities.
\end{proof}

We can also impose a stronger requirement on semilinear subsets.

\begin{definition} A semilinear subset $X\subset \Sigma$ is called {\it linear} if, for every cone $\sigma\rightarrow\Sigma$, the intersection $X\cap \sigma$ is the union of the finitely many subsets, each of which is the quotient of the intersection of an $\Aut(\sigma/\Sigma)$-invariant linear subspace of $\Span(\sigma)$ with $\sigma$.
\end{definition}

An equality is equivalent to a pair of non-strict inequalities, hence a linear subset is semilinear. We note that the image of a linear subset is not necessarily linear. For example, the subset $\{(x,y,z):x=y+z\}\subset \RR_{\geq 0}^3$ is linear, but its image under the map $\RR_{\geq 0}^3\to \RR_{\geq 0}^2$ sending $(x,y,z)$ to $(x,y)$ is the subset $\{(x,y):x\geq y\}\subset \RR_{\geq 0}^2$, which is only semilinear.


\subsection{Tropical curves and their moduli}\label{section_Mgntrop}

We now recall the construction of the moduli space $M_{g,n}^{trop}$ of stable tropical curves of genus $g$ with $n$ marked points, essentially following the treatment in \cite{ACP}.
Let $G$ be a graph with $n$ legs. A {\it marking} of $G$ is a bijection 
$$
p:\{1,\ldots,n\}\to L(G),
$$
where we denote $p_i=p(i)$. A {\it weighted marked graph} is a weighted graph with a choice of marking. A \emph{marking} of a tropical curve $\Ga$ is a marking of a model of $\Ga$; this does not depend on the choice of the model. In this section we recall from \cite[Section 4]{ACP} how to construct a generalized cone complex $\Mbar_{g,n}^{trop}$ that acts as a coarse moduli space of stable tropical curves of genus $g$ with $n$ marked legs. 

Let $g$ and $n$ be non-negative integers such that $2g-2+n>0$. Denote by $I_{g,n}$ the category whose objects are stable weighted marked graphs $(G,g_G,p)$ of genus $g$ with $n$ marked points. The set of morphisms between two such graphs $(G',g_{G'}',p')$ and $(G,g_G,p)$ is the set of contractions (see Def.~\ref{def:contraction}) $\ph:(G',g_{G'}',p')\to (G,g_G,p)$ that preserve the markings, in other words such that $\ph\circ p'=p$. Passing to the skeleton category, we can drop the marking and simply denote the legs of any object in $I_{g,n}$ by $p_1,\ldots,p_n$, where $p_i=p(i)$. Furthermore, given a morphism $\ph:G'\to G$ in $I_{g,n}$, we can identify $G$ with the contraction of $G'$ along a subset of edges $S'\subset E(G')$, so any such map induces an injective map $\ph^{-1}:E(G)\to E(G')$. 

The initial objects in $I_{g,n}$ are weighted marked graphs $G$ having $g(v)=0$ and $\val(v)=3$ for all $v\in V(G)$; we call such graphs {\it maximally degenerate}. The unique terminal object in $I_{g,n}$ is the graph $\bullet_{g,n}$ consisting of a genus $g$ vertex with $n$ legs and no edges. 

Given $G\in I_{g,n}$, we denote 
$$
M^{trop}_G:=(\RR_{\geq 0})^{E(G)} \ .
$$
This cone is naturally a parameter space for pairs $(\Ga,\phi)$ consisting of a stable tropical curve $\Ga$ and a weighted edge contraction $\phi\colon G\rightarrow G_{min}(\Ga)$, where $G_{min}(\Ga)$ denotes the minimal model of $\Ga$. The tropical curve $\Ga$ is hereby identified with $(l(e))_{e\in E(\Ga)}\in M^{trop}_G$.

A morphism $\phi\colon (G,g_G, p)\rightarrow (G',g_{G'},p')$ in $I_{g,n}$ induces a face morphism $M_{G}^{trop}\rightarrow M_{G'}^{trop}$:
\begin{itemize}
\item An automorphism of $(G,g_G,p)$ induces an automorphism of the cone $M_G^{trop}$ given by permuting the entries according to the permutation of $E(G)$. 
\item A weighted edge contraction $\phi\colon (G,g_G,p)\rightarrow(G',g_{G'},p')$ contracting the edges $S\subseteq E(G)$ induces a face map $M_{G'}^{trop}\hookrightarrow M_G^{trop}$ onto the face that is given by setting the $S$-coordinates in $M_G^{trop}=\R_{\geq 0}^{E(G)}$ equal to zero. 
\end{itemize}

\begin{proposition}
The contravariant functor $(G,g_{G},p)\mapsto M_G^{trop}$ defines a generalized cone complex $M_{g,n}^{trop}$ that functions as a parameter space of stable tropical curves of genus $g$ with $n$ marked legs. 
\end{proposition}

\begin{proof}
This easy verification of the axioms is left to the avid reader.
\end{proof}

\begin{definition}
The generalized cone complex $M_{g,n}^{trop}$ is called the \emph{(coarse) moduli space of stable tropical curves of genus $g$ with $n$ marked points.}
\end{definition}

\begin{remark}
 In our paper, we restrict our attention to $M_{g,n}^{trop}$ and do not study the compactified moduli space $\overline{M}_{g,n}^{trop}$ (as introduced in \cite[Section 3.3]{Caporaso_handbook}), which is obtained by allowing edge lengths to be infinite. 
\end{remark}

We now recall the tropical forgetful maps (as e.g. introduced in \cite[Section 8]{ACP}).

\begin{definition} We define the map $\pi:M^{trop}_{g,n+1}\to M^{trop}_{g,n}$ {\it forgetting the leg} $p_{n+1}$ in the following way. Let $\Ga\in M^{trop}_{g,n+1}$ be a metric graph with combinatorial type $G=C(\Ga)$, and let $v=r_G(p_{n+1})\in V(G)$ be the vertex at which the leg $p_{n+1}$ is attached. If $\chi(v)=1$, then the condition $2g-2+n>0$ implies that $g(v)=0$ and $\val(v)=3$, and that the other two half-edges at $v$ either correspond to two distinct edges, or to an edge and a leg. 
\begin{itemize}

\item Suppose that $g(v)=0$, $\val(v)=3$, and that $v$ is the endpoint of two edges $e_1,e_2\in E(G)$ whose other endpoints are $v_1$ and $v_2$, respectively. We let $\pi(G)$ be the graph obtained from $G$ by removing $p_{n+1}$, $v$, $e_1$, and $e_2$, and attaching an edge $e$ at $v_1$ and $v_2$, and we let $\pi(\Ga)$ be the metric graph obtained from $\pi(G)$ by setting $l_{\pi(G)}(e)=l_{G}(e_1)+l_{G}(e_2)$, and setting all other edge lengths to be the same as in $\Ga$.

\item Suppose that $g(v)=0$, $\val(v)=3$, and that $v$ is the endpoint of an edge $e\in E(G)$ whose other endpoint is $u$, and a leg $p_i$. We let $\pi(G)$ be the graph obtained from $G$ by removing $p_i$, $p_{n+1}$, $v$, and $e$ and reattaching $p_i$ to $u$, and we let $\pi(\Ga)$ be the metric graph whose combinatorial type is $\pi(G)$, and with the same edge lengths as $\Ga$.

\item If $\chi(v)\geq 2$, we let $\pi(\Ga)$ be the metric graph whose combinatorial type is $G\backslash \{p_{n+1}\}$, and with the same edge lengths as $\Ga$.

\end{itemize}

\end{definition}


\subsection{Moduli of tropical divisors}\label{section_modulitropicaldivisors}

Let $g\geq 1$, and let $d_+,d_-\geq 0$. In this section, we construct a generalized cone complex $\Div_{g,d_+,d_-}^{trop}$ that parameterizes triples $(\Ga, D_+, D_-)$ consisting of a stable tropical curve $\Ga$ of genus $g$ and two effective divisors $D_+\in Div_{d_+}^+(\Ga)$ and $D_-\in\Div_{d_-}^+(\Ga)$ on $\Ga$. The idea is that, for a fixed tropical curve $\Ga$, the locus where the support of $D_+$ and $D_-$ are disjoint parameterizes divisors of degree $d=d_+-d_-$ that can be written as a difference $D_+-D_-$ without cancellation. 

In this section only, we consider divisors on weighted graphs. For a weighted graph $G$, we define $\Div (G)$ to be the free abelian group on $V(G)$. It is clear that a divisor $D$ on a weighted graph $G$ induces a divisor on any tropical curve of which $G$ is a model, and conversely any divisor on a tropical curve can be obtained in this way for an appropriate choice of model.

We consider the category $I_{g,(d_+, d_-)}$. Its objects are triples $(G, D_+,D_-)$ of the following form.

\begin{itemize}
\item $G$ is a semistable weighted graph of genus $g$.

\item $D_+\in\Div_{d_+}(G)$ and $D_-\in\Div_{d_-}(G)$ are effective divisors of degrees $d_+$ and $d_-$, respectively, such that the union of the supports of $D_+$ and $D_-$ contains all simple vertices of $G$.  
\end{itemize}
The morphisms in the category $I_{g,(d_+, d_-)}$ are generated by the following:

\begin{itemize}
\item For every object $(G,D_+,D_-)$, an isomorphism $\ph:G\to G'$ induces a corresponding isomorphism $\ph:(G,D_+,D_-)\to (G',D'_+,D'_-)$, where $D'_{\pm}=\ph_*D_{\pm}$. 

\item For every object $(G,D_+,D_-)$ and every edge set $S\subset E(G)$, the edge contraction $\ph_S:G\to G/S$ induces a morphism $\ph_S:(G,D_+,D_-)\to (G/S,D^S_+,D^S_-)$, where $D^S_{\pm}=[\ph_S]_*(D_\pm)$.
\end{itemize} 
We remark that automorphisms and weighted edge contractions are (non-finite) harmonic morphisms of degree one, so for a divisor $D=\sum a_i v_i$ its pushforward under such a $\ph$ is simply $\ph_*D=\sum a_i \ph(v_i)$. 

Given an object $\left(G,D_+, D_-\right)$ in $I_{g,(d_+,d_-)}$, we denote by $\Div_{(G,D_+,D_-)}$ the cone 
\begin{equation}\label{eq_HassetttoDiv}
\Div_{(G,D_+,D_-)}=(\R_{\geq 0})^{E(G)} \ . 
\end{equation}
This association defines a contravariant functor $I_{g,(d_+,d_-)}\rightarrow \mathbf{RPC}^f$ as follows:
\begin{itemize}
\item Given an object $(G,D_+,D_-)$ and an isomorphism $\ph:G\to G'$, there is a corresponding isomorphism of cones from  $\Div_{(G',D'_+,D'_-)}$ to $\Div_{(G,D_+,D_-)}$ that permutes the entries according to the permutation $\ph:E(G)\to E(G')$ of the edges. 

\item Given an object $(G,D_+,D_-)$ and an edge contraction $\ph_S:G\to G/S$, there is a face morphism from $\Div_{(G/S,D^S_+,D^S_-)}$ isomorphically onto the face of $\Div_{(G,D_+,D_-)}$ obtained by setting all coordinates corresponding to $S\subset E(G)$ to zero.
\end{itemize}

Given a stable tropical curve $\Ga$ with a pair of divisors $D_+$ and $D_-$, there is a unique minimal semistable model $G_{\Ga,D_+,D_-}$ of $\Ga$ whose vertex set contains the supports of $D_+$ and $D_-$. It is clear that the points of $\Div_{(G,D_+,D_-)}$ parameterize triples $(\Ga, D_+,D_-)$ consisting of a tropical curve $\Ga$ and two effective divisors $D_+, D_-$ of degree $d_+, d_-$ respectively, together with a choice of isomorphism of $G_{\Ga,D_+,D_-}$ with $G$.

\begin{proposition}
The association $(G,D_+,D_-)\mapsto \Div_{(G,D_+,D_-)}$ defines a generalized cone complex $\Div_{g,d_+,d_-}^{trop}\colon I_{g,d_+,d_-}\rightarrow \mathbf{RPC}^f$ whose geometric realization parameterizes stable tropical curves $\Ga$ of genus $g$ together with two effective divisors $D_+$ of degree $d_+$ and $D_-$ of degree $d_-$ on $\Ga$.
\end{proposition}

\begin{proof} 
We again leave this straightforward verification of the two axioms to the avid reader.
\end{proof}

\begin{definition}
The generalized cone complex $\Div_{g,d_+,d_-}^{trop}$ is called the \emph{moduli space of tropical divisors} of zero degree $d_+$ and polar degree $d_-$, or short of degree $(d_+,d_-)$. 
\end{definition}

Let $\epsilon=\frac{1}{d_++d_-}$. Denote by $M_{g,\epsilon^{d_+ +d_-}}^{trop}$ the moduli space of weighted stable tropical curves of type $(g,\epsilon^{d_++d_-})$ introduced in \cite{Ulirsch_tropHassett} and \cite{CavalieriHampeMarkwigRanganathan}, a tropical analogue of Hassett's moduli spaces of weighted stable curves in \cite{Hassett}. The space $M_{g,\epsilon^{d_++d_-}}^{trop}$ has the structure of a generalized cone complex and parameterizes tropical curves of genus $g$ with $d_++d_-$ marked legs that are \emph{stable of type $\epsilon^{d_++d_-}$}, i.e. that fulfill
\begin{equation*}
2g(v)-2+\vert v\vert_{\epsilon}>0
\end{equation*}
for all vertices of the minimal model $G$ of $\Ga$. Here $\vert v\vert_{\epsilon}$ denotes the weighted valence of $v$ in $\Ga$ given by 
\begin{equation*}
\vert v\vert_{\epsilon}=\vert v\vert_E+\epsilon \cdot \vert v\vert_L
\end{equation*}
where $\vert v\vert_E$ and $\vert v\vert_L$ denote respectively the number of bounded edges and the number of legs emanating from $v$. Note, that, given an $\epsilon^{d_++d_-}$-stable curve $(\Ga,p_1,\ldots, p_{d_++d_-})$, the underlying tropical curve $\Ga^\circ$ obtained by removing the legs is already stable.

\begin{proposition}\label{prop_Divgddtrop=quot}
There is a natural morphism $M_{g,\epsilon^{d_++d_-}}^{trop}\rightarrow \Div_{g,d_+,d_-}^{trop}$ given by 
\begin{equation*}
\big(\Ga, p_1,\ldots, p_{d_++d_-}\big)\longmapsto \Big(\Ga, \sum_{i=1}^{d_+} p_i, \sum_{i=1}^{d_-}p_{d_++i}\Big) 
\end{equation*}
that induces an isomorphism
\begin{equation*}
M_{g,\epsilon^{d_++d_-}}^{trop}\Big/(S_{d_+}\times S_{d_-})\xlongrightarrow{\simeq} \Div_{g,d_+,d_-}^{trop} \ .
\end{equation*}
\end{proposition}

\begin{proof} A graph $G$ with $d_++d_-$ legs is stable of type $\epsilon^{d_++d_-}$ if and only if the graph $G^\circ$ obtained by removing the legs is semistable. So the association defines a morphism of generalized cone complexes that is surjective, since we may replace every divisor $D_{\pm}$ at a vertex $v$ by $D_\pm(v)$-many legs, making the resulting legged graph stable of type $\epsilon^{d_++d_-}$. This morphism is not injective: the extra datum in $M_{g,\epsilon^{d_++d_-}}^{trop}$ is precisely the full order on the marked points, which disappears when we take a quotient by $S_{d_+}\times S_{d_-}$. 
\end{proof}


\subsection{Moduli of tropical admissible covers} \label{sec:tropadmissible}

In this section, we recall the construction (see~\cite{CavalieriMarkwigRanganathan_tropadmissiblecovers}) of the moduli space $H^{\trop}_{g\rightarrow h,d}(\mu)$ of tropical admissible covers. We fix integers $g,h\geq 0$ and $d>0$, and a vector of partitions $\mu=(\mu^1,\ldots, \mu^r)$ of $d$. We write $\vert \mu^i\vert$ for the length of the partition $\mu^i$. We consider the category $I_{g\rightarrow h, d}(\mu)$ defined as follows.

The objects of $I_{g\rightarrow h, d}(\mu)$ are morphisms $\ph\colon G'\rightarrow G$ of the following kind:

\begin{itemize}

\item $G'$ is a stable weighted graph of genus $g$ with marked legs $p'_{ij}$ for $i=1,\ldots,r$ and $j=1,\ldots,|\mu^i|$, and additional marked legs $q'_{ij}$ for $i=1,\ldots,s$ and $j=1,\ldots,d-1$. 

\item $G$ is a stable weighted graph of genus $h$ with marked legs $p_i$ for $i=1,\ldots,r$, and additional marked legs $q_i$ for $i=1,\ldots,s$.

\item $\ph$ is an unramified harmonic morphism of degree $d$, mapping the legs $p'_{ij}$ to $p_i$ with dilation profile $\mu^i$, and the legs $q'_{ij}$ to $q_i$ with dilation profile $(2,1,\ldots,1)$.

\end{itemize}
Since $\ph$ is unramified, we have $\chi(G')=d\chi(G)$, which means that $s$ is uniquely determined by the following formula:
\begin{equation}
s=2g-2dh+2d-2+\sum_{i=1}^r |\mu^i|-rd
\label{eq:explicits}
\end{equation}
The morphisms in $I_{g\rightarrow h,d}(\mu)$ are generated by the following:
\begin{itemize}
\item Isomorphisms of $\ph\colon G'\rightarrow G$, i.e. pairs of isomorphisms $\ga'\colon G'\xrightarrow{\simeq} \widetilde{G}'$ and $\ga\colon G\xrightarrow{\simeq}\widetilde{G}$ such that $\ph\circ \ga'=\ga\circ \ph$.
\item Weighted edge contractions $\ph_S:G'/S'\to G/S$ induced by the weighted contraction of the base graph $G$ along a subset $S\subset E(G)$ of its edges (see Def.~\ref{def_contractionofunramifiedharmonic} for details; it is clear that for any $\ph:G'\to G$ the multiplicity profiles of any such contracted $\ph_S$ along the legs are the same).

\end{itemize}

Given an object $\ph\colon G'\rightarrow G$, we define a rational polyhedral cone
\begin{equation*}
M_\ph=\R_{\geq 0}^{E(G)}
\end{equation*}
that parametrizes unramified harmonic morphisms of tropical curves which on the underlying graphs agrees with a weighed edge contraction of $\ph$, where the set of contracted edges $S\subset E(G)$ is the set of edges whose corresponding coordinate is zero. The association $\ph\mapsto M_\ph$ defines a functor $I_{g\rightarrow h,d}(\mu)\rightarrow \mathbf{RPC}^f$ as follows: 
\begin{itemize}
\item An automorphism of $\ph$ induces a permutation of the entries of $M_\ph=\R_{\geq 0}^{E(H)}$.
\item A weighted edge contraction induces the embedding of the face of $M_\ph=\R_{\geq 0}^{E(H)}$ corresponding to the contracted edges.
\end{itemize}

\begin{proposition}
The association $\ph\mapsto M_\ph$ defines a generalized cone complex $H^{trop}_{g\rightarrow h,d}(\mu)$ that parameterizes unramified harmonic morphisms $[\Ga'\rightarrow \Ga]$ from a genus $g$ stable tropical curve $\Ga'$ to a genus $h$ stable tropical curve $\Ga$; having dilation profile $\mu^i$ on marked legs $p'_{ij}\subset \Ga'$ over marked legs $p_i\subset \Ga$ (for $i=1,\ldots, r$ and $j=1,\ldots, \vert \mu_i\vert$) and dilation profile $(2,1,\ldots,1)$ on marked legs $q'_{ij}\subset \Ga'$ over marked legs $q_i$ (for $i=1,\ldots,s$ and $j=1,\ldots,d-1$). 

\end{proposition}

\begin{proof}
Both axioms follow from the construction of weighted edge contractions of unramified harmonic morphisms in Definition \ref{def_contractionofunramifiedharmonic} and Lemma~\ref{lemma_contractionofunramifiedharmonic}. We leave the details to the avid reader.
\end{proof}

Let $N=\sum_{i=1}^r \vert \mu^i\vert+(d-1)s$ be the number of marked legs on the source graph. As explained in \cite{CavalieriMarkwigRanganathan_tropadmissiblecovers}, there are two tautological morphisms associated with the moduli space $H^{trop}_{g\rightarrow h,d}(\mu)$:
\begin{itemize}
\item The \emph{source map} $H^{trop}_{g\rightarrow h,d}(\mu)\rightarrow M_{g,N}^{trop}$ that sends $\big[\Ga'\rightarrow\Ga\big]$ to the tropical curve $\Ga'$ with its $N$ marked legs. 
\item The \emph{target map } $H^{trop}_{g\rightarrow h,d}(\mu)\rightarrow M_{h,r+s}^{trop}$ that sends $\big[\Ga'\rightarrow\Ga\big]$ to the tropical curve $\Ga$ with its $r+s$ marked legs. 
\end{itemize}

It is readily verified that the source and target maps are morphisms of generalized cone complexes.


\section{The locus of principal divisors}

In this section, we define the locus of principal divisors $\PD_{g,a}^{trop}\subset M_{g,n}^{trop}$ of a given multiplicity profile $a\in \ZZ^n$ and prove that it is a generalized cone complex of the same dimension as $M_{g,n}^{trop}$.

\begin{definition}\label{def_PDharmonic}
Let $n\geq 2$, and let $a=(a_1,\ldots, a_n)\in\Z^n$ be nonzero integers such that $a_1+\cdots+a_n=0$. The \emph{locus of principal divisors} $\PD_{g,a}^{trop}\subseteq M_{g,n}^{trop}$ is the set of stable marked tropical curves $(\Ga, p_1,\ldots, p_n)$ such that there is a harmonic function $f$ on $\Ga$ whose (outgoing) slope along the leg $p_i$ is $-a_i$.
\end{definition}

Note that by Proposition \ref{prop_harmonicfunction=rationalfunction} a harmonic function $f$ with slopes $-a_i$ along $p_i$ exists if and only if the divisor $\sum_{i=1}^n a_i c(p_i)$ on $\Ga^{\circ}$ is principal. So Definition \ref{def_PDharmonic} is equivalent to Definition \ref{def_introPDDR} in the introduction. Whenever there is no chance of confusion we drop the superscript and simply write $\PD_{g,a}$ instead of $\PD_{g,a}^{trop}$ for the locus of principal divisors. 

Denote the analytification of the algebraic double ramification locus $\calDR_{g,a}$ by $\calDR_{g,a}^{an}$. By the slope formula \cite[Theorem 5.14]{BPR} the restriction of the tropicalization map $\trop_{g,n}\colon  \calM_{g,n}^{an}\longrightarrow M_{g,n}^{trop}$ to  $\calDR_{g,a}^{an}$ naturally factors through $\PD_{g,a}\subseteq M_{g,n}^{trop}$. This also follows from Corollary \ref{cor_tree->linearequivalence} and Proposition \ref{prop_factorization} below.

\begin{theorem}\label{thm:PDlocus}
The locus of principal divisors $\PD_{g,a}$ is a semilinear subset in $M_{g,n}^{trop}$ that contains maximal cones of every codimension between zero and $g$. 
\end{theorem}

Theorem \ref{thm:PDlocus} implies Theorem \ref{thm_DRlocus} (i) from the introduction. Its proof follows the ideas laid out in the proof of \cite[Theorem 1.2]{LinUlirsch} describing the structure of the tropical Hodge bundle, which in turn is based on the polyhedral description of tropical linear series in \cite{GathmannKerber,HaaseMusikerYu, MikhalkinZharkov}. 

\begin{proof}[Proof of Theorem \ref{thm:PDlocus}]
Let $G$ be a stable weighted graph of genus $g$ with $n$ marked legs. We need to show that $M_G^{trop}\cap \PD_{g,a}$ is an intersection of the cone $M_G^{trop}=\R_{\geq 0}^{E(G)}$ with a union of finitely many linear subspaces of $\R^{E(G)}$.

Denote by $S_a\subseteq \Z^{H(G)}$ the affine subspace of possible choices of slopes $s=(s_h)_{h\in H(G)}$ of a rational function on $G$ having outgoing slopes $a_i$ on the legs $p_i$. In other words, the numbers $s_h$ satisfy $s_h=-s_{\iota(h)}$ for every edge $e=\{h,\iota(h)\}$ of $G$, $s_{p_i}=-a_i$ for all legs $p_i$, and the harmonicity condition
\begin{equation*}
\sum_{r(h)=v} s_h = 0 
\end{equation*}
at every vertex $v\in V(G)$. The slopes on the legs are fixed, so we think of $S_a$ as a subset of $\Z^{E(G)}$.

Let $\ell:E(G)\to \RR_{\geq 0}$ be an edge length function on $G$, where we allow edge lengths to be zero by contracting the corresponding edges. A tuple $(s_h)_{h\in H(G)}\in S_a$ is represented by a function on $(G,\ell)$ if and only if the edge lengths satisfy a set of constraints. Denote by $C_1(G,\Z)=\Z^{E(G)}$ the group of simplicial 1-chains, and by  $H_1(G,\Z)\subseteq C_1(G,\Z)$ the group of 1-cycles on $G$. The metric $\ell$ defines an edge length pairing 
\begin{equation*}
\langle .,.\rangle_{\ell} \colon C_1(G,\Z)\times C_1(G,\Z)\longrightarrow \R
\end{equation*}
by the formula
\begin{equation*}
\Big\langle\sum_{e\in E(G)} a_e\cdot e,\sum_{e\in E(G)}b_e\cdot e\Big\rangle_{\ell}= \sum_{e\in E(G)} a_e b_e\cdot \ell(e) \ .
\end{equation*}
For a given $s\in S_a$, denote by $E_{s}\subset\R_{\geq 0}^{E(G)}$ the set of edge lengths on $G$ (zero edge lengths are allowed) such that there exists a rational function  on $(G,\ell)$ having slopes $s$. It is equal to
\begin{equation*}
E_{s}=\Big\{\ell\in \R_{\geq 0}^{E(G)}\Big\vert \langle s,\gamma\rangle_{\ell} = 0 \textrm{ for all } \gamma\in H_1(G,\Z) \Big\} \ .
\end{equation*}
The set $E_s\in \R_{\geq 0}^{E(G)}$ is a rational polyhedral cone, since it is the intersection of $M_G^{trop}=\R_{\geq 0}^{E(G)}$ with finitely many $\Z$-linear subspaces (corresponding to a $\Z$-basis of $H_1(G,\Z)$). Moreover, we have that 
\begin{equation*}
M_G^{trop}\cap \PD_{g,a} = \bigcup_{s\in S_G} E_{s} \ .
\end{equation*}
Therefore, we need to show that $E_s$ is equal to $\{0\}\in\R_{\geq 0}^{E(G)}$ for all but finitely many $s\in S_G$. In fact, it is enough to show that $E_s$ is a subset of a proper face of $\R_{\geq 0}^{E(G)}$ for all but finitely many $s$, since these faces correspond to $M_{G'}$ with a $G'$ a weighted edge contraction of $G$.

We proceed by induction on the genus $g$ of $G$. If $g=0$, the graph $G$ is a tree with zero vertex weights, and the slopes along the legs uniquely determine the slopes along the finite edges, in other words the entire vector $s$. So, in this case we have 
\begin{equation*}
M_G^{trop}\cap \PD_{g,a} = E_s \ .
\end{equation*}
for a unique vector $s$, proving the claim.

Suppose now that $g>0$. We first claim that we can assume that $G$ is trivalent, with all vertex weights equal to zero, and that every vertex has at most one leg rooted at it. Indeed, if $G$ does not satisfy the above properties, we can perform one of the following operations to it:
\begin{itemize}
\item Replace a vertex $v$ of positive weight $g(v)$ with a genus zero vertex having $g(v)$ loops.
 
\item Split a vertex of valence $\val(v)>3$ into two vertices of valence at least 3, connected by an edge. 

\item If there is a trivalent vertex having two legs and an edge, replace the two legs and the edge with a single leg, and the slopes with the sum of the slopes (in other words, there is a secondary induction on $n$). 

\end{itemize}
Repeating this operation, we obtain a trivalent graph having each leg rooted at a unique vertex. By harmonicity, a choice of slopes on the original graph uniquely determines a choice of slopes on the resulting graph, so it is sufficient to prove our theorem for graphs satisfying the above conditions. 

We claim that, whenever $E_s$ is not a subset of a proper face of $\R_{\geq 0}^{E(G)}$, there is a leg $p$ rooted at a vertex $v$ such that the outgoing slopes along the two edges $e_1$ and $e_2$ rooted at $v$ have the same sign. Indeed, suppose that that there is no leg with this property. Orient each edge of $G$ so that $s$ has positive slope with respect to the orientation. At every vertex $v\in V(G)$ with a leg, one of the two edges is incoming, while the other is outgoing. Therefore the edge orientation on $G$ descends to an edge orientation on the graph obtained from $G$ by removing the legs. By the harmonicity condition on $G$ there no vertex with all negative or all positive incoming slopes, hence this orientation is not acyclic. Choose a positively oriented path $\ga$, then any rational function with slopes given by $s$ has positive slopes along $\ga$. The only way that this is possible is if the lengths of all edges in $\ga$ are zero, so $E_s$ is a subset of a proper face of $\R_{\geq 0}^{E(G)}$.

Now let $p_i$ be a leg rooted at a vertex $v$, so that the outgoing slopes $s_e$ and $s_{e'}$ of the two edges $e$ and $e'$ rooted at $v$ have the same sign. We now consider the graph $\widetilde{G}$ obtained from $G$ as follows: remove $p$ and $v$, and replace $e$ and $e'$ with two new legs $l$ and $l'$. There are two possibilities. First, it may happen that $\widetilde{G}$ has two connected components, each of which necessarily has genus greater than zero, and therefore less than $g$. The slopes $s_e$ and $s_{e'}$ on the legs $l$ and $l'$ are then uniquely determined by the slopes of the other legs on each component, since the slope of all legs on a single component add to zero. By induction, there are finitely many ways to choose the slope vectors on each component, and hence on all of $G$. On the other hand, if $\widetilde{G}$ is connected, then it has genus $g-1$. The slopes $s_e$ and $s_{e'}$ along the legs $l$ and $l$ have the same sign and add up to $a_i$, so there are finitely many possibilities for them. We now apply induction for each such choice and complete the proof.

The statement concerning the existence of a cone of every codimension between $0$ and $g$ follows from the following two examples. \end{proof}

\begin{example}[Bridges]

Let $G^{\circ}$ be a maximally degenerate stable graph of genus $g$ with no legs and a bridge edge $e$. Split $e$ into $n+1$ bridges at vertices $v_1,\ldots,v_n$, labelled left to right, and let $G$ be the stable graph formed by attaching legs $p_1,\ldots,p_n$ at $v_1,\ldots,v_n$. For any choice of edge length $\ell\colon E(G)\rightarrow \R_{>0}$, there exists a harmonic function $f$ on $\Gamma=(G,\ell)$ such that its slope along the leg $l_i$ is $-a_i$. Indeed, $f$ is constant on the part of $\Ga$ to the left of $v_1$, has slope $a_1$ on the edge connecting $v_1$ and $v_2$, slope $a_1+a_2$ on the edge connecting $v_2$ and $v_3$, and so on, and then again constant to the right of $v_n$. Therefore the whole cell $M_G^{trop}=\R_{\geq 0}^{E(G)}$ of dimension $3g-3+n$ is part of the principal divisor locus. This shows that $\PD_{g,a}$ always has a cell of codimension zero.
\end{example}

\begin{example}[Chain of loops]
Let $G^{\circ}$ be a chain of $g$ loops, and let $1\leq k\leq g$. Pick a bridge $e\in E(G^{\circ})$ and edges $f_1,\ldots,f_k\in E(G^{\circ})$ lying on distinct loops. Similarly to the construction above, let $G$ be the stable graph obtained by attaching $k$ legs $p_1,\ldots,p_k$ to vertices on $f_1,\ldots,f_k$, and the remaining legs $p_{k+1},\ldots,p_n$ to vertices on $e$. Then the metric graph $(G,\ell)$ lies in $PD_{g,a}$ if and only if the lengths of the subdivisions of $f_i$ satisfy $k$ independent constraints, while the legs on the bridge impose no constraints. This shows that $\PD_{g,a}$ has a maximal cone of codimension $k$ in $M_{g,n}^{trop}$ for all $1\leq k\leq g$.
\end{example}






\section{The tropical double ramification locus}

As we have noted in the introduction, a basic problem in tropical geometry is that "tropical curves have too many principal divisors". Specifically, we saw in Theorem~\ref{thm:PDlocus} that, given a multiplicity profile $a=(a_1,\ldots,a_n)$ with $a_1+\cdots+a_n=0$, the locus $PD_{g,a}\subset M_{g,n}^{trop}$ of principal divisors of profile $a$ has cells of top dimension $3g-3+n$. On the other hand, by a toroidal version of the Bieri--Groves theorem \cite[Theorem 1.1]{Ulirsch_tropcomplogreg}, the tropicalization of the double ramification locus 
$\calDR_{g,a}\subset\calM_{g,n}$ is a semilinear subset of $M_{g,n}^{trop}$ of dimension at most $2g-3$. For this reason, tropical linear series are typically much larger in dimension than expected, and, in particular, have larger dimensions than the tropicalizations of the corresponding algebraic linear series.

In this section we provide a solution to this problem by defining a {\it tropical double ramification locus} $DR_{g,a}\subset M_{g,n}^{trop}$, which contains (and is, in general, strictly greater than) the tropicalization of $\calDR_{g,a}$, but has the correct dimension. A marked curve $(X,p_1,\ldots,p_n)\in \calM_{g,n}$ lies in $\calDR_{g,a}$ if it admits a map $X\to \PP^1$ with prescribed ramification profiles over $0$ and $\infty$, so we define $DR_{g,a}$ as the set of tropical marked curves admitting a morphism to a tree $\Delta$ with specified degrees over two legs $0$ and $\infty$ of $\Delta$. Borrowing an idea from Cools and Draisma~\cite{CoolsDraisma}, we require that the morphism be finite and effective, and we allow tropical modifications on the source curve $\Ga$; this turns out to be necessary to produce a locus of the correct dimension. We then prove Theorem~\ref{thm:main}, stating that the tropical double ramification locus $DR_{g,a}$ is a semilinear subset of $\PD_{g,a}\subset M_{g,n}^{trop}$ of dimension $2g-3+n$, which is the main result of our paper.

\begin{definition} Let $g\geq 1$, $n\geq 2$, and let $a=(a_1,\ldots,a_n)\in \ZZ^n$ be an $n$-tuple of nonzero integers such that $a_1+\cdots+a_n=0$. We define the {\it double ramification locus} $DR_{g,a}\subset M_{g,n}^{trop}$ to be the set of stable tropical curves $(\Ga,p_1,\ldots,p_n)$ such that there exists a finite effective morphism $\tau:\Ga'\to \De$ of the following kind:

\begin{enumerate} \item $\Ga'$ is a tropical modification of $\Ga$.

\item $\De$ is a marked tree of genus 0 having two legs $0$ and $\infty$.

\item $\tau$ acts in the following way on the legs:

\begin{equation}
\tau(p_i)=\left\{\begin{array}{cc} 0, & \textrm{ if } a_i>0\ \\ \infty, & \textrm{ if } a_i<0\end{array}\right.\quad \textrm{ and }\quad d_{\tau}(p_i)=|a_i|. 
\label{eq:DRcondition}
\end{equation}

\end{enumerate}

We call the integer $d=\displaystyle\sum_{i:a_i>0}a_i=-\displaystyle\sum_{i:a_i<0}a_i$ the {\it degree} of the double ramification locus $DR_{g,a}$; it is equal to the degree of any such map $\tau$. \label{def:DR}
\end{definition}

The following Theorem \ref{thm_structureDRlocus} is Theorem \ref{thm_DRlocus} (b) from the introduction. 

\begin{theorem}\label{thm_structureDRlocus}
Suppose that $d\geq 2$. The double ramification locus $DR_{g,a}$ is a semilinear subset of $\PD_{g,a}^{trop}$ of dimension $2g-3+n$.
\label{thm:main}
\end{theorem}

Before we proceed with the proof, we explain why all of the conditions in the theorem are necessary.

\begin{enumerate}

\item If we drop the requirement that the morphism $\tau$ be finite, then we will obtain the principal divisor locus $PD_{g,a}$. Indeed, a piecewise-linear function $f:\Ga\to \RR$ can be viewed as a harmonic, not necessarily finite, morphism to the tree $\RR$, and condition~\ref{eq:DRcondition} then implies that the divisor $\sum a_i p_i$ is principal. Hence the condition that $\tau$ be finite is necessary to obtain a locus of dimension $2g-3+n$.

\item A somewhat more subtle observation is that if we relax the requirement that $\tau$ be effective, then the resulting double ramification locus will also be of dimension greater than $2g-3+n$. As an example, let $\Ga$ be one of the stable genus three hyperelliptic curves mentioned in Remark~\ref{rem:genus3hyperelliptic}, and let $\ph:\Ga\to \De$ be the hyperelliptic morphism. Attaching two legs $p_1$ and $p_2$ to $\Ga$ at any two points in the interior bold subgraph, and corresponding legs $0$ and $\infty$ to $\De$, we obtain a cone in the double ramification locus $DR_{3,(2,-2)}$ of dimension 8, in other words, of codimension zero in $M_{3,2}^{trop}$.

\item We also observe that if $\Ga'\supset \Ga$ is a tropical modification of $\Ga$ and $\tau:\Ga'\to \De$ is a finite effective harmonic morphism as above, then the restriction of $\tau$ to $\Ga$ need not be a harmonic morphism. Therefore, the locus in $M_{g,n}^{trop}$ consisting of stable curves $\Ga$ admitting a morphism $\tau:\Ga\to \De$ of the type described above is strictly smaller than $DR_{g,a}$, and will not in general contain the tropicalization of the algebraic double ramification locus. For example, let $g=1$ and $a=(2,1,-2,-1)$, and consider the following morphism $\tau:\Ga'\to \De$ ($\tau$ has degree $2$ on the thick edges and degree $1$ elsewhere):
\begin{center}
\begin{tikzpicture}

\draw [thick] (0,0) .. controls (0,0.5) and (2,0.5) .. (2,0);
\draw [thick] (0,0) .. controls (0,-0.5) and (2,-0.5) .. (2,0);
\draw[ultra thick] (2,0) -- (3.5,0);
\draw[fill] (2,0) circle(.8mm);
\draw[fill] (3.5,0) circle(.8mm);
\draw[fill] (5,0) circle(.8mm);
\draw[fill] (5,-1) circle(.8mm);
\draw[fill] (5,-2) circle(.8mm);

\draw[thick] (3.5,0) -- (5,0);
\draw[thick] (3.5,0) -- (5,-1);
\draw[thick] (3.5,-1) -- (5,0);
\draw[thick] (0,-1) -- (3.5,-1);
\draw[ultra thick] (5,0) -- (5.8,0.25) node[right] {$p_1$};
\draw[ultra thick] (5,0) -- (5.8,-0.25) node[right] {$p_3$};
\draw[thick] (5,-1) -- (5.8,-0.75) node[right] {$p_2$};
\draw[thick] (5,-1) -- (5.8,-1.25) node[right] {$p_4$};

\draw[thick] (0,-2) -- (5,-2);
\draw[thick] (5,-2) -- (5.8,-1.75) node[right] {$0$};
\draw[thick] (5,-2) -- (5.8,-2.25) node[right] {$\infty$};

\end{tikzpicture} 

\end{center}
Removing the unstable edge of $\Ga'$ would violate the harmonicity condition at its root vertex. The morphism $\tau$ has degree three, and all Hurwitz numbers of degree three are nonzero, hence $\tau$ is algebraizable and the stabilization of $\Ga$ is contained in the tropicalization of the algebraic double ramification locus. 

We will later see in Proposition~\ref{prop:hyperellipticrestriction} that for $d=2$, any morphism $\tau:\Ga'\to \De$ naturally restricts to the stabilization $\Ga$ of $\Ga'$.

\end{enumerate}

The proof of Theorem~\ref{thm_structureDRlocus} is given in Section~\ref{sec:mainproof}. First, we prove Lemma~\ref{lem:edgecontraction} on reversing edge contractions, this is the content of Section~\ref{sec:edgecontraction}. The proof itself consists of two parts. First, we give an algorithm, explicit but in general computationally intractable, for describing $DR_{g,a}$. The algorithm also establishes the dimension bound $\dim DR_{g,a}\leq 2g-3+n$. The second part of the proof consists of constructing a cone of $DR_{g,a}$ of the maximal dimension $2g-3+n$. This construction uses the inductive procedure of grafting a tree and is borrowed from~\cite{CoolsDraisma}. 

We conclude with a conjecture on the topological structure of the double ramification locus.

\begin{conjecture} The double ramification locus $DR_{g,a}$ is connected in codimension one and all of its maximal cones have dimension $2g-3+n$.

\end{conjecture}

This conjecture is supported by two extended examples: in Section~\ref{sec:example} we compute $DR_{1,(a,-a)}\subset M_{1,2}^{trop}$ for all $a\geq 2$, and in Section~\ref{sec:hyperelliptic} we compute the hyperelliptic loci $DR_{g,(2,-2)}\subset M_{g,2}^{trop}$, $DR_{2,(2,-1,-1)}\subset M_{g,3}^{trop}$, and $DR_{2,(1,1,-1,-1)}\subset M_{g,4}^{trop}$.

\subsection{A lemma on reversing edge contractions} \label{sec:edgecontraction}

Before proving the main theorem, we first prove a result about reversing contractions of unramified harmonic morphisms (see Definition~\ref{def:unramifiedcontraction}).

\begin{lemma} Let $\ph:G'\to G$ be an unramified harmonic morphism of weighted graphs, let $K$ be a weighted graph, and let $S\subset E(K)$ be a subset of edges of $K$ such that $K/S$ is isomorphic to $G$. Then there exists a weighted graph $K'$ and an unramified harmonic morphism $\psi:K'\to K$ such that the contraction $\psi_S$ of $\psi$ along $S$ is isomorphic to $\ph$. 
\label{lem:edgecontraction}
\end{lemma}

\begin{remark} If the morphism $\ph$ is realizable, then the result follows from the properness of the target map on the moduli space of admissible covers. Indeed, let $f:X'\to X$ be an admissible cover of stable algebraic curves over a non-Archimedean field $K$, such that $G'$ and $G$ are respectively the dual graphs of $X'$ and $X$ and the induced map on the dual graphs is $\ph$. Further, assume that the graph $K$ is the dual graph of a stable algebraic curve $Y$ obtained by degenerating $X$, with the edges $S$ corresponding to the nodes of $Y$ that are smoothed in $X$. By properness, the admissible cover $f$ extends to an admissible cover $g:Y'\to Y$, and it is clear that $\ph$ is the contraction of the tropicalization $\psi:K'\to K$ of $g$ along $S$. In other words, we can interpret the above result as saying that "the target map of the moduli space of tropical unramified covers is proper". 

\end{remark}

\begin{proof} To avoid talking about isomorphisms, we identify $K/S$ with $G$. Factoring into edge contractions, we can assume that $S$ consists of a single edge $e\in E(K)$.

We first consider the simpler case when $e$ is a loop at a vertex $v\in V(K)$, in which case $V(K)$ is identified with $V(G)$. Let $\ph^{-1}(v)=\{v'_1,\ldots,v'_m\}$, then $g_{G}(v)=g_K(v)+1\geq 1$, hence for each $i$ $\Ram_{\ph}(v'_i)=0$ implies that
$$
g_{G'}(v'_i)=d_{\ph}(v'_i)(g_G(v)-1)+d_{\ph}(v'_i)\val(v)-\val(v'_i)+1\geq 1.
$$
We therefore construct $K'$ from $G'$ by attaching a loop $e'_i$ at each vertex $v'_i$ and reducing its genus by one. We define $\psi$ by $\psi(e'_i)=e$ with $d_{\psi}(e'_i)=d_{\ph}(v'_i)$, . We note that $\chi_K(v)=\chi_G(v)$ and $\chi_{K'}(v'_i)=\chi_{G'}(v)$, therefore $\psi$ is unramified at $v'_i$. It follows that $\psi$ is unramified, and it is clear that $\psi_S$ is isomorphic to $\ph$.

We now suppose that $e=\{e_v,e_w\}$ is rooted at distinct vertices $v,w\in V(K)$. Let $u\in V(G)$ be the vertex of genus $g(v)+g(w)$ obtained from merging $v$ and $w$.
We construct $K'$ from $G'$ in the following way. Pick a vertex $u'\in V(G')$ such that $\ph(u')=u$. We replace $u'$ with an appropriate bipartite subgraph $L'$ that $\psi$ maps to the subgraph $L=\{v,w,e\}\subset K$. The subgraph $L'$ consists of vertices $v'_i$ and $w'_j$ mapping to $v$ and $w$, respectively, and a number of edges that are all mapped to $e$. Every half-edge $h\in T_uG=(T_vK\cup T_wK)\backslash\{e_v,e_w\}$ is rooted either at $v$ or at $w$, and we correspondingly need to root the half-edges $\ph^{-1}(h)$ either at the $v'_i$ or at the $w'_j$, in such a way that the harmonicity condition~\eqref{eq:harmonic} is satisfied. Finally, we need to define the local degrees of $\psi$ at the edges of $L'$ so that harmonicity is preserved, and assign genera on $V(L')$ by formula~\eqref{eq:admissiblegenus}, making sure that these are non-negative integers. This operation is performed independently at each $u'\in \ph^{-1}(u)$.

To construct the subgraph $L'$ for a given $u'\in \ph^{-1}(u)$, we look at the degrees and genera of the vertices $u'$, $v$, and $w$. It turns out that there are three possibilities. We may attempt to choose $L'$ to be the simplest subgraph possible, consisting of two vertices $v'$ and $w'$ connected by a single edge $e'$ (case~(\ref{item:extensioncase1}) below). The harmonicity condition is then trivially satisfied, but the genera $g(v')$ and $g(w')$ given by equation~\eqref{eq:admissiblegenus} may be half-integers, and one of them, say $g(v')$, may be negative. The former case can be resolved by adding a second edge to $L'$ (case~(\ref{item:extensioncase2})), but the latter (case~(\ref{item:extensioncase3})) requires splitting $v'$ into several vertices of genus zero, each connected by a single edge to $w'$.

For each of the three cases, we give an example. In all three cases, the vertex $u$ has genus one and valency four, while $g(v)=0$, $g(w)=1$, and $\val(v)=\val(w)=3$:

\begin{center}
\begin{tikzpicture}

\begin{scope}[shift={(0,0)}]

\draw [thick] (0,0) -- (-0.8,0.4) node[left]{$f_1$};
\draw [thick] (0,0) -- (-0.8,-0.4) node[left]{$f_2$};
\draw [thick] (0,0) -- (0.8,0.4) node[right]{$h_1$};
\draw [thick] (0,0) -- (0.8,-0.4) node[right]{$h_2$};
\draw[fill,white](0,0) circle(2mm);
\draw[thick](0,0) circle(2mm);
\draw (0,0) node{$1$};
\draw(0,-0.2) node[below]{$u$};

\end{scope}

\begin{scope}[shift={(3,0)}]

\draw [thick] (0,0) -- (-0.8,0.4) node[left]{$f_1$};
\draw [thick] (0,0) -- (-0.8,-0.4) node[left]{$f_2$};
\draw[fill](0,0) circle(.8mm);
\draw(0,-0.2) node[below]{$v$};

\draw [ultra thick] (0,0) -- (1.8,0) node[midway,above]{$e$};

\begin{scope}[shift={(2,0)}]
\draw [thick] (0,0) -- (0.8,0.4) node[right]{$h_1$};
\draw [thick] (0,0) -- (0.8,-0.4) node[right]{$h_2$};
\draw[fill,white](0,0) circle(2mm);
\draw[thick](0,0) circle(2mm);
\draw (0,0) node{$1$};
\draw(0,-0.2) node[below]{$w$};
\end{scope}

\end{scope}

\end{tikzpicture}
\end{center}
We label each edge and half-edge of $O(u')$ and $O(L')$ by its degree if it is greater than one. The contracted edge $e$ and the corresponding contracted edges in $L'$ are shown thick for emphasis.

Consider the tangent spaces at $v$, $w$ and $u$:
$$
T_vK=\{e_v,f_1,\ldots,f_m\},\quad T_wK=\{e_w,h_1,\ldots,h_n\},\quad \textrm{ and }\quad T_uG=\{f_1,\ldots,f_m,h_1,\ldots,h_n\}.
$$
Denote
$$
d=d_{\ph}(u'),\quad k=\#( \ph^{-1}\{f_1,\ldots,f_m\}),\quad \textrm{ and }\quad l=\# (\ph^{-1}\{h_1,\ldots,h_n\}).
$$
The condition that $\ph$ is unramified at $u'$ reads
\begin{equation}
2-2g(u')-k-l=\chi(u')=d\chi(u)=d\big[\chi(v)+\chi(w)\big],
\label{eq:localRHu'}
\end{equation}
which we rewrite as
$$
\big[k-1+d\chi(v)\big]+\big[l-1+d\chi(w)\big]=-2g(u').
$$
The two integers $k-1+d\chi(v)$ and $l-1+d\chi(w)$ sum to a non-positive even integer, so there are three possibilities:

\begin{enumerate}

\item \label{item:extensioncase1} $k-1+d\chi(v)=-2x$, $l-1+d\chi(w)=-2y$ for some non-negative integers $x$ and $y$. In this case we let $L'$ be the graph consisting of two vertices $v'$ and $w'$ of genera $g(v')=x$ and $g(w')=y$, joined by one edge $e'$. We glue $L'$ to $K$ in place of $u$ by attaching the half-edges $\ph^{-1}\{f_1,\ldots,f_m\}$ to $v'$ and the $\ph^{-1}\{h_1,\ldots,h_n\}$ to $w'$, and define $\psi$ on $L'$ by
$$
\psi(v')=v,\quad \psi(w')=w,\quad \psi(e')=e,\quad  \textrm{ and }\quad d_{\psi}(v')=d_{\psi}(w')=d_{\psi}(e')=d.
$$
We have $\chi(v')=1-2x-k$ and $\chi(w')=1-2y-l$, so by the definition of $x$ and $y$ the morphism $\psi$ is unramified at $v'$ and $w'$.

An example of this case, with $k=2$, $l=4$, and $d=3$, is given below.

\begin{center}
\begin{tikzpicture}

\begin{scope}[shift={(0,0)}]

\draw [thick] (0,0) -- (-0.8,0.4) node[left]{$f_1$};
\draw [thick] (0,0) -- (-0.8,-0.4) node[left]{$f_2$};
\draw [thick] (0,0) -- (0.8,0.4) node[right]{$h_1$};
\draw [thick] (0,0) -- (0.8,-0.4) node[right]{$h_2$};
\draw[fill,white](0,0) circle(2mm);
\draw[thick](0,0) circle(2mm);
\draw (0,0) node{$1$};
\draw(0,-0.2) node[below]{$u$};

\end{scope}

\begin{scope}[shift={(3,0)}]

\draw [thick] (0,0) -- (-0.8,0.4) node[left]{$f_1$};
\draw [thick] (0,0) -- (-0.8,-0.4) node[left]{$f_2$};
\draw[fill](0,0) circle(.8mm);
\draw(0,-0.2) node[below]{$v$};

\draw [ultra thick] (0,0) -- (1.8,0) node[midway,above]{$e$};

\begin{scope}[shift={(2,0)}]
\draw [thick] (0,0) -- (0.8,0.4) node[right]{$h_1$};
\draw [thick] (0,0) -- (0.8,-0.4) node[right]{$h_2$};
\draw[fill,white](0,0) circle(2mm);
\draw[thick](0,0) circle(2mm);
\draw (0,0) node{$1$};
\draw(0,-0.2) node[below]{$w$};
\end{scope}

\end{scope}

\begin{scope}[shift={(0,2)}]

\draw [thick] (0,0) -- (-0.8,0.4) node[left]{$3$};
\draw [thick] (0,0) -- (-0.8,-0.4) node[left]{$3$};
\draw [thick] (0,0) -- (0.8,0.6) node[right]{$2$};
\draw [thick] (0,0) -- (0.8,0.2);
\draw [thick] (0,0) -- (0.8,-0.2) node[right]{$2$};
\draw [thick] (0,0) -- (0.8,-0.6);
\draw[fill,white](0,0) circle(2mm);
\draw[thick](0,0) circle(2mm);
\draw (0,0) node{$4$};
\draw(0,-0.2) node[below]{$u'$};

\end{scope}

\begin{scope}[shift={(3,2)}]

\draw [thick] (0,0) -- (-0.8,0.4) node[left]{$3$};
\draw [thick] (0,0) -- (-0.8,-0.4) node[left]{$3$};
\draw [ultra thick] (0,0) -- (2,0) node[midway,above]{$3$};
\draw[fill,white](0,0) circle(2mm);
\draw[thick](0,0) circle(2mm);
\draw (0,0) node{$1$};
\draw(0,-0.2) node[below]{$v'$};

\begin{scope}[shift={(2,0)}]
\draw [thick] (0,0) -- (0.8,0.6) node[right]{$2$};
\draw [thick] (0,0) -- (0.8,0.2);
\draw [thick] (0,0) -- (0.8,-0.2) node[right]{$2$};
\draw [thick] (0,0) -- (0.8,-0.6);
\draw[fill,white](0,0) circle(2mm);
\draw[thick](0,0) circle(2mm);
\draw (0,0) node{$3$};
\draw(0,-0.2) node[below]{$w'$};
\end{scope}

\end{scope}

\end{tikzpicture}
\end{center}

\item \label{item:extensioncase2} 
$k-1+d\chi(v)=-2x-1$, $l-1+d\chi(w)=-2y-1$ for some non-negative integers $x$ and $y$. This case is similar to the one above. We let $L'$ be the graph consisting of two vertices $v'$ and $w'$ of genera $g(v')=x$ and $g(w')=y$, joined by two edges $e'_1$ and $e'_2$. We attach the half-edges $\ph^{-1}\{f_1,\ldots,f_m\}$ to $v'$ and the $\ph^{-1}\{h_1,\ldots,h_n\}$ to $w'$, and define $\psi$ on $L'$ by
$$
\psi(v')=v,\quad \psi(w')=w,\quad \psi(e'_i)=e, \quad d_{\psi}(v')=d_{\psi}(w')=d, \quad d_{\psi}(e'_i)=d_i,
$$
where $d_1$ and $d_2$ are any two positive integers summing to $d$. We have $\chi(v')=-2x-k$ and $\chi(w')=-2y-l$, and again $\psi$ is unramified at $v'$ and $w'$ by the definition of $x$ and $y$.

An example of this case, with $k=3$, $l=3$, and $d=3$, is given below.

\begin{center}
\begin{tikzpicture}

\begin{scope}[shift={(0,0)}]

\draw [thick] (0,0) -- (-0.8,0.4) node[left]{$f_1$};
\draw [thick] (0,0) -- (-0.8,-0.4) node[left]{$f_2$};
\draw [thick] (0,0) -- (0.8,0.4) node[right]{$h_1$};
\draw [thick] (0,0) -- (0.8,-0.4) node[right]{$h_2$};
\draw[fill,white](0,0) circle(2mm);
\draw[thick](0,0) circle(2mm);
\draw (0,0) node{$1$};
\draw(0,-0.2) node[below]{$u$};

\end{scope}

\begin{scope}[shift={(3,0)}]

\draw [thick] (0,0) -- (-0.8,0.4) node[left]{$f_1$};
\draw [thick] (0,0) -- (-0.8,-0.4) node[left]{$f_2$};
\draw[fill](0,0) circle(.8mm);
\draw(0,-0.2) node[below]{$v$};
\draw [ultra thick] (0,0) -- (2,0) node[midway,above]{$e$};

\begin{scope}[shift={(2,0)}]
\draw [thick] (0,0) -- (0.8,0.4) node[right]{$h_1$};
\draw [thick] (0,0) -- (0.8,-0.4) node[right]{$h_2$};
\draw[fill,white](0,0) circle(2mm);
\draw[thick](0,0) circle(2mm);
\draw (0,0) node{$1$};
\draw(0,-0.2) node[below]{$w$};
\end{scope}

\end{scope}

\begin{scope}[shift={(0,2)}]

\draw [thick] (0,0) -- (-0.8,0.6) node[left]{$2$};
\draw [thick] (0,0) -- (-0.8,0.2);
\draw [thick] (0,0) -- (-0.8,-0.4) node[left]{$3$};
\draw [thick] (0,0) -- (0.8,0.6) node[right]{$2$};
\draw [thick] (0,0) -- (0.8,0.2);
\draw [thick] (0,0) -- (0.8,-0.4) node[right]{$3$};
\draw[fill,white](0,0) circle(2mm);
\draw[thick](0,0) circle(2mm);
\draw (0,0) node{$4$};
\draw(0,-0.2) node[below]{$u'$};

\end{scope}

\begin{scope}[shift={(3,2)}]

\draw[fill](0,0) circle(.8mm);
\draw(0,-0.2) node[below]{$v'$};
\draw [thick] (0,0) -- (-0.8,0.6) node[left]{$2$};
\draw [thick] (0,0) -- (-0.8,0.2);
\draw [thick] (0,0) -- (-0.8,-0.4) node[left]{$3$};
\draw [ultra thick] (0,0) .. controls (0,0.5) and (2,0.5) .. (2,0) node[midway,above] {$2$} ;
\draw [ultra thick] (0,0) .. controls (0,-0.5) and (2,-0.5) .. (2,0);

\begin{scope}[shift={(2,0)}]
\draw [thick] (0,0) -- (0.8,0.6) node[right]{$2$};
\draw [thick] (0,0) -- (0.8,0.2);
\draw [thick] (0,0) -- (0.8,-0.4) node[right]{$3$};
\draw[fill,white](0,0) circle(2mm);
\draw[thick](0,0) circle(2mm);
\draw (0,0) node{$3$};
\draw(0,-0.2) node[below]{$w'$};
\end{scope}

\end{scope}

\end{tikzpicture}
\end{center}

\item \label{item:extensioncase3} 
One of the integers $k-1+d\chi(v)$ and $l-1+d\chi(w)$ is positive, so we assume without loss of generality that $t=k+d\chi(v)\geq 2$. In this case we cannot replace $u'$ with a one-edge graph $L'$ as in cases~(\ref{item:extensioncase1} )and~(\ref{item:extensioncase2}) above, as that would require assigning a negative genus to the vertex $v'$ over $v$. Instead, the vertex $v$ will have $t$ preimages, and we need to split up the preimages of the $f_i$ among them in a way that satisfies the harmonicity condition \eqref{eq:harmonic}. This requires a statement about partitions and their refinements (Lemma~\ref{lemma:partition} below).

First, a few more definitions. Let $d$ be a positive integer. A {\it partition} $a=(a_1,\ldots,a_k)$ of $d$ is a multiset of positive integers adding up to $d$. The {\it length} $|a|$ of $a$ is $k$. Let $a=(a_1,\ldots,a_k)$ and $b=(b_1,\ldots,b_l)$ be two partitions of $d$. A {\it refinement of $a$ by $b$} is a map of multisets $r:b\to a$ such that for each $i$, $r^{-1}(a_i)$ is a partition of $a_i$. 

Denote
$$
\{f'_{i1},\ldots,f'_{ik_i}\}=\ph^{-1}(f_i),\quad a_{ip}=d_{\ph}(f'_{ip}),\quad \textrm{ and }\quad i=1,\ldots,m,
$$
then by the harmonicity condition \eqref{eq:harmonic} each $a_i=(a_{i1},\ldots,a_{ik_i})$ is a partition of $d$, while $(k_1,\ldots,k_m)$ is a partition of $k$. We now observe that $k_i\leq d$ for all $i$, hence $k\leq dm$ and therefore $t\leq d(1-2g(v))$. This implies that $g(v)=0$, $g(w)=g(u)$ and
$$
t=k-d(m-1)=k_1+\cdots+k_m-d(m-1).
$$
Therefore by Lemma~\ref{lemma:partition} there exists a partition $c=(c_1,\ldots,c_t)$ of $d$ and refinements $r_i:a_i\to c$ such that for each $j=1,\ldots,t$, the total length of the partitions $r_1^{-1}(c_j),\ldots,r_m^{-1}(c_j)$ is equal to $c_j(m-1)+1$. 

We now let $L'$ be the graph having $t$ vertices $v'_1,\ldots,v'_t$ of genus $0$, each connected by a single edge $e'_i$ to a vertex $w'$ of genus $g(u')$. For each $i=1,\ldots,m$, attach those of the $e'_{ip}$ to $v'_j$ whose corresponding degrees are in $r_i^{-1}(c_j)$, and attach all $\ph^{-1}(h_1,\ldots,h_n)$ to $w'$. We define $\psi$ on $L'$ by
$$
\psi(v'_j)=v,\quad \psi(w')=w,\quad \psi(e'_i)=e,\quad d_{\psi}(v'_j)=d_{\psi}(e'_j)=c_j,\quad d_{\psi}(w')=d.
$$
For each $j=1,\ldots,t$, we have
$$
g(v'_j)=0,\quad \val(v'_j)=\sum_{i=1}^{m}|r_i^{-1}(c_j)|+1=c_j(m-1)+2,\quad \textrm{ and }\quad \chi(v'_j)=c_j(1-m),
$$
while $g(v)=0$ and $\chi(v)=1-m$. Hence the morphism $\psi$ is unramified at $v'_j$. Similarly,
$$
g(w')=g(u'),\quad \val(w')=l+t,\quad \textrm{ and }\quad\chi(w')=2-2g(u')-l-t
$$
while $g(w)=g(u)$ and $\chi(w)=1-2g(u)-n$. Equation~\eqref{eq:localRHu'} and the definition of $t$ imply that $\psi$ is unramified at $w'$. This completes the proof.

An example is given below with $k=5$, $l=3$, and $d=3$.

\begin{center}
\begin{tikzpicture}

\begin{scope}[shift={(0,0)}]

\draw [thick] (0,0) -- (-0.8,0.4) node[left]{$f_1$};
\draw [thick] (0,0) -- (-0.8,-0.4) node[left]{$f_2$};
\draw [thick] (0,0) -- (0.8,0.4) node[right]{$h_1$};
\draw [thick] (0,0) -- (0.8,-0.4) node[right]{$h_2$};
\draw[fill,white](0,0) circle(2mm);
\draw[thick](0,0) circle(2mm);
\draw (0,0) node{$1$};
\draw(0,-0.2) node[below]{$u$};

\end{scope}

\begin{scope}[shift={(3,0)}]

\draw [thick] (0,0) -- (-0.8,0.4) node[left]{$f_1$};
\draw [thick] (0,0) -- (-0.8,-0.4) node[left]{$f_2$};
\draw[fill](0,0) circle(.8mm);
\draw(0,-0.2) node[below]{$v$};
\draw [ultra thick] (0,0) -- (2,0) node[midway,above]{$e$};

\begin{scope}[shift={(2,0)}]
\draw [thick] (0,0) -- (0.8,0.4) node[right]{$h_1$};
\draw [thick] (0,0) -- (0.8,-0.4) node[right]{$h_2$};
\draw[fill,white](0,0) circle(2mm);
\draw[thick](0,0) circle(2mm);
\draw (0,0) node{$1$};
\draw(0,-0.2) node[below]{$w$};
\end{scope}

\end{scope}

\begin{scope}[shift={(0,2)}]

\draw [thick] (0,0) -- (-0.8,0.8);
\draw [thick] (0,0) -- (-0.8,0.5);
\draw [thick] (0,0) -- (-0.8,0.2);
\draw [thick] (0,0) -- (-0.8,-0.2) node[left]{$2$};
\draw [thick] (0,0) -- (-0.8,-0.5);
\draw [thick] (0,0) -- (0.8,0.6) node[right]{$2$};
\draw [thick] (0,0) -- (0.8,0.2);
\draw [thick] (0,0) -- (0.8,-0.4) node[right]{$3$};
\draw[fill,white](0,0) circle(2mm);
\draw[thick](0,0) circle(2mm);
\draw (0,0) node{$3$};
\draw(0,-0.2) node[below]{$u'$};

\end{scope}

\begin{scope}[shift={(3,2)}]

\begin{scope}[shift={(0,0.7)}]
\draw [thick] (0,0) -- (-0.8,0.6);
\draw [thick] (0,0) -- (-0.8,0.2);
\draw [thick] (0,0) -- (-0.8,-0.4) node[left]{$2$};
\draw[fill](0,0) circle(.8mm);
\draw(0,-0.2) node[below]{$v'_1$};
\end{scope}

\begin{scope}[shift={(0,-0.7)}]
\draw [thick] (0,0) -- (-0.8,0.4);
\draw [thick] (0,0) -- (-0.8,-0.4);
\draw[fill](0,0) circle(.8mm);
\draw(0,-0.2) node[below]{$v'_2$};
\end{scope}

\draw [ultra thick] (0,0.7) -- (2,0) node[midway,above] {$2$} ;
\draw [ultra thick] (0,-0.7) -- (2,0);

\begin{scope}[shift={(2,0)}]
\draw [thick] (0,0) -- (0.8,0.6) node[right]{$2$};
\draw [thick] (0,0) -- (0.8,0.2);
\draw [thick] (0,0) -- (0.8,-0.4) node[right]{$3$};
\draw[fill,white](0,0) circle(2mm);
\draw[thick](0,0) circle(2mm);
\draw (0,0) node{$3$};
\draw(0,-0.2) node[below]{$w'$};
\end{scope}

\end{scope}

\end{tikzpicture}
\end{center}

\end{enumerate}

\end{proof}

\begin{lemma} \label{lemma:partition} Let $d$ be a positive integer, let $a_1=(a_{11},\ldots,a_{1k_1}),\ldots,a_m=(a_{m1},\ldots,a_{mk_m})$ be an $m$-tuple of partitions of $d$ of lengths $k_1,\ldots,k_m$, and suppose that
$$
t=k_1+\cdots+k_m-d(m-1)\geq 1.
$$
Then there is a partition $c=(c_1,\ldots,c_t)$ of $d$ of length $t$ and refinements $r_i:a_i\to c$ with the property that, for each $j=1,\ldots,t$, the sum of the lengths of the partitions $r_1^{-1}(c_j),\ldots,r_n^{-1}(c_j)$ of $c_j$ is equal to $c_j(d-1)+1$.

\end{lemma}

\begin{proof} We order each partition $a_i$ in an arbitrary way, and we visualize the entire $m$-tuple as a rectangular wall of size $d\times m$, with rows of horizontal blocks representing the partitions, located on the $xy$-plane with the origin in the lower left corner. For example, if $d=6$, $m=3$, $a_1=(1,2,2,1)$, $a_2=(1,1,1,2,1)$ and $a_3=(2,1,1,1,1)$, we obtain the following picture:
\begin{center}
\begin{tikzpicture}
\draw[thick] (0,0) -- (3,0);
\draw[thick] (0,0.5) -- (3,0.5);
\draw[thick] (0,1) -- (3,1);
\draw[thick] (0,1.5) -- (3,1.5);
\draw[thick] (0,0) -- (0,1.5);
\draw[thick] (0.5,0.5) -- (0.5,1.5);
\draw[thick] (1,0) -- (1,1);
\draw[thick] (1.5,0) -- (1.5,1.5);
\draw[thick] (2,0) -- (2,0.5);
\draw[thick] (2.5,0) -- (2.5,1.5);
\draw[thick] (3,0) -- (3,1.5);
\end{tikzpicture}
\end{center}
For each integer $i=0,\ldots,m$, we define the function
$$
F(i)=\#\big\{\mbox{blocks lying to the left of the vertical line }x=i\big\}-i(m-1).
$$
In the example above, we have
$$
F(0)=0,\quad F(1)=0,\quad F(2)=0, \quad F(3)=1,\quad F(4)=0,\quad F(5)=1,\quad F(6)=2.
$$
The function $F$ satisfies the following properties:
$$
F(0)=0,\quad F(d)=k_1+\cdots+k_m-d(m-1)=t,\quad \textrm{ and }\quad 1-m\leq F(i+1)-F(i)\leq 1.
$$
In other words, $F$ changes from $F(0)=0$ to $F(d)=t$ and, at each step, either decreases, or increases by at most $1$. Furthermore, if $F(i)-F(i-1)=1$, then there are $m$ more blocks to the left of $x=i$ than to the left of $x=i-1$, meaning that the vertical line $x=i$ does not bisect any blocks. 

By the pigeonhole principle, there is an increasing (not necessarily unique) sequence of integers $d_1,\ldots,d_t=d$ such that
$$
F(d_j)=j\quad \textrm{ and }\quad F(d_j-1)=j-1.
$$
In the example above, we can choose $d_1=3$ or $d_1=5$. The vertical lines $x=d_j$ do not bisect any blocks, and the number of blocks between $x=d_{j-1}$ and $x=d_j$ is equal to
\begin{equation*}\begin{split}
\#\big\{\mbox{blocks between }x=d_{j-1}\mbox{ and }x=d_j\big\}&
=F(d_j)-F(d_{j-1})+(d_j-d_{j-1})(m-1)\\&=(d_j-d_{j-1})(m-1)+1.
\end{split}\end{equation*}
It follows that $c_j=d_j-d_{j-1}$ is the desired partition, and the refinement maps $r_i:a_i\to c$ send the blocks in the $i$-th row between $x=d_{j-1}$ and $x=d_j$ to $c_j$.  
\end{proof}

\subsection{Proof of Theorem \ref{thm_structureDRlocus}}
\label{sec:mainproof}
We are now ready to prove Theorem~\ref{thm:main}. By Corollary \ref{cor_tree->linearequivalence} we have a natural inclusion $\DR_{g,a}^{trop}\subseteq \PD_{g,a}^{trop}$, so it is enough to show that $DR_{g,a}$ is a semilinear subset of $M_{g,n}^{trop}$ of dimension $2g-3+n$. The proof consists of two parts. First, we give an explicit algorithm for computing $DR_{g,a}$, which shows that it is a semilinear subset of $M_{g,n}^{trop}$ of dimension less than or equal to $2g-3+n$. The idea is that, given a finite effective harmonic morphism $\tau:\Ga'\to \De$ as above, we use Remark~\ref{rem:addramification} and add enough legs to $\Ga'$ and $\De$ to promote $\tau$ to an unramified harmonic morphism $\ttau:\tGa\to \tDe$ to a {\it stable} tree $\tDe\in M_{0,2g+n}^{trop}$. We then enumerate all possible stable trees $\tDe$ and all unramified covers of $\tDe$, and forget the additional legs to recover $DR_{g,a}$. This construction can be seen as a tropical analogue of classical dimension counts for gonality loci. 

Second, we inductively construct an explicit subset of $DR_{g,a}$ of dimension $2g-3+n$. We borrow an idea from \cite{CoolsDraisma}, and indeed our theorem can be viewed as a generalization of \cite[Theorem 1]{CoolsDraisma}.

\medskip \noindent {\bf Algorithm for constructing $DR_{g,a}$}. Let $\Ga\in DR_{g,a}$, and let $\tau:\Ga'\to \De$ be a finite effective harmonic morphism, where $(\Ga')_{st}=\Ga$, $\De$ is a tree with two legs marked $0$ and $\infty$, and $\tau$ acts on the legs according to~\eqref{eq:DRcondition}. Equation~\eqref{eq:RHformula} implies that 
$$
\Ram(\tau)=-\chi(\Ga')=2g+n-2. 
$$
We now use Remark~\ref{rem:addramification} and add legs to $\Ga'$ and $\De$ to kill the ramification. Specifically, we attach $N=2g+n-2$ legs $q_{11},\ldots,q_{N1}$ to $\Ga'$ at the support of the ramification divisor of $\tau$, then add $N$ legs $r_1,\ldots,r_N$ to $\De$ at the images of the ramification points, and finally $N(d-2)$ additional legs $q_{jk}$ to $\Ga'$ at the remaining preimages of the attachment points of the $r_j$, where $j=1,\ldots,N$ and $k=2,\ldots,d-1$. We denote the resulting curves by $\tGa$ and $\tDe$, and we extend $\tau$ to the new legs by mapping $q_{jk}$ to $r_j$, with degree 2 on the $q_{j1}$ and degree 1 on the rest. We then remove any remaining extremal edges from the source and target by passing to the stabilization (see Definition~\ref{def:stabilization}). As a result, we obtain an unramified harmonic morphism $\ttau:\tGa\to \tDe$, where $\tGa\in M_{g,n+N(d-1)}^{trop}$ and $\tDe\in M_{0,2+N}^{trop}$ are {\it stable} curves. We consider the locus of curves in $M_{g,n+N(d-1)}^{trop}$ that we obtain in this way:

\begin{definition} Let $g$ and $n$ be such that $2g-2+n>0$, let $a=(a_1,\ldots,a_n)\in \ZZ^n$ be an $n$-tuple of nonzero integers such that $a_1+\cdots+a_n=0$, let $d$ be the sum of the positive $a_i$, and let $N=2g+n-2$. We define $\tDR_{g,a}\subset M_{g,n+N(d-1)}^{trop}$ to be the locus of marked stable curves $\tGa$ admitting an unramified harmonic morphism $\ttau:\tGa\to \tDe$ of the following kind:

\begin{enumerate}

\item $\tGa$ is a stable tropical curve of genus $g$ with $n$ marked legs $p_1,\ldots,p_n$ and $N(d-1)$ additional marked legs $q_{jk}$, where $1\leq j\leq N$ and $1\leq k\leq d-1$.

\item $\tDe\in M_{0,2+N}^{trop}$ is a stable tree with marked legs $0$ and $\infty$, and $N$ additional marked legs $r_j$, where $1\leq j\leq N$.

\item $\ttau$ is an unramified harmonic morphism acting in the following way on the legs:
\begin{equation}\begin{split}
\ttau(p_i)&=\left\{\begin{array}{cc} 0, & \textrm{ if }a_i>0, \\ \infty, & \textrm{ if }a_i<0,\end{array}\right.\quad \textrm{ and }\quad d_{\ttau}(p_i)=|a_i|\\
\ttau(q_{jk})&=r_j\quad \textrm{ and }\quad d_{\ttau}(q_{jk})=\left\{\begin{array}{cl} 2, & \textrm{ if }k=1, \\ 1, & \textrm{ if }k=2,\ldots,d-1.\end{array}\right.
\label{eq:allram}
\end{split}\end{equation}
\end{enumerate} \label{def:DRtilde}

\end{definition}

Let $\pi:M_{g,n+N(d-1)}^{trop}\to M_{g,n}^{trop}$ be the map that forgets the additional marked legs $q_{jk}$. We claim that $\pi(\tDR_{g,a})=DR_{g,a}$. Indeed, the construction above shows that $\pi(\tDR_{g,a})\supset DR_{g,a}$, since we can recover $\Ga$ from $\tGa$ by removing the additional legs $q_{jk}$ and stabilizing. On the other hand, suppose that $\ttau:\tGa\to \tDe$ is an unramified harmonic morphism satisfying conditions~\eqref{eq:allram}. Let $\Ga$ and $\De$ be respectively $\tGa$ and $\tDe$ with the additional legs $q_{jk}$ and $r_j$ removed, and let $\tau:\Ga\to \De$ be the restriction of $\ttau$ to $\Ga$. Then by Remark~\ref{rem:restriction} the morphism $\tau$ is finite and effective, hence $\Ga_{st}$ lies in $DR_{g,a}$ and therefore $\pi(\tDR_{g,a})\subset DR_{g,a}$. 

We now have an algorithm for constructing $DR_{g,a}=\pi(\tDR_{g,a})$:
\begin{enumerate}
\item Enumerate all combinatorial types $\tT$ of tropical curves in $M_{0,2+N}^{trop}$.

\item For each combinatorial type $\tT$, use the algorithm of Proposition~\ref{prop:finitelymany} to enumerate all unramified harmonic morphisms $\tph:\tG\to \tT$ of graphs satisfying~\eqref{eq:allram}; there are finitely many such morphisms.

\item By Remark~\ref{rem:edgelengths}, for each choice of metric on $\tT$ there is a unique choice of metric on $\tG$ consistent with $\tph$. We thus obtain a cone of $\tDR_{g,a}$ for each unramified harmonic morphism $\tph:\tG\to \tT$ constructed in Steps 1 and 2.

\item Apply the forgetful map to recover $DR_{g,a}=\pi(\tDR_{g,a})$.

\end{enumerate}
We observe that in Step 2, for any combinatorial type $\tT$ there does exist at least one unramified harmonic morphism $\tph:\tG\to \tT$ satisfying~\eqref{eq:allram}. Indeed, the trivial graph $\bullet_{0,2+N}$ is the contraction of $\tT$ along the set of its edges. The trivial morphism $\bullet_{g,n+N(d-1)}\to \bullet_{0,2+N}$ satisfying~\eqref{eq:allram} is an unramified harmonic morphism of degree $d$, and by Lemma~\ref{lem:edgecontraction} extends to an unramified harmonic morphism $\tph:\tG\to \tT$, which also satisfies~\eqref{eq:allram}. We note that in Step 1, it is sufficient to enumerate all maximally degenerate combinatorial types in $M_{0,2g+n}^{trop}$, since by Lemma~\ref{lem:edgecontraction} any unramified harmonic morphism $\tph:\tG\to \tT$ is an edge contraction of an unramified harmonic morphism whose target is maximally degenerate.

We now finish the proof of the first part of the theorem. For each unramified harmonic morphism $\tph:\tG\to \tT$ satisfying~\eqref{eq:allram}, the corresponding cone of $\tDR_{g,a}$ has the same dimension as $M_{\tT}^{trop}\subset M_{0,2+N}^{trop}$. Furthermore, the conditions on the metric on $\tG$ are $\ZZ$-linear, hence the set of metric graphs $\tG$ that we obtain this way (together with their edge contractions) is a linear subset of $M_{g,n+N(d-1)}^{trop}$, of dimension $\dim M_{\tT}^{trop}$. Taking the union of these subsets over all combinatorial types $\tT$ in $M_{0,2+N}^{trop}$ and all unramified harmonic morphisms $\tph:\tG\to \tT$, we obtain that $\tDR_{g,a}$ is a linear subset of $M_{g,n+N(d-1)}^{trop}$ of dimension $N-1=2g+n-3$. Hence by Proposition \ref{prop:linear} the double ramification locus $DR_{g,a}=\pi(\tDR_{g,a})$ is a semilinear subset of $M_{g,n}^{trop}$ of dimension less than or equal to $2g+n-3$. 

If $d\geq 3$, then the map $\pi$ may have positive-dimensional fibers on $\tDR_{g,a}$ (an example is given in Remark~\ref{rem:dimensiondrop}). Therefore, we cannot yet conclude that $DR_{g,a}$ has dimension $2g-3+n$. However, we will see in Section~\ref{sec:hyperelliptic} that $\pi$ has finite fibers on $\tDR_{g,a}$ when $d=2$. 

We remark that the moduli space $M_{g,n}^{trop}$ is {\it equidimensional} in the sense that its maximal cones with respect to inclusion have the same dimension. In addition, $M_{g,n}^{trop}$ is connected through codimension one. The locus $\tDR_{g,a}$ inherits these properties from $M_{g,2+N}^{trop}$, however, we cannot conclude the same about its projection $DR_{g,a}$. We nevertheless believe that $DR_{g,a}$ is also equidimensional and connected through codimension one.

\medskip \noindent {\bf A cell of top dimension.} We now produce a cone of $DR_{g,n}$ of dimension $2g-3+n$. We first prove that $\dim DR_{g,a}=2g-3+n$ when $n=2d$ and $a=(1,\ldots,1,-1,\ldots,-1)\in \ZZ^{2d}$. If $d\leq g/2+1$, this follows directly from the results of Cools and Draisma~\cite{CoolsDraisma}. Indeed, the principal result of~\cite{CoolsDraisma}, rephrased in our language, is the following: the locus of curves in $M_g^{trop}$ having a tropical modification that admits a finite effective harmonic morphism of degree $d$ to a metric tree has dimension $\min(2g+2d-5,3g-3)$, which is equal to $2g+2d-5$ if $d\leq g/2+1$. Given such a morphism $\tau:\Ga\to \De$, we attach two legs $0$ and $\infty$ to $\De$ at arbitrary points, and attach $d_{\tau}(x)$ legs to $\Ga$ at every point $x$ of $\tau^{-1}(\{0,\infty\})$. Extending $\tau$ to the legs with degree one does not change the ramification degree, hence the curve $\Ga$ with the new legs lies in $DR_{g,a}$. The choice of the attaching points of $0$ and $\infty$ was arbitrary, hence
$$
\dim DR_{g,a}=2g+2d-5+2=2g-3+n. 
$$
For $d> g/2+1$ we proceed by induction, using the procedure of grafting a tree (see Section 4.1 in \cite{CoolsDraisma}). Assume that $\dim DR_{g,a}=2g-3+n$ when $n=2d$ and $a=(1,\ldots,1,-1,\ldots,-1)\in \ZZ^{2d}$. Pick a curve $\Ga\in DR_{g,a}$ and let $\tau:\Ga'\to \De$ be the corresponding morphism. Let $p_1,\ldots,p_d$ and $q_1,\ldots,q_d$ be the legs of $\Ga'$ mapping to $0$ and $\infty$, respectively. Let
$$
\Xi=\big\{x'\in \Ga'\big\vert d_{\ph}(x')=1\big\}\subset \Ga'
$$
then $\Xi$ is open and non-empty, since the degree function is upper semi-continuous on $\Ga'$ and since $\Xi$ contains the legs of $\Ga'$. Pick three points $x'_i\in \Xi$ and three positive lengths $l_i$, for $i=1,2,3$. Associated to this data, we construct a curve $\tGa\in DR_{g+2,\ta}$, where $\ta=(1,\ldots,1,-1,\ldots,-1)\in \ZZ^{2d+2}$, in the following steps.

\begin{enumerate} \item Let $\tDe$ be the tree obtained from $\De$ by attaching a finite edge $e_i$ of length $l_i$ to the point $x_i=\tau(x'_i)$, for $i=1,2,3$. We denote $y_i$ the free endpoint of $e_i$.

\item Denote $\ph^{-1}(x_i)=\{x'_{i1},\ldots,x'_{im_i}\}$, so that $x'_{i1}=x'_i$.
Let $\Ga''$ denote the graph obtained from $\Ga'$ by attaching finite edges $e'_{ij}$ of lengths $l_i$ to the points $x'_{ik}$, where $i=1,2,3$ and $j=1,\ldots,d$, in such a way that $e'_{i1}$ is attached to $x'_i$ and the number of edges $e'_{ij}$ that are attached to $x'_{ik}$ is equal to $d_{\ph}(x'_{ik})$. Denote $y'_i$ the free endpoint of $e'_{i1}$.

\item Let $\tGa'$ be the union of $\Ga''$ and a copy of $\tDe$, where we glue $y'_i\in \Ga''$ to $y_i\in \tDe$ for $i=1,2,3$. The curve $\tGa'$ has genus $g+2$, $2d$ legs $p_j$ and $q_j$, and two additional legs coming from $\tDe$, which we denote $0'$ and $\infty'$.

\item We define $\tph:\tGa'\to \tDe$ using $\ph$ on $\Ga'\subset \Ga''\subset\tGa'$, by sending the edges $e'_{ij}\subset \Ga''$ to $e_i$ with degree one, and as the identity on $\tDe$.

\end{enumerate}

It is easy to check that $\tph$ is an unramified harmonic morphism of degree $d+1$, mapping the legs $p_j$ and $0'$ to $0$ and the legs $q_j$ and $\infty'$ to $\infty$. Therefore we have constructed a curve $\tGa=(\tGa')_{st}\in \DR_{g+2,\ta}$. We now count parameters. The curve $\tGa$ depends on choosing three points $x'_i\in \De$ and three lengths $l_i$, so we have constructed a six-dimensional family of curves $\tGa\in \DR_{g+2,\ta}$ for each $\Ga\in \DR_{g,a}$. It follows that
$$
\dim \DR_{g+2,\ta}=\dim \DR_{g,a}+6=2g+2d-3+6=2(g+2)+(2d+2)-3,
$$
which proves the induction step, because $\ta$ has $2d+2$ entries.

To prove that $\dim \DR_{g,a}=2g-3+n$ for any $a=(a_1,\ldots,a_n)$, we proceed by induction on $\max |a_i|$, the base case $a_i=\pm 1$ having been established above. Let $a'=(a_1,\ldots,a_n,b,c)$ be a ramification profile with $b$ and $c$ both positive (or both negative, the argument is identical), and suppose that we have established that $\dim \DR_{g,a'}=2g-1+n$. We claim that $\dim \DR_{g+1,a''}=2g+n$, where $a''=(a_1,\ldots,a_n,b+c)$. Indeed, pick a curve $\Ga\in \DR_{g,a'}$, and let $\ph:\Ga'\to \De$ be the corresponding morphism, where $(\Ga')^{st}=\Ga$. The curve $\Ga'$ has legs $p_1,\ldots,p_n$ mapping with degree $|a_i|$ to $0$ or $\infty$ (depending on whether $a_i>0$ or $a_i<0$), and two additional legs $l_1$ and $l_2$ mapping to the leg $0$ in $\De$. 

We now construct a curve $\tGa\in \DR_{g+1,a''}$ as follows. Pick a point $x$ anywhere on the leg $0\subset \De$, and let $x'_1$ and $x'_2$ be the preimages of $x$ on $l_1$ and $l_2$, respectively. We define $\tGa$ by gluing the two legs $l_1$ and $l_2$ together into a single leg starting at the points $x'_1$ and $x'_2$. We define the morphism $\tph:\tGa\to \De$ by sending $l$ to $0$ with degree $b+c$, and as $\ph$ on the remainder of $\tGa$. The curve $\tGa$ has genus $g+1$, and $\tph$ has ramification degree one at the identified point $x'_1=x'_2$, therefore $\tph$ is finite and effective, and hence $(\tGa)^{st}\in \DR_{g+1,a''}$. Since $\dim \DR_{g,a'}=2g-1+n$ and the construction involved choosing an arbitrary point $x\in 0$, it follows that  $\dim \DR_{g+1,a''}=2g+n$. This completes the proof.


\subsection{Elliptic curves with two marked points: $DR_{1,(d,-d)}$} \label{sec:example}

In this section, we demonstrate the algorithm of Theorem~\ref{thm:main} by computing the double ramification locus $DR_{1,(d,-d)}\subset M_{1,2}^{trop}$, where $d\geq 2$. We recall that $DR_{1,(d,-d)}=\pi(\tDR_{1,(d,-d)})$, where $\tDR_{1,(d,-d)}\subset M_{1,2d}^{trop}$ is the locus is given by Definition~\ref{def:DRtilde}, and $\pi:M_{1,2d}^{trop}\to M_{1,2}^{trop}$ is the forgetful map. Specifically, $\tDR_{1,(d,-d)}$ consists of curves $\tGa$ admitting an unramified harmonic morphism $\ttau:\tGa\to \tDe$ of the following kind:

\begin{enumerate}

\item \label{item:1dd1} $\tGa\in M_{1,2d}^{trop}$ is a stable tropical curve with legs $p_1$ and $p_2$, and additional legs $q_{ij}$, where $i=1,2$ and $j=1,\ldots,d-1$.

\item $\tDe\in M_{0,4}^{trop}$ is a stable tree with legs $0$, $\infty$, and additional legs $r_1$ and $r_2$.

\item $\ttau$ acts on the legs as follows:
\begin{equation*}\begin{split}
\ttau(p_1)& =0\quad \textrm{ and }\quad \ttau(p_2)=\infty\quad \textrm{ with } \quad d_{\ttau}(p_1)=d_{\ttau}(p_2)=d,\\
 \ttau(q_{ij})&=r_i\quad \textrm{ with } \quad d_{\ttau}(q_{ij})=\left\{\begin{array}{cl} 2 & \textrm{ if } j=1, \\ 1 & \textrm{ if } j=2,\ldots,d-1.\end{array}\right.
\end{split}\end{equation*}

\end{enumerate}
We enumerate the possible target trees $\tDe\in M_{0,4}^{trop}$, and for each tree we construct the possible covers $\ttau:\tGa\to \tDe$. For each $\tGa$, we then forget the additional legs $q_{ij}$ and stabilize to obtain the locus $DR_{1,(d,-d)}$.

\begin{enumerate} \item We first consider the following tree $\tDe$:
\begin{center}
\begin{tikzpicture}

\draw [thick] (0,0) -- (3,0) node[midway,above] {$e$} ;
\draw[fill] (0,0) circle(0.8mm);
\draw[fill] (3,0) circle(0.8mm);

\draw (0.1,-0.1) node[below] {$v_1$};

\draw [thick] (0,0) -- (-0.8,0.4) node[left] {$0$};
\draw [thick] (0,0) -- (-0.8,-0.4) node[left] {$r_1$};
\begin{scope}[shift={(3,0)}]
\draw [thick] (0,0) -- (0.8,0.4) node[right]{$\infty$};
\draw [thick] (0,0) -- (0.8,-0.4) node[right]{$r_2$};
\draw (0.0,-0.1) node[below] {$v_2$};
\end{scope}
\end{tikzpicture}
\end{center}
Let $\ttau:\tGa\to \tDe$ be an unramified harmonic morphism of the type described above. Since $0$ and $\infty$ each have one preimage in $\tGa$, it follows that the graph $\tGa$ has two vertices $v'_1$ and $v'_2$, mapping with degree $d$ to $v_1$ and $v_2$, respectively, and that the legs $q_{ij}$ are attached to $v'_i$. Let $\ttau^{-1}(e)=\{e'_1,\ldots,e'_n\}$ be the edges of $\tGa$, then $\val(v'_1)=\val(v'_2)=d+n$, hence the conditions 
$$
0=\Ram_{\ttau}(v'_i)=d\chi(v_i)-\chi(v'_i)=2g(v'_i)+n-2
$$
imply that $n=2$ and $g(v'_1)=g(v'_2)=0$. The harmonicity condition implies that the degrees $b_1=d_{\ph}(e'_1)$ and $b_2=d_{\ph}(e'_2)$ are arbitrary positive integers adding up to $d$, and the lengths of $e'_1$ and $e'_2$ satisfy the condition $b_1l(e'_1)=l(e)=b_2l(e'_2)$. Hence we obtain the following unramified harmonic morphism $\ttau:\tGa\to \tDe$:

\begin{center}
\begin{tikzpicture}

\begin{scope}[shift={(0,0)}]

\draw [thick] (0,0) .. controls (0,0.5) and (3,0.5) .. (3,0) node[midway,above] {$e'_1$} ;
\draw [thick] (0,0) .. controls (0,-0.5) and (3,-0.5) .. (3,0) node[midway,above] {$e'_2$} ;
\draw[fill] (0,0) circle(.8mm);
\draw[fill] (3,0) circle(.8mm);

\draw (0.1,-0.1) node[below] {$v'_1$};

\draw [thick] (0,0) -- (-0.8,0.4) node[left] {$p_1$};

\draw [thick] (0,0) -- (-0.8,0);
\draw [thick] (0,0) -- (-0.8,-0.2) ;
\draw [thick] (0,0) -- (-0.8,-0.4) ;
\draw [thick] (0,0) -- (-0.8,-0.6) ;
\draw (-0.8,-0.4) node[left]{$q_{1j}$};

\begin{scope}[shift={(3,0)}]
\draw (0.0,-0.1) node[below] {$v'_2$};
\draw [thick] (0,0) -- (0.8,0.4) node[right] {$p_2$};
\draw [thick] (0,0) -- (0.8,0);
\draw [thick] (0,0) -- (0.8,-0.2) ;
\draw [thick] (0,0) -- (0.8,-0.4) ;
\draw [thick] (0,0) -- (0.8,-0.6) ;
\draw (0.8,-0.4) node[right]{$q_{2j}$};
\end{scope}

\end{scope}

\begin{scope}[shift={(0,-2)}]

\draw [thick] (0,0) -- (3,0) node[midway,above] {$e$} ;
\draw[fill] (0,0) circle(0.8mm);

\draw (0.1,-0.1) node[below] {$v_1$};

\draw [thick] (0,0) -- (-0.8,0.4) node[left] {$0$};
\draw [thick] (0,0) -- (-0.8,-0.4) node[left] {$r_1$};

\begin{scope}[shift={(3,0)}]
\draw [thick] (0,0) -- (0.8,0.4) node[right]{$\infty$};
\draw [thick] (0,0) -- (0.8,-0.4) node[right]{$r_2$};
\draw[fill] (0,0) circle(0.8mm);
\draw (0,-0.1) node[below] {$v_2$};
\end{scope}

\end{scope}

\end{tikzpicture}  \end{center}

\item Exchanging the legs $r_1$ and $r_2$ in $\tDe$ above corresponds to relabeling the legs $q_{ij}$ in $\tGa$, and produces the same curve $\Ga\in DR_{1,(d,-d)}$. 

\item We now consider the following tree $\tDe$:
\begin{center}
\begin{tikzpicture}

\draw [thick] (0,0) -- (3,0) node[midway,above] {$e$} ;
\draw[fill] (0,0) circle(0.8mm);
\draw (0.1,-0.1) node[below] {$v_1$};
\draw [thick] (0,0) -- (-0.8,0.4) node[left] {$0$};
\draw [thick] (0,0) -- (-0.8,-0.4) node[left] {$\infty$};

\begin{scope}[shift={(3,0)}]
\draw[fill] (0,0) circle(0.8mm);
\draw (0,-0.1) node[below] {$v_2$};
\draw [thick] (0,0) -- (0.8,0.4) node[right]{$r_1$};
\draw [thick] (0,0) -- (0.8,-0.4) node[right]{$r_2$};
\end{scope}

\end{tikzpicture}
\end{center}
Let $\ttau:\tGa\to \tDe$ be an unramified harmonic morphism, then $v_1$ has a unique preimage $v'\in V(\tGa)$, and $d_{\ph}(v')=d$. Let $\ttau^{-1}(e)=\{e'_1,\ldots,e'_n\}$ be the legs of $\tGa$. The condition 
$$
0=\Ram_{\ttau}(v')=d\chi(v_1)-\chi(v')=2g(v')+n-d
$$
implies that either $g(v')=0$ and $n=d$, or $g(v')=1$ and $n=d-2$, because the total genus of $\tGa$ is equal to one.

\begin{enumerate} \item \label{item:1dd3a} If $g(v')=0$ and $n=d$, then $d_{\ph}(e'_j)=1$ for all $j=1,\ldots,d$, and therefore all $e'_j$ have the same length as $e$. Let $\ttau^{-1}(v_2)=\{u'_1,\ldots,u'_m\}$ be the remaining vertices of $\tGa$. Without loss of generality assume that the leg $q_{11}$ is attached to $u'_1$, then $d_{\ttau}(u'_1)\geq 2$ and hence $m\leq d-1$. The genus of $\tGa$ is equal to $d-m$ plus the sum of the genera of the $u'_j$, which implies that $m=d-1$, $g(u'_j)=0$ for $j=1,\ldots, d-1$, $d_{\ttau}(u'_1)=2$, and $d_{\ttau}(u'_j)=1$ for $j=2,\ldots,d-1$. Looking at the local degrees, and relabeling as necessary, we see that the edges $e'_j$ and the legs $q_{ij}$ are attached to $u'_j$ for $j=1,\ldots,d-1$, while $e'_n$ is attached to $u'_1$. Hence we obtain the following unramified harmonic morphism $\ttau:\tGa\to \tDe$:

\begin{center}
\begin{tikzpicture}
\begin{scope}[shift={(0,0)}]
\draw [thick] (0,0) .. controls (0.2,0.5) and (2,1.4) .. (3,1.5) node[midway,above]{$e'_1$};
\draw [thick] (0,0) .. controls (1,0.1) and (2.8,1) .. (3,1.5)node[midway,above]{$e'_2$};
\draw[fill](0,0) circle(.9mm);
\draw(0,-0.1) node[below]{$v'$};
\draw [thick] (0,0) -- (-0.8,0.4) node[left]{$p_1$};
\draw [thick] (0,0) -- (-0.8,-0.4) node[left]{$p_2$};
\draw[fill](3,1.5) circle(.9mm);
\draw(3,1.4) node[below]{$u'_1$};
\draw [thick] (3,1.5) -- (3.8,1.9) node[right] {$q_{11}$};
\draw [thick] (3,1.5) -- (3.8,1.1) node[right] {$q_{21}$};
\draw [thick] (0,0) -- (3,0) node[near end,above]{$e'_3$};
\draw[fill](3,0) circle(.8mm) node[below]{$u'_2$};
\draw [thick] (3,0) -- (3.8,0.4) node[right]{$q_{12}$};
\draw [thick] (3,0) -- (3.8,-0.4) node[right]{$q_{22}$};

\draw [thick] (0,0) -- (3,-1.5) node[midway,below]{$e'_d$};
\draw[fill](3,-1.5) circle(.8mm) node[below]{$u'_{d-1}$};
\draw [thick] (3,-1.5) -- (3.8,-1.1) node[right]{$q_{1,d-1}$};
\draw [thick] (3,-1.5) -- (3.8,-1.9) node[right]{$q_{2,d-1}$};
\draw [thick] (0,0) -- (2,-0.25);
\draw [thick] (0,0) -- (2,-0.5) node[right]{$e'_i$};
\draw [thick] (0,0) -- (2,-0.75);

\end{scope}

\begin{scope}[shift={(0,-3)}]

\draw [thick] (0,0) -- (3,0) node[midway,above] {$e$} ;
\draw[fill] (0,0) circle(0.8mm);

\draw (0.1,-0.1) node[below] {$v_1$};

\draw [thick] (0,0) -- (-0.8,0.4) node[left] {$0$};
\draw [thick] (0,0) -- (-0.8,-0.4) node[left] {$\infty$};

\begin{scope}[shift={(3,0)}]
\draw[fill] (0,0) circle(0.8mm);
\draw (0.0,-0.1) node[below] {$v_2$};
\draw [thick] (0,0) -- (0.8,0.4) node[right]{$r_1$};
\draw [thick] (0,0) -- (0.8,-0.4) node[right]{$r_2$};
\end{scope}

\end{scope}

\end{tikzpicture}

\end{center}

\item \label{item:1dd3b} If $g(v')=1$ and $n=d-2$, the genus constraint implies that each $e'_j$ is attached to a distinct vertex $u'_j\in \ttau^{-1}(v_2)$ of genus zero. The degrees $d_{\ttau}(e'_j)=d_{\ttau}(u'_j)$ form a partition of $d$ of length $d-2$. Relabeling if necessary, we see that there are two possibilities:

\begin{itemize} \item $d\geq 3$,  $d_{\ttau}(u'_1)=3$, $d_{\ttau}(u'_j)=1$ for $i=2,\ldots,d-2$, $q_{ij}$ are attached to $u'_j$ for $j=1,\ldots,d-2$, $q_{i,d-1}$ are attached to $u'_1$.

\item $d\geq 4$, $d_{\ttau}(u'_1)=d_{\ttau}(u'_2)=2$, $d_{\ttau}(u'_j)=1$ for $j=3,\ldots,d-2$, $q_{11}$, $q_{22}$, and $q_{2,d-1}$ are attached to $u'_1$, $q_{21}$, $q_{12}$, and $q_{1,d-1}$ are attached to $u'_2$, and $q_{ij}$ are attached to $u'_j$ for $j=3,\ldots,d-2$.

\end{itemize}
We obtain the following unramified harmonic morphisms $\ttau:\tGa\to \tDe$:

\begin{center}
\begin{tikzpicture}
\begin{scope}[shift={(0,0)}]
\draw [thick] (0,0) -- (3,1.5) node[midway,above]{$e'_1$};
\draw[thick](0,0) circle(2mm);
\draw (0,0) node{$1$};
\draw(0,-0.1) node[below]{$v'$};
\draw [thick] (0,0) -- (-0.8,0.4) node[left]{$p_1$};
\draw [thick] (0,0) -- (-0.8,-0.4) node[left]{$p_2$};
\draw[fill](3,1.5) circle(.9mm);
\draw(3,1.4) node[below]{$u'_1$};
\draw [thick] (3,1.5) -- (3.8,2.1) node[right] {$q_{11}$};
\draw [thick] (3,1.5) -- (3.8,1.7) node[right] {$q_{21}$};
\draw [thick] (3,1.5) -- (3.8,1.3) node[right] {$q_{1,d-1}$};
\draw [thick] (3,1.5) -- (3.8,0.9) node[right] {$q_{2,d-1}$};
\draw [thick] (0,0) -- (3,0) node[near end,above]{$e'_2$};
\draw[fill](3,0) circle(.8mm) node[below]{$u'_2$};
\draw [thick] (3,0) -- (3.8,0.4) node[right]{$q_{12}$};
\draw [thick] (3,0) -- (3.8,-0.4) node[right] {$q_{22}$};
\draw [thick] (0,0) -- (3,-1.5) node[midway,below]{$e'_{d-2}$};
\draw[fill](3,-1.5) circle(.8mm) node[below]{$u'_{d-2}$};
\draw [thick] (3,-1.5) -- (3.8,-1.1) node[right]{$q_{1,d-2}$};
\draw [thick] (3,-1.5) -- (3.8,-1.9) node[right]{$q_{2,d-2}$};
\draw [thick] (0,0) -- (2,-0.25);
\draw [thick] (0,0) -- (2,-0.5) node[right]{$e'_i$};
\draw [thick] (0,0) -- (2,-0.75);
\draw[fill, white](0,0) circle(2mm);
\draw[thick](0,0) circle(2mm);
\draw (0,0) node{$1$};

\end{scope}

\begin{scope}[shift={(0,-3)}]

\draw [thick] (0,0) -- (3,0) node[midway,above] {$e$} ;
\draw[fill] (0,0) circle(0.8mm);
\draw[fill] (3,0) circle(0.8mm);

\draw (0.1,-0.1) node[below] {$v_1$};
\draw (3.0,-0.1) node[below] {$v_2$};

\draw [thick] (0,0) -- (-0.8,0.4) node[left] {$0$};
\draw [thick] (0,0) -- (-0.8,-0.4) node[left] {$\infty$};
\draw [thick] (3,0) -- (3.8,0.4) node[right]{$r_1$};
\draw [thick] (3,0) -- (3.8,-0.4) node[right]{$r_2$};

\end{scope}

\begin{scope}[shift={(7,0)}]

\draw [thick] (0.18,0.09) -- (3,1.5) node[midway,above]{$e'_1$};

\draw(0,-0.1) node[below]{$v'$};
\draw [thick] (0,0) -- (-0.8,0.4) node[left]{$p_1$};
\draw [thick] (0,0) -- (-0.8,-0.4) node[left]{$p_2$};
\draw[fill](3,1.5) circle(.9mm);
\draw(3,1.4) node[below]{$u'_1$};
\draw [thick] (3,1.5) -- (3.8,1.9) node[right] {$q_{11}$};
\draw [thick] (3,1.5) -- (3.8,1.5) node[right] {$q_{22}$};
\draw [thick] (3,1.5) -- (3.8,1.1) node[right] {$q_{2,d-1}$};
\draw [thick] (0,0) -- (3,0) node[near end,above]{$e'_2$};
\draw[fill](3,0) circle(.8mm) node[below]{$u'_2$};
\draw [thick] (3,0) -- (3.8,0.4) node[right]{$q_{21}$};
\draw [thick] (3,0) -- (3.8,0) node[right]{$q_{12}$};
\draw [thick] (3,0) -- (3.8,-0.4) node[right] {$q_{1,d-1}$};
\draw [thick] (0,0) -- (3,-1.5) node[midway,below]{$e'_{d-2}$};
\draw[fill](3,-1.5) circle(.8mm) node[below]{$u'_{d-2}$};
\draw [thick] (3,-1.5) -- (3.8,-1.1) node[right]{$q_{1,d-2}$};
\draw [thick] (3,-1.5) -- (3.8,-1.9) node[right]{$q_{2,d-2}$};
\draw [thick] (0,0) -- (2,-0.25);
\draw [thick] (0,0) -- (2,-0.5) node[right]{$e'_i$};
\draw [thick] (0,0) -- (2,-0.75);

\draw[fill, white](0,0) circle(2mm);
\draw[thick](0,0) circle(2mm);
\draw (0,0) node{$1$};

\end{scope}

\begin{scope}[shift={(7,-3)}]

\draw [thick] (0,0) -- (3,0) node[midway,above] {$e$} ;
\draw[fill] (0,0) circle(0.8mm);
\draw[fill] (3,0) circle(0.8mm);

\draw (0.1,-0.1) node[below] {$v_1$};
\draw (3.0,-0.1) node[below] {$v_2$};

\draw [thick] (0,0) -- (-0.8,0.4) node[left] {$0$};
\draw [thick] (0,0) -- (-0.8,-0.4) node[left] {$\infty$};
\draw [thick] (3,0) -- (3.8,0.4) node[right]{$r_1$};
\draw [thick] (3,0) -- (3.8,-0.4) node[right]{$r_2$};

\end{scope}

\end{tikzpicture}

\end{center}

\end{enumerate}

\item \label{item:1dd4}Finally, if $\tDe$ consists of a single vertex $v$ with all legs attached to it, then $\tGa$ consists of a single vertex $v'$ of genus one with all legs attached to it. 

\end{enumerate}

We now consider the image of the locus $\tDR_{1,(d,-d)}$ described above under the map $\pi:M_{1,2d}^{trop}\to M_{1,2}^{trop}$ that forgets the legs $q_{ij}$ and stabilizes. The curves $\tGa$ described in \ref{item:1dd1} and \ref{item:1dd3a} above project to the following one-dimensional families of curves in $DR_{1,(d,-d)}\subset M_{1,2}^{trop}$:
\begin{center}
\begin{tikzpicture}

\draw [thick] (0,0) .. controls (0,0.5) and (3,0.5) .. (3,0) node[midway,above] {$e'_1$} ;
\draw [thick] (0,0) .. controls (0,-0.5) and (3,-0.5) .. (3,0) node[midway,above] {$e'_2$} ;
\draw[fill] (0,0) circle(.8mm);
\draw[fill] (3,0) circle(.8mm);

\draw [thick] (0,0) -- (-0.8,0) node[left] {$p_1$};
\draw [thick] (3,0) -- (3.8,0);
\draw (3.8,0) node[right]{$p_2$};

\begin{scope}[shift={(7,0)}]

\draw [thick] (0,0) .. controls (0,0.5) and (3,0.5) .. (3,0) ;
\draw [thick] (0,0) .. controls (0,-0.5) and (3,-0.5) .. (3,0) ;
\draw[fill] (0,0) circle(.8mm);
\draw [thick] (0,0) -- (-0.8,0.4) node[left] {$p_1$};
\draw [thick] (0,0) -- (-0.8,-0.4) node[left] {$p_2$};
\end{scope}

\end{tikzpicture} 
\end{center}
The curve on the right has a loop of arbitrary length, while the edge lengths of the curve on the left satisfy the constraint 
\begin{equation}
b_1l(e'_1)=b_2l(e'_2),\quad b_1,b_2\geq 1,\quad b_1+b_2=d.
\label{eq:lengthsum}
\end{equation}
On the other hand, the curves described in~\ref{item:1dd3b} and~\ref{item:1dd4} all project to the curve
\begin{center}
\begin{tikzpicture}

\draw [thick] (0,0) -- (-1,0) node[left]{$p_1$};
\draw [thick] (0,0) -- (1,0) node[right]{$p_2$};
\draw[fill,white](0,0) circle(2mm);
\draw[thick](0,0) circle(2mm);
\draw (0,0) node{$1$};
\end{tikzpicture}
\end{center}

\begin{remark} We observe that the curves of type~\ref{item:1dd3b} form one-dimensional strata in $\tDR_{1,(d,-d)}$, but project to a single point in $\DR_{1,(d,-d)}$. Hence the map $\pi:\tDR_{g,a}\to \DR_{g,a}$ may in general have positive-dimensional fibers. Note that this behavior appears only when $d\geq 3$.
\label{rem:dimensiondrop}
\end{remark}

We see that the double ramification locus $\DR_{1,(d,-d)}\subset M_{1,2}^{trop}$ consists of $\lfloor d/2\rfloor+1$ one-dimensional faces corresponding to the partitions $d=b_1+b_2$, and an additional one-dimensional face.

It is instructive to compare the loci $\DR_{1,(d,-d)}$ and $PD_{1,(d,-d)}$. The former is contained in the latter, and geometrically it is clear that Equation~\eqref{eq:lengthsum} states that the point $p_1-p_2$ is a $d$-torsion point in the Jacobian of $\Ga$. However, $\PD_{1,(d,-d)}$ contains the following additional cone, which has empty intersection with $\DR_{1,(d,-d)}$:
\begin{center}
\begin{tikzpicture}
\draw [thick] (7,0) .. controls (7,0.5) and (9,0.5) .. (9,0) ;
\draw [thick] (7,0) .. controls (7,-0.5) and (9,-0.5) .. (9,0) ;
\draw [thick] (6,0) -- (7,0);
\draw[fill] (6,0) circle(.8mm);
\draw[fill] (7,0) circle(.8mm);
\draw [thick] (6,0) -- (5.2,0.4) node[left] {$p_1$};
\draw [thick] (6,0) -- (5.2,-0.4) node[left] {$p_2$};
\end{tikzpicture}
\end{center}
Unlike $\DR_{1,(d,-d)}$, this stratum has codimension zero in $M_{1,2}^{trop}$.

\subsection{The hyperelliptic case} \label{sec:hyperelliptic}\label{section_hyperelliptic}

In this section, we work out the example of hyperelliptic double ramification loci. Specifically, we give a detailed description of the three (up to relabeling) double ramification loci of degree two, namely $\DR_{g,(2,-2)}\subset M_{g,2}^{trop}$, $\DR_{g,(2,-1,-1)}\subset M_{g,3}^{trop}$, and $\DR_{g,(1,1,-1,-1)}\subset M_{g,4}^{trop}$. We also determine the relationship of these double ramification loci to the hyperelliptic locus $\calH_g\subset \calM_g$. Our main result is Theorem~\ref{thm_hyperellipticDC}, which we now restate: 

\begin{theorem} \label{thm:DRdeg2} The double ramification loci 
\begin{equation*}
\DR_{g,(2,-2)}\subset M_{g,2}^{trop}, \quad \DR_{g,(2,-1,-1)}\subset M_{g,3}^{trop}, \quad \textrm{ and } \quad \DR_{g,(1,1,-1,-1)}\subset M_{g,4}^{trop}
\end{equation*}
are linear subsets, connected in codimension one, having maximal cones of dimensions $2g-1$, $2g$, and $2g+1$, respectively. Their projections to $M_g^{trop}$ are equal to the realizable hyperelliptic locus $H_g$:
$$
\pi_2\big(\DR_{g,(2,-2)}\big)=H_g,\quad \pi_3\big(\DR_{g,(2,-1,-1)}\big)=H_g,\quad \textrm{ and } \quad \pi_4\big(\DR_{g,(1,1,-1,-1)}\big)=H_g.
$$
The fibers of this projection have dimensions $0$, $1$, and $2$, respectively.
\end{theorem}

\begin{proof} This theorem will follow from the explicit description for these loci that is given in Props.~\ref{prop:DRhyperelliptic},~\ref{prop:22},~\ref{prop:211}, and~\ref{prop:1111}. 
\end{proof}

We extensively use the theory of hyperelliptic tropical curves, as discussed in the introduction (see~\cite{Chan_hyperelliptic},~\cite{ABBRII},~\cite{BBC}). First, we give some definitions and results about harmonic morphisms of degree two. 

\begin{definition} Let $G$ be a weighted graph. A {\it hyperelliptic morphism} on $G$ is a finite harmonic morphism $\ph:G\to T$ of degree two, where $T$ is a weighted graph of genus zero, such that $d_{\ph}(v)=2$ for any vertex $v\in V(G)$ such that $g(v)>0$. We say that a hyperelliptic morphism is {\it effective} or {\it unramified} if it is so as a harmonic morphism, and we similarly define (effective, unramified) hyperelliptic morphisms of weighted metric graphs and tropical curves. 

\end{definition}

Let $\ph:G\to T$ be a hyperelliptic morphism. Given a vertex, edge, or leg $x\in X(T)$, there are two possibilities: either $x$ has two preimages on each of which $\ph$ has degree one, or $x$ has a unique preimage on which $\ph$ has degree two. If $\ph$ has degree two on a half-edge of $G$, then it necessarily has degree two on its endpoint, so the subset of $X(T)$ on which $\ph$ has degree two forms a subgraph:

\begin{definition} Let $\ph:G\to T$ be a hyperelliptic morphism of weighed graphs. We define the {\it dilation subgraph} $C_{\ph}\subset T$ to be
$$
C_{\ph}=\big\{x\in X(T)\big\vert\#(\ph^{-1}(x))=1\big\}=\ph\big(\big\{x'\in X(G)\big\vert \deg_{\ph}(x')=2\big\}\big).
$$
We similarly define the dilation subgraph of a hyperelliptic morphism of weighted metric graphs or tropical curves. The preimage $\ph^{-1}(C_{\ph})\subset G$ is isomorphic, viewed as a graph without weights, to $C_{\ph}$, so by abuse of notation we will occasionally refer to the preimage $\ph^{-1}(C_{\ph})$ as the dilation subgraph, and denote it $C_{\ph}$. 

\end{definition}

\begin{remark} Given the dilation subgraph $C_{\ph}\subset T$ of a hyperelliptic morphism $\ph:G\to T$, we can uniquely reconstruct the graph $G$ and the map $\ph$, but not the vertex weighting on $G$. Indeed, as a set $X(G)$ is the disjoint union of one copy $D$ of $X(C_{\ph})$ and two copies $U_1$ and $U_2$ of $X(T)\backslash X(C_{\ph})$. We define the graph structure on $G$ by gluing each $U_i$ to $D$ according to how they are glued in $T$. Finally, we define $\ph:G\to T$ by sending $D$ to $X(C_{\ph})$ with degree 2 everywhere, and each $U_i$ to $X(T)\backslash X(C_{\ph})$ with degree 1 everywhere.\label{rem:cycletocover}
\end{remark}

We will be frequently using the ramification divisor of a hyperelliptic morphism:

\begin{proposition} Let $\ph:G\to T$ be a hyperelliptic morphism, where $G$ is a graph of genus $g$. The ramification divisor of $\ph$ is equal to the inverse of the canonical divisor of the preimage $\ph^{-1}(C_{\ph})$ of the dilation subgraph:
\begin{equation}
\Ram \ph=-K_{\ph^{-1}(C_{\ph})}=\sum_{v\in V(C_{\ph})}\big[2g(v)+2-\val_{C_{\ph}}(v)\big]v.
\label{eq:ramification}
\end{equation}
If $G$ and $T$ have no legs, then $\deg \Ram \ph=2g+2$.

\end{proposition}

\begin{proof} Let $v\in V(G)$. If $v\notin V(C_{\ph})$, then $\ph$ has degree one at $v$ and $g(v)=0$, so $\ph$ is a local isomorphism in the neighborhood of $v$, and therefore $\Ram_{\ph}(v)=0$. If $v\in V(C_{\ph})$, then there are $\val_{C_{\ph}}(v)$ dilated and $\val_{G}(v)-\val_{C_{\ph}}(v)$ undilated tangent directions at $v$. Therefore
$$
\val_T(\ph(v))=\val_{G}(v)+\frac{\val_{G}(v)-\val_{C_{\ph}}(v)}{2}=\frac{\val_{G}(v)+\val_{C_{\ph}}(v)}{2},
$$
and since $g(v)=0$ we obtain that
$$
\Ram_{\ph}(v)=2\chi_{T}(\ph(v))-\chi_{G}(v)=2g(v)+2-\val_{C_{\ph}}(v).
$$
Finally, if $G$ and $T$ have no legs, then by the global Riemann--Hurwitz formula~\eqref{eq:RHformula}
$$
\deg \Ram \ph=\deg \ph \cdot \chi(T)-\chi(G)=2g+2.
$$
\end{proof}

A hyperelliptic morphism $\ph:G\to G$ defines an involution on $G$ whose fixed locus is $C_{\ph}$, providing an alternative and equivalent point of view.

\begin{definition} Let $G$ be a weighted graph. A {\it hyperelliptic involution} on $G$ is a non-trivial involution $\si:G\to G$ such that the quotient graph $G/\si$, equipped with the trivial weighting, has genus zero, and such that $\si(v)=v$ for all $v\in V(G)$ with $g(v)>0$. A hyperelliptic involution is called {\it effective} if for any $v\in V(G)$ fixed by $\si$, the number $\kappa(v)$ of tangent directions fixed by $\si$ is less than or equal to $2g(v)+2$, and {\it unramified} if $\kappa(v)=2g(v)+2$ for all $v\in V(G)$ fixed by $\si$. We define (effective, unramified) hyperelliptic involutions of weighted metric graphs and tropical curves similarly. 

\end{definition}

The following proposition is elementary:

\begin{proposition} Let $\ph:G\to T$ be a hyperelliptic morphism with dilation subgraph $C_{\ph}$. Define the involution $\si:G\to G$ to act trivially on $\ph^{-1}(C_{\ph})$, and by exchanging the two preimages of any $x\in X(T)\backslash X(C_{\ph})$. Then $\si$ is a hyperelliptic involution. Conversely, let $\si:G\to G$ is a hyperelliptic involution and let $\ph:G\to T$ be the induced map to the quotient $T=G/\si$. For $x\in X(G)$ let 
$$
d_{\ph}(x)=\left\{\begin{array}{cc} 2 & \textrm{ if } \si(x)=x, \\ 1 & \textrm{ if } \si(x)\neq x.\end{array}\right.
$$
Then $\ph$ is a hyperelliptic morphism. Under this bijection, effective and unramified hyperelliptic morphisms correspond to respectively effective and unramified hyperelliptic involutions.

\end{proposition}

\begin{proof} Let $\ph:G\to T$ be a hyperelliptic morphism, and let $\si:G\to G$ be as above. Then clearly $G/\si=T$ is a tree, and if $v\in V(G)$ is a vertex of positive genus, then $d_{\ph}(v)=2$, so $\si$ fixes $v$. Hence $\si$ is a hyperelliptic involution. Conversely, if $\si$ is a hyperelliptic involution and $\ph$ is the morphism defined above, then $\ph$ is harmonic, and for any $v\in V(G)$, if $g(v)>0$, then $\si(v)=v$ and hence $d_{\ph}(v)=2$. Finally, for any $v\in V(G)$ fixed by $\si$ we have $\kappa(v)=\val_{C_{\ph}}(v)$, so Equation~\eqref{eq:ramification} implies that $\ph$ is effective (respectively, unramified) if and only if $\si$ is.
\end{proof}

\begin{example} We give an example of a hyperelliptic morphism $\eta:\Ga\to \De$, where $\Ga$ is a stable genus two curve:

\begin{center}

\begin{tikzpicture}

\draw [thick] (1,0) .. controls (1,0.5) and (2,0.5) .. (2,0);
\draw [thick] (1,0) .. controls (1,-0.5) and (2,-0.5) .. (2,0);
\draw[ultra thick] (2,0) -- (3,0);

\draw[fill] (2,0) circle(.8mm);
\draw[fill] (3,0) circle(.8mm);
\draw (2.25,0) node[below]{$v_2$};
\draw (2.85,0) node[below]{$v_3$};
\draw[fill] (1,0) circle(.8mm) node[left]{$v_1$};
\draw[fill] (4.5,0) circle(.8mm) node[right]{$v_4$};
\draw [thick] (3,0) .. controls (3,0.5) and (4.5,0.5) .. (4.5,0);
\draw [thick] (3,0) .. controls (3,-0.5) and (4.5,-0.5) .. (4.5,0);

\begin{scope} [shift={(0,-1)}]

\draw[thick] (1,0) -- (2,0);
\draw[ultra thick] (2,0) -- (3,0);
\draw[thick] (3,0) -- (4.5,0);
\draw[fill] (1,0) circle(.8mm);
\draw[fill] (2,0) circle(.8mm);
\draw[fill] (3,0) circle(.8mm);
\draw[fill] (4.5,0) circle(.8mm);

\end{scope}

\end{tikzpicture} 

\end{center}

The dilation cycle on $\De$ is marked in bold and includes the end vertices. The ramification divisor is $R=v_1+v_2+v_3+v_4$. 
\label{ex:hyperellipticcurve}

\end{example}

If $\ph:G\to T$ is an unramified hyperelliptic morphism with dilation subgraph $C_{\ph}$, then the condition
$$
2g(v')+2=\kappa(v)=\val_{C_{\ph}}(v)\mbox{ for every }v=\ph(v')\in C_{\ph}
$$
implies that $\val_{C_{\ph}}(v)$ is even for all $v\in C_{\ph}$. Conversely, as we will see, any subgraph satisfying this constraint is the dilation subgraph of a unique unramified hyperelliptic morphism.

\begin{definition} Let $G$ be a graph. A {\it cycle} $C$ in $G$ is a semistable subgraph such that $\val_C(v)$ is even for all $v\in V(C)$. 

\end{definition}

We observe that a cycle on a graph with legs is not necessarily a topological cycle, since there is no point at the end of a leg. In particular, a tree with at least two legs has nontrivial cycles, and such a cycle is uniquely determined by its set of legs (this result generalizes Lemma 2.4 in~\cite{BBC} to arbitrary $D$):

\begin{proposition} Let $T$ be a tree, and let $D\subset L(T)$ be a subset of legs of $T$ with an even number of elements. Then there exists a unique cycle $C$ in $T$ whose set of legs is $D$.
\label{prop:cyclelegs}
\end{proposition}

\begin{proof} We proceed induction on the number of vertices of $T$. If $T$ consists of a single vertex $v$, then we set $V(C)=\{v\}$ and $L(C)=D$.

Now suppose that the result holds for trees with $n$ vertices. Let $T$ be a tree with $n+1$ vertices, and let $D\subset L(T)$ be a set of legs of $T$ such that $\#D$ is even. Pick a vertex $v\in V(T)$ that has a unique edge $e$ attached to it, and let $D_v\subset D$ be the legs attached to $v$. Let $T'$ be the tree obtained from $T$ by removing $v$, $e$, and $D_v$. If $\#D_v$ is even, then by induction there is a unique cycle $C'\subset T'$ whose set of legs is $D\backslash D_v$, and the required cycle $C\subset T$ is equal to $C'\cup \{v\}\cup D_v$ if $\#D_v$ is positive and $C'$ if $D_v$ is empty. If $\#D_v$ is odd, instead consider the tree $T'\cup \{e\}$, with $e$ attached as a leg. By induction, there is a unique cycle $C'\subset T'$ whose set of legs is $\{e\}\cup D\backslash D_v$. Hence the required cycle $C\subset T$ is $C'\cup \{v\}\cup \{e\} \cup D_v$.
\end{proof}

\begin{proposition} Let $\ph:G\to T$ be an unramified hyperelliptic morphism. Then the dilation subgraph $C_{\ph}$ is a cycle, called the {\it dilation cycle} of $\ph$. Conversely, given a tree $T$ and a cycle $C\subset T$, there exists a unique unramified hyperelliptic morphism $\ph:G\to T$ having dilation cycle $C$.
\label{prop:cycleunique}
\end{proposition}

\begin{proof} We have already seen above that $C_{\ph}$ is a cycle if $\ph:G\to T$ is an unramified hyperelliptic morphism. Conversely, let $C\subset T$ be a cycle on a tree. Remark~\ref{rem:cycletocover} describes how to construct a unique hyperelliptic morphism $\ph:G\to T$ having dilation subgraph $C$. The weight $g(v')$ of a vertex $v'\in V(G')$ is uniquely specified by the condition $2g(v')+2=\val_C\big(\ph(v)\big)$, where we note that $g(v')\geq 0$ because $C$ is semistable. 
\end{proof}

\begin{remark} This terminology is borrowed from Jensen and Len (see Chapter 5 of \cite{JensenLen}), who considered unramified harmonic covers $\ph:\tGa\to \Ga$ of degree two, where $\Ga$ is a tropical curve of genus $g$ with trivial weights and no legs. Such a curve has $\# H_1(\Ga,\ZZ/2\ZZ)=2^g$ different cycles, and for any cycle $C\subset \Ga$ there are $2^h$ degree two unramified harmonic covers having dilation cycle $C$, where $h=g(\Ga\backslash C)$ (see \cite[Lemma 5.8]{JensenLen}). Hence Proposition~\ref{prop:cycleunique} follows as a natural generalization to double covers of tropical curves with legs. In \cite{LUZI}, further generalize this construction to unramified harmonic covers $\ph:\tGa\to \Ga$ with an action of a finite abelian group $G$, and Proposition~\ref{prop:cycleunique} follows from \cite[Theorem~4.1]{LUZI}.

\end{remark}

In the algebraic setting, a curve is hyperelliptic if and only if it has a $g^1_2$. Additionally, a hyperelliptic curve is defined by $2g+2$ distinct points in $\PP^1$ up to the action of $\Aut (\PP^1)$. In the tropical setting, the relationship is more complex, and has been studied by a number of authors. We recall the connection.

\begin{definition} (see  \cite{AminiCaporaso}) Let $g\geq 2$, and $\Ga\in M_g^{trop}$ be a stable tropical curve with no legs. Let $\Ga^{\#}\supset \Ga$ be the tropical curve obtained by attaching $g(x)$ loops of arbitrary positive length to each $x\in \Ga$ with $g(x)>0$. For a divisor $D$ on $\Ga$, let $r^{\#}(D)$ denote its rank on $\Ga^{\#}$. We say that $\Ga$ is {\it hyperelliptic} if there exists a divisor $D$ on $\Ga$ of degree two such that $r^{\#}(D)=1$. 

\end{definition}

We now reformulate Theorems 4.12 and 4.13 of \cite{ABBRII} using our definitions. We note in passing that there is a typo in the formulation Theorem 4.12 of \cite{ABBRII}: condition (3) should not include the word "effective". 

\begin{theorem} A stable tropical curve $\Ga$ is hyperelliptic if and only if it admits a hyperelliptic morphism, which is unique if it exists. Furthermore, a hyperelliptic tropical curve is the tropicalization of an algebraic hyperelliptic curve if and only if the hyperelliptic morphism is effective. 
\label{thm:ABBRhyperelliptic}
\end{theorem}

\begin{remark} It has long been observed that tropical curves defined by Brill--Noether conditions have unexpectedly large dimensions in moduli; this is of course one of the primary motivations for our paper. For example, any genus three curve of the following kind
\begin{center}
\begin{tikzpicture}
\begin{scope}[shift={(0,1)}]
\draw [thick] (1,0) .. controls (1,0.5) and (2,0.5) .. (2,0);
\draw [thick] (1,0) .. controls (1,-0.5) and (2,-0.5) .. (2,0);
\draw[fill] (2,0) circle(.8mm);
\end{scope}
\draw[ultra thick] (2,1) -- (3,0);

\draw [thick] (0.5,0) .. controls (0.5,0.5) and (2,0.5) .. (2,0);
\draw [thick] (0.5,0) .. controls (0.5,-0.5) and (2,-0.5) .. (2,0);
\draw[fill] (2,0) circle(.8mm);
\draw[ultra thick] (2,0) -- (4.5,0);
\draw[fill] (3,0) circle(.8mm);

\begin{scope}[shift={(1.5,0)}]
\draw[fill] (3,0) circle(.8mm);
\draw [thick] (3,0) .. controls (3,0.5) and (4.5,0.5) .. (4.5,0);
\draw [thick] (3,0) .. controls (3,-0.5) and (4.5,-0.5) .. (4.5,0);
\end{scope}

\end{tikzpicture} 
\end{center}
is hyperelliptic, with the dilation subgraph of the unique hyperelliptic morphism marked in bold. Curves of this kind form a cone of maximal dimension in $M_3^{trop}$, while we expect the hyperelliptic locus to have codimension one. Requiring that the hyperelliptic curve be effective simultaneously solves the realizability problem and cuts out a locus of the correct dimension. 
\label{rem:genus3hyperelliptic}
\end{remark}

\begin{definition} Let $\Ga\in M_g^{trop}$ be a hyperelliptic curve. We say that $\Ga$ is {\it effective} if the hyperelliptic morphism on $\Ga$ (which is unique by Theorem~\ref{thm:ABBRhyperelliptic}) is effective. We denote the locus of effective hyperelliptic curves by $H_g\subset M_g^{trop}$. 

\end{definition}

Hyperelliptic curves have genus at least two by definition, so to present our results in a uniform manner we formally introduce the cone complex $H_1=M_1^{trop}=\RR_{\geq 0}$. A point $l\in \RR_{>0}$ corresponds to the tropical curve $\Ga$ consisting of two vertices $v_1$ and $v_2$ joined by two edges $e_1$ and $e_2$ of length $l$. The target tree $T$ of the hyperelliptic morphism $\eta:\Ga\to T$ has a single edge of length $l$, the hyperelliptic involution $\si:\Ga\to T$ fixes $v_1$ and $v_2$ and exchanges $e_1$ and $e_2$, and the ramification divisor of $\eta$ is $\Ram\eta= 2v_1+2v_2$. Finally, the origin $0\in H_1$ corresponds to the tropical curve $\Ga$ consisting of a single point $v$ of genus one, with $\si(v)=v$ and ramification divisor $4v$.

An algebraic hyperelliptic curve of genus $g$ is uniquely determined by $2g+2$ distinct points of $\PP^1$, and we can identify the hyperelliptic locus $\calH_g\subset \calM_g$ with the quotient $\calM_{0,2g+2}/S_{2g+2}$. The same picture holds in the tropical setting if we restrict to the locus of effective hyperelliptic curves. We summarize the relationship between effective hyperelliptic curves and $M_{0,2g+2}^{trop}$.

\begin{proposition} Let $\tDe\in M_{0,2g+2}^{trop}$ be a stable tree with legs $q_1,\ldots,q_{2g+2}$. Then there exists a unique stable curve $\tGa\in M_{g,2g+2}^{trop}$ with legs $p_1,\ldots,p_{2g+2}$ admitting an unramified hyperelliptic morphism $\tau:\tGa\to \tDe$ with $\tau(p_i)=q_i$. Denote $H_{g,2g+2}\subset M_{g,2g+2}^{trop}$ the locus of such curves, and $\pi:H_{g,2g+2}\to M_g^{trop}$ the restriction of the forgetful map to $H_{g,2g+2}$. Then $\pi(H_{g,2g+2})=H_g$, and for any $\Ga\in H_g$ the fiber $\pi^{-1}(\Ga)$ is finite. In particular, $\dim H_g=2g-1$. 
\label{prop:M02g+2}
\end{proposition}

\begin{proof} This result follows from Lemma 2.4 in~\cite{BBC}. We give a proof for the sake of completeness.

The existence and uniqueness of $\tau:\tGa\to \tDe$ follows from Propositions~\ref{prop:cyclelegs} (taking $D$ to be the set of all legs of $\tDe$) and~\ref{prop:cycleunique}. Let $\Ga$ denote the curve $\tGa$ with all legs removed, we claim that $\Ga$ is stable. Indeed, let $x\in \tGa$ be the point of attachment for $n\geq 2$ legs. If $x$ is unstable in $\Ga$, then $g(x)=0$ and $\val_{\tGa}(x)=n+1$, and the unique edge at $x$ is also dilated, so $\val_{\tDe}(\tau(x))=n+1$. Computing the ramification degree
$$
0=\Ram_{x}(\tau)=2\chi_{\tDe}(\tau(x))-\chi_{\tGa}(x)=n-1,
$$
we obtain a contradiction. It follows that $\Ga=\pi(\tGa)$, and the restriction of $\tau$ to $\Ga$ is an effective hyperelliptic morphism to the tree $\De$ obtained by removing all legs of $\tDe$. Hence $\Ga\in H_g$.

Conversely, given $\Ga\in H_g$ with effective hyperelliptic morphism $\ph:\Ga\to \De$, we reconstruct all curves $\tau:\tGa\to \tDe$ projecting to $\Ga$ by attaching the legs $p_1,\ldots,p_{2g+2}$ to the $2g+2$ points (counted with multiplicity) of the effective ramification divisor of $\ph$.
\end{proof}

We now consider the double ramification loci $DR_{g,(2,-2)}\subset M_{g,2}^{trop}$, $DR_{g,(2,-1,-1)}\subset M_{g,3}^{trop}$, and $DR_{g,(1,1,-1,-1)}\subset M_{g,4}^{trop}$ of degree two. Recall that a curve $\Ga$ lies inside a double ramification locus if a tropical modification $\tGa$ of $\Ga$ admits a finite effective harmonic morphism to a tree with specified action on the legs. We first observe that for hyperelliptic morphisms, we can in fact always assume that $\tGa=\Ga$. 

\begin{proposition} Let $\Ga$ be a tropical curve, let $\Ga'\supset\Ga$ be a tropical modification of $\Ga$, and let $\tau:\Ga'\to \De$ be a hyperelliptic morphism. Then the restriction of $\tau$ to $\Ga$ is a hyperelliptic morphism, which is effective if $\tau$ is effective. \label{prop:hyperellipticrestriction}

\end{proposition}

\begin{proof} Fix a model $\ph:G'\to G$ for $\tau$. Let $v'\in V(G')$ be an extremal vertex attached to an edge $e'\in E(G')$, and let $v=\ph(v')$ and $e=\ph(e')$. If $d_{\ph}(v')=2$, then $d_{\ph}(e')=2$, so $\ph^{-1}(v)=\{v'\}$ and $\ph^{-1}(e)=\{e'\}$, and the restriction of $\ph$ to $G'\backslash\{v',e'\}$ is a hyperelliptic morphism onto $G\backslash\{v,e\}$ (see Remark~\ref{rem:restriction}). On the other hand, if $d_{\ph}(v')=1$, then $d_{\ph}(e')=1$, and $v''=\si(v')$ is an extremal vertex attached to $e''=\si(e')$. Hence $\ph^{-1}(v)=\{v',v''\}$ and $\ph^{-1}(e)=\{e',e''\}$, and the restriction of $\ph$ to $G'\backslash\{v',v'',e',e''\}$ is a hyperelliptic morphism onto $G\backslash\{v,e\}$. Proceeding in this way, we remove all extremal edges of $G'$ and obtain a model of $\Ga$. Finally, if $\tau$ is effective, then $\ph$ is effective by Remark~\ref{rem:restriction}.
\end{proof}

We note that this result is specific for morphisms of degree two, and does not hold for higher degrees (see example following Thm.~\ref{thm:main}). 

We are now ready to determine the structure of the double ramification loci of degree two. 

\begin{proposition} Let $g\geq 1$, and denote $\pi_n:M_{g,n}^{trop}\to M_g^{trop}$ the projection maps for $n=2,3,4$. Then the images of the double ramification loci of degree two are all equal to the effective hyperelliptic locus:
$$
\pi_2\big(\DR_{g,(2,-2)}\big)=H_g,\quad \pi_3\big(\DR_{g,(2,-1,-1)}\big)=H_g,\quad \textrm{ and } \quad \pi_4\big(\DR_{g,(1,1,-1,-1)}\big)=H_g.
$$
For each $\Ga\in H_g$, the fibers 
\begin{equation*}
\big(\pi_2^{-1}(\Ga)\big)\cap \DR_{g,(2,-2)}, \quad \big(\pi_3^{-1}(\Ga)\big)\cap \DR_{g,(2,-1,-1)}, \quad \textrm{ and } \quad \big(\pi_4^{-1}(\Ga)\big)\cap \DR_{g,(1,1,-1,-1)}
\end{equation*}
have dimensions 0, 1, and 2, respectively. \label{prop:DRhyperelliptic}
\end{proposition}

\begin{proof} Let $\tGa$ be a curve lying in one of the double ramification loci $\DR_{g,(2,-2)}$, $\DR_{g,(2,-1,-1)}$, or $\DR_{g,(1,1,-1,-1)}$, and let $\Ga=\pi_n(\tGa)\in M_g^{trop}$ for the appropriate $n$. By definition, there exists an effective hyperelliptic morphism from a tropical modification of $\tGa$ to a tree. Proposition~\ref{prop:hyperellipticrestriction} implies that in fact this morphism restricts to an effective hyperelliptic morphism $\ph:\tGa\to \tDe$, where $\tDe$ is a tree with legs $0$ and $\infty$. Remove the legs from $\tGa$ and $\tDe$ to obtain the curves $\Ga_0$ and $\De$, respectively, then $\ph$ restricts to an effective hyperelliptic morphism $\ph\vert_{\Ga_0}:\Ga_0\to \De$. Since $\Ga_0$ is a tropical modification of $\Ga$, Proposition~\ref{prop:hyperellipticrestriction} implies that $\ph$ further restricts to an effective hyperelliptic morphism on $\Ga$. Hence $\Ga\in H_g$, and we have shown that 
$$
\pi_2\big(\DR_{g,(2,-2)}\big)\subseteq H_g,\quad \pi_3\big(\DR_{g,(2,-1,-1)}\big)\subseteq H_g,\quad \textrm{ and }\quad \pi_4\big(\DR_{g,(1,1,-1,-1)}\big)\subseteq H_g.
$$
We prove that these inclusions are equalities, and the dimension results for the fibers, by giving an explicit description of all curves in the three double ramification loci whose underlying unmarked curve is a given effective hyperelliptic curve. For the remainder of this section, we fix the following notation:
\begin{enumerate}
\item $\Ga\in H_g$ is an effective hyperelliptic curve.

\item $\si:\Ga\to \Ga$ is the hyperelliptic involution.

\item $\eta:\Ga\to \De$ is the hyperelliptic morphism.

\item $R\in \Div_{2g+2}(\Ga)$ is the ramification divisor of the hyperelliptic morphism. 
\end{enumerate}
\end{proof}

\begin{proposition} \label{prop:22} The curves $\tGa\in DR_{g,(2,-2)}$ with $\pi_2(\tGa)=\Ga$ are obtained by choosing points $x_1,x_2\in \Ga$ such that $x_1+x_2\leq R$, and attaching the legs $p_1$ and $p_2$ to $x_1$ and $x_2$, respectively.

\end{proposition}

\begin{proof} Pick $x_1,x_2\in \Ga$ such that $x_1+x_2\leq R$, these points are dilated. Attach $p_1$ and $p_2$ as dilated legs to $x_1$ and $x_2$, respectively, and attach legs $0$ and $\infty$ to $\De$ at $\eta(p_1)$ and $\eta(p_2)$, respectively, to obtain a hyperelliptic morphism $\ph:\tGa\to \tDe$. Attaching a single leg to the source and target reduces the ramification degree at the attaching point by one, hence the ramification divisor $\Ram\ph=R-x_1-x_2$ of $\ph$ is effective, and therefore $\tGa\in \DR_{g,(2,-2)}$. We note that there are finitely many such curves $\tGa$. 

Conversely, let $\tGa\in \DR_{g,(2,-2)}$ be such that $\pi(\tGa)=\Ga$, let $\ph:\tGa\to \tDe$ be the hyperelliptic morphism, and let $\Ga_0\subset \tGa$ be $\tGa$ with $p_1$ and $p_2$ removed. We claim that $\Ga_0$ is equal to its stabilization $\Ga$. Indeed, let $x_1\in \Ga_0$ and $x_2\in \Ga_0$ be the attachment points of $p_1$ and $p_2$, respectively. If $x_1\neq x_2$, or if $x_1=x_2$ is a point of positive genus, then $\Ga_0$ is stable. Now suppose that $x_1=x_2=x$ and $g(x)=0$. Since $\val_{\tDe}\big(\ph(x)\big)\geq 3$, the Riemann--Hurwitz condition  
$$
\Ram_{\ph}(x)=2-2\val_{\tDe}\big(\ph(x))\big)+\val_{\tGa}(x)\geq 0
$$
at $x$ implies that $\val_{\tGa}(x)\geq 4$. Therefore $\val_{\Ga_0}(x)\geq 2$, hence $\Ga_0$ is stable. Finally, comparing the ramification divisors $R=\Ram \ph|_{\Ga}=\Ram \ph+x_1+x_2$, we see that $x_1+x_2\leq R$ because $\Ram \ph$ is effective.
\end{proof}

The proofs of the following propositions are similar.

\begin{proposition} The curves $\tGa\in DR_{g,(2,-1,-1)}$ with $\pi_3(\tGa)=\Ga$ are obtained in one of the following ways. 
\begin{enumerate} \item Choose points $x_1,x_2,x_3\in \Ga$ such that $x_1\leq R$ and $\si(x_2)=x_3$. Attach the legs $p_1$, $p_2$, and $p_3$ to $\Ga$ at $x_1$, $x_2$, and $x_3$, respectively. \label{item:211type1}

\item Choose points $x_1,x_2\in \Ga$ such that $x_1+x_2\leq R$. Attach the leg $p_1$ to $x_1$, attach a dilated edge $e$ of arbitrary length to $x_2$, and attach the legs $p_2$ and $p_3$ to the other endpoint of $e$. \label{item:211type2}

\item Choose a point $x_1\in \Ga$ such that $x_1\leq R$. Attach a dilated edge $e$ of arbitrary length to $x_1$, and attach the legs $p_1,p_2,p_3$ to the other endpoint of $e$. \label{item:211type3}
\end{enumerate} \label{prop:211}
\end{proposition}

\begin{figure}

\begin{tabular}{cc}

\begin{tikzpicture}

\draw [thick] (1,0) .. controls (1,0.5) and (2,0.5) .. (2,0);
\draw [thick] (1,0) .. controls (1,-0.5) and (2,-0.5) .. (2,0);
\draw[ultra thick] (2,0) -- (2.4,0.8) node[above] {$p_1$};
\draw[ultra thick] (2,0) -- (3,0);

\draw[fill] (2,0) circle(.8mm);
\draw[fill] (3,0) circle(.8mm);
\draw [thick] (3,0) .. controls (3,0.5) and (4.5,0.5) .. (4.5,0);
\draw [thick] (3,0) .. controls (3,-0.5) and (4.5,-0.5) .. (4.5,0);
\draw[fill] (4,0.355) circle(.9mm);
\draw[fill] (4,-0.355) circle(.8mm);

\draw (3.51,0.255) -- (3.49,0.455);
\draw (3.51,-0.255) -- (3.49,-0.455);

\draw[thick] (4,0.355) -- (4.8,0.755) node[right] {$p_2$};
\draw[thick] (4,-0.355) -- (4.8,-0.755) node[right] {$p_3$};

\end{tikzpicture} 

&

\begin{tikzpicture}

\draw [thick] (1,0) .. controls (1,0.5) and (2,0.5) .. (2,0);
\draw [thick] (1,0) .. controls (1,-0.5) and (2,-0.5) .. (2,0);
\draw[fill] (2,0) circle(.8mm);
\draw[ultra thick] (2,0) -- (2.5,0.8) node[above] {$p_1$};
\draw[ultra thick] (2,0) -- (3,0);
\draw[fill] (3,0) circle(.8mm);

\draw [thick] (3,0) .. controls (3,0.5) and (4.5,0.5) .. (4.5,0);
\draw [thick] (3,0) .. controls (3,-0.5) and (4.5,-0.5) .. (4.5,0);
\draw[fill] (4.5,0) circle(.8mm);

\draw[ultra thick] (4.5,0) -- (5.5,0);
\draw[fill] (5.5,0) circle(.8mm);

\draw[thick] (5.5,0) -- (6.3,0.4) node[right]{$p_2$};
\draw[thick] (5.5,0) -- (6.3,-0.4) node[right]{$p_3$};
\draw (3.75,0.275) -- (3.75,0.475);
\draw (3.75,-0.275) -- (3.75,-0.475);

\end{tikzpicture}

\\ 
Type (\ref{item:211type1}) & Type (\ref{item:1111type2}) \\

\begin{tikzpicture}

\draw [thick] (1,0) .. controls (1,0.5) and (2,0.5) .. (2,0);
\draw [thick] (1,0) .. controls (1,-0.5) and (2,-0.5) .. (2,0);
\draw[fill] (2,0) circle(.8mm);
\draw[ultra thick] (2,0) -- (3,0);
\draw[fill] (3,0) circle(.8mm);

\draw [thick] (3,0) .. controls (3,0.5) and (4.5,0.5) .. (4.5,0);
\draw [thick] (3,0) .. controls (3,-0.5) and (4.5,-0.5) .. (4.5,0);
\draw[fill] (4.5,0) circle(.8mm);

\draw[ultra thick] (4.5,0) -- (5.5,0);
\draw[fill] (5.5,0) circle(.8mm);

\draw[thick] (5.5,0) -- (6.3,0.4) node[right]{$p_2$};
\draw[thick] (5.5,0) -- (6.3,-0.4) node[right]{$p_3$};
\draw[ultra thick] (5.5,0) -- (6.4,0) node[right] {$p_1$};

\draw (3.75,0.275) -- (3.75,0.475);
\draw (3.75,-0.275) -- (3.75,-0.475);

\end{tikzpicture}

&

\\ 
Type (\ref{item:211type3}) &  \\

\end{tabular}
\caption{Examples of marked curves in $\DR_{g,(2,-1,-1)}$ with the same underlying hyperelliptic curve of Example~\ref{ex:hyperellipticcurve}. The thick edges and the leg $p_1$ are dilated, pairs of edges with dashes have the same length.}
\label{fig:211}
\end{figure}
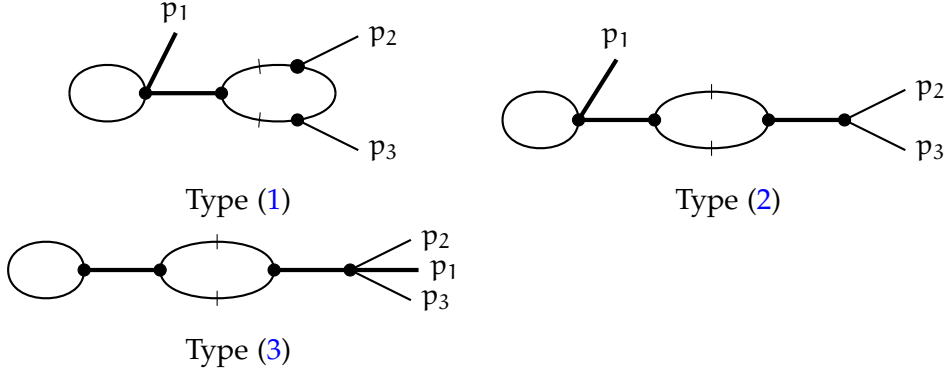

\begin{proof} Let $\tGa$ be one of the curves defined above, it is clear that $\pi_3(\tGa)=\Ga$. We obtain the target tree $\tDe$ from $\De$ as follows:

\begin{enumerate}

\item In case (\ref{item:211type1}), attach $0$ to $\eta(x_1)$ and $\infty$ to $\eta(x_2)=\eta(x_3)$, respectively.

\item In case (\ref{item:211type2}), attach $0$ to $\eta(x_1)$, an edge of length $l(e)/2$ to $\eta(x_2)$, and $\infty$ to the free endpoint of the edge.

\item In case (\ref{item:211type3}), attach an edge of length $l(e)/2$ to $\eta(x_1)$, and attach $0$ and $\infty$ to the free endpoint of the edge.

\end{enumerate}
It is clear how to extend $\eta:\Ga\to \De$ to a hyperelliptic morphism $\ph:\tGa\to \tDe$, and the restrictions on the points imply that $\ph$ is effective, hence all $\tGa$ constructed in this way lie in $DR_{g,(2,-1,-1)}$. We observe that for any $\Ga\in H_g$ constructions (\ref{item:211type2}) and (\ref{item:211type3}) give one-dimensional families of curves in $DR_{g,(2,-1,-1)}$ with underlying unmarked curve $\Ga$, and so does construction (\ref{item:211type1}) unless $\Ga=\bullet_g$. We also observe that these three one-dimensional families have common degenerations, for example setting $x_1=x_2=x_3$ in (\ref{item:211type1}) gives the same curve as choosing an edge $e$ of length zero in (\ref{item:211type3}). 

Conversely, let $\tGa\in \DR_{g,(2,-1,-1)}$ be a stable curve with dilated leg $p_1$ and undilated legs $p_2$ and $p_3$, such that $\Ga=\pi(\tGa)$. Let $\ph:\tGa\to \tDe$ be the effective hyperelliptic morphism to a tree $\tDe$ with legs $0$ and $\infty$, with $p_1$ mapping to $0$, and $p_2$ and $p_3$ mapping to $\infty$. Denote $\Ga_0\subset \tGa$ the graph $\tGa$ with the legs $p_i$ removed, so $\Ga\subset \Ga_0$ is its stabilization. 

Let $y_i\in \Ga_0$ be the points at which the legs $p_i$ are attached to $\tGa$. Then $y_1$ is dilated, and either $y_2\neq y_3$ and they are both undilated, or $y_2=y_3$ is dilated. In addition, $\si(y_2)=y_3$. Hence we have the following possibilities:

\begin{enumerate} \item The points $y_1$ are all distinct. In this case, removing the legs $p_i$ does not destabilize the graph, hence $\Ga_0=\Ga$. Denoting $x_i=y_i$ for $i=1,2,3$, we see that $\si(x_2)=x_3$, and comparing the ramification divisors $R=\Ram\ph|_{\Ga}=\Ram \ph+x_1$ we see that $x_1\leq R$. Hence we obtain a curve of type (\ref{item:211type1}).

\item $y_1\neq y_2=y_3$. If $g(y_2)\geq 1$ or $\val_{\tGa}(y_2)\geq 4$, then $\Ga_0=\Ga$ is already stable. Denoting $x_i=y_i$ for $i=1,2,3$, we see that $\si(x_2)=x_3$, and $R=\Ram \ph+x_1$, hence $x_1\leq R$ and we obtain a curve of type (\ref{item:211type1}).

Now suppose that $\chi_{\tGa}(y_2)=-1$, then $y_2$ is an unstable point of $\Ga_0$ and is the free endpoint of a dilated edge $e$ attached to the rest of $\Ga_0$ at a point $z$. If $z\neq y_1$, then removing $e$ does not destabilize $y_1$, hence $y_1\in \Ga$. Denoting $x_1=y_1$ and $x_2=y_2=y_3$, we see that $R=\Ram \ph+x_1+x_2$, where $x_1=y_2=y_3$, hence $x_1+x_2\leq R$ and we have a curve of type (\ref{item:211type2}). If $z=y_1$, we claim that $y_1\in \Ga$. Indeed, either $\tGa$ has no edges other than $e$, in which case $g(y_1)\geq 1$ and $y_1\in \Ga$, or there are additional edges at $y_1$, in which case $\val_{\tDe}(\ph(y_1))\geq 3$. The Riemann--Hurwitz condition 
$$
\Ram_{\ph}(c)=2\big(2-\val_{\tDe}(\ph(y_1))\big)-\chi_{\tGa}(y_1)\geq 0
$$
at $y_1$ then implies that $\chi_{\tGa}(y_1)\leq -2$, so $\chi_{\Ga_0}(y_1)\leq -1$ and $y_1\in \Ga$. Therefore, the curve $\tGa$ is of type (\ref{item:211type2}) with $x_1=x_2=y_1$.

\item $y_1=y_2=y_3$. If $g(y_1)\geq 1$ or $\val_{\tGa}(y_1)\geq 5$, then $\Ga_0=\Ga$ is stable. Denoting $x_i=y_i$ for $i=1,2,3$, we see that $\si(x_2)=x_3$, and $R=\Ram \ph+x_1$, hence $x_1\leq R$ and we obtain a curve of type (\ref{item:211type1}).

If $g(y_1)=0$ and $\val_{\tGa}(y_1)=4$, then $y_1$ is an unstable point of $\Ga_0$, hence $\Ga_0$ consists of $\Ga$ with a dilated edge $e$ attached at a point $x_1\in \Ga$, whose free endpoint is $y_1$. In this case $R=\Ram \ph+x_1$, hence $x_1\leq R$ and we have a curve of type (\ref{item:211type3}).

\end{enumerate}

\end{proof}

\begin{proposition} The curves $\tGa\in DR_{g,(1,1,-1,-1)}$ with $\pi_4(\tGa)=\Ga$ are obtained in one of the following ways (see Figure~\ref{fig:1111}):

\begin{enumerate} \item Choose points $x_1,x_2,x_3,x_4\in \Ga$ such that $\si(x_1)=x_2$ and $\si(x_3)=x_4$. For $i=1,2,3,4$, attach the leg $p_i$ to $\Ga$ at $x_i$. \label{item:1111type1}

\item Choose points $x_1,x_2,x_3\in \Ga$ such that $x_1\leq R$ and $\si(x_2)=x_3$. Attach a dilated edge $e$ of arbitrary length to $x_1$. Attach the legs $p_1$ and $p_2$ to the free endpoint of $e$, and attach $p_3$ and $p_4$ to $x_2$ and $x_3$, respectively. Alternatively, attach $p_3$ and $p_4$ to $e$, and $p_1$ and $p_2$ to $x_2$ and $x_3$, respectively. \label{item:1111type2}

\item Choose points $x_1,x_2\in \Ga$ such that $x_1+x_2\leq R$. Attach dilated edges $e_1$ and $e_2$ of arbitrary lengths to $x_1$ and $x_2$, respectively. Attach the legs $p_1$ and $p_2$ to $e_1$, and attach $p_3$ and $p_4$ to $e_2$. \label{item:1111type3}

\item Choose points $x_1,x_2\in \Ga$ such that $\si(x_1)=x_2$. Attach undilated edges $e_1$ and $e_2$ of equal lengths to $x_1$ and $x_2$. Attach the legs $p_1$ and $p_3$ to the free endpoint of $e_1$, and attach $p_2$ and $p_4$ to the free endpoint of $e_2$. Alternatively, attach $p_1$ and $p_4$ to $e_1$, and $p_2$ and $p_3$ to $e_2$.  \label{item:1111type4}

\item Choose a point $x_1\in \Ga$ such that $x_1\leq R$. Attach a dilated edge $e$ of arbitrary length to $x_1$, and attach undilated edges $f_1$ and $f_2$ of equal lengths to the other endpoint of $e$. Attach the legs $p_1$ and $p_3$ to the free endpoint of $f_1$, and the legs $p_2$ and $p_4$ to the free endpoint of $f_2$. Alternatively, attach $p_1$ and $p_4$ to $f_1$, and $p_2$ and $p_3$ to $f_2$. \label{item:1111type5}

\item Choose a point $x_1\in \Ga$ such that $x_1\leq R$. Attach a dilated edge $e$ to $x_1$ of arbitrary length, and attach a dilated edge $f$ of arbitrary length to the free endpoint $y$ of $e$. Attach the legs $p_1$ and $p_2$ to $y$, and attach $p_3$ and $p_4$ to the free endpoint of $f$. Alternatively, attach $p_3$ and $p_4$ to $y$, and attach $p_1$ and $p_2$ to $f$.  \label{item:1111type6}

\end{enumerate}
\label{prop:1111}
\end{proposition}

\begin{figure}

\begin{tabular}{cc}

\begin{tikzpicture}

\draw [thick] (1,0) .. controls (1,0.5) and (2,0.5) .. (2,0);
\draw [thick] (1,0) .. controls (1,-0.5) and (2,-0.5) .. (2,0);
\draw[fill] (2.5,0) circle(.8mm);
\draw[thick] (2.5,0) -- (2.1,0.8) node[above] {$p_1$};
\draw[thick] (2.5,0) -- (2.9,0.8) node[above] {$p_2$};
\draw[ultra thick] (2,0) -- (3,0);

\draw[fill] (2,0) circle(.8mm);
\draw[fill] (3,0) circle(.8mm);
\draw [thick] (3,0) .. controls (3,0.5) and (4.5,0.5) .. (4.5,0);
\draw [thick] (3,0) .. controls (3,-0.5) and (4.5,-0.5) .. (4.5,0);
\draw[fill] (4,0.355) circle(.9mm);
\draw[fill] (4,-0.355) circle(.8mm);

\draw (3.51,0.255) -- (3.49,0.455);
\draw (3.51,-0.255) -- (3.49,-0.455);

\draw[thick] (4,0.355) -- (4.8,0.755) node[right] {$p_3$};
\draw[thick] (4,-0.355) -- (4.8,-0.755) node[right] {$p_4$};

\end{tikzpicture} 

&

\begin{tikzpicture}

\draw [thick] (1,0) .. controls (1,0.5) and (2,0.5) .. (2,0);
\draw [thick] (1,0) .. controls (1,-0.5) and (2,-0.5) .. (2,0);
\draw[fill] (2,0) circle(.8mm);
\draw[thick] (2.5,0) -- (2.1,0.8) node[above] {$p_1$};
\draw[thick] (2.5,0) -- (2.9,0.8) node[above] {$p_2$};
\draw[ultra thick] (2,0) -- (3,0);
\draw[fill] (2.5,0) circle(.8mm);
\draw[fill] (3,0) circle(.8mm);

\draw [thick] (3,0) .. controls (3,0.5) and (4.5,0.5) .. (4.5,0);
\draw [thick] (3,0) .. controls (3,-0.5) and (4.5,-0.5) .. (4.5,0);
\draw[fill] (4.5,0) circle(.8mm);

\draw[ultra thick] (4.5,0) -- (5.5,0);
\draw[fill] (5.5,0) circle(.8mm);

\draw[thick] (5.5,0) -- (6.3,0.4) node[right]{$p_3$};
\draw[thick] (5.5,0) -- (6.3,-0.4) node[right]{$p_4$};
\draw (3.75,0.275) -- (3.75,0.475);
\draw (3.75,-0.275) -- (3.75,-0.475);

\end{tikzpicture}

\\ 
Type (\ref{item:1111type1}) & Type (\ref{item:1111type2}) \\

\begin{tikzpicture}

\draw [thick] (1,0) .. controls (1,0.5) and (2,0.5) .. (2,0);
\draw [thick] (1,0) .. controls (1,-0.5) and (2,-0.5) .. (2,0);
\draw[fill] (2.5,1) circle(.8mm);
\draw[ultra thick](2.5,1) -- (3,0);
\draw[thick] (2.5,1) -- (1.7,1.4) node[left] {$p_1$};
\draw[thick] (2.5,1) -- (1.7,0.6) node[left] {$p_2$};
\draw[fill] (3,0) circle(.8mm);
\draw[fill] (2,0) circle(.8mm);
\draw[ultra thick] (2,0) -- (3,0);
\draw (2.9,0);
\draw [thick] (3,0) .. controls (3,0.5) and (4.5,0.5) .. (4.5,0);
\draw [thick] (3,0) .. controls (3,-0.5) and (4.5,-0.5) .. (4.5,0);
\draw[fill] (4.5,0) circle(.8mm);
\draw[ultra thick](4.5,0) -- (5.5,0);
\draw[fill] (5.5,0) circle(.8mm);

\draw (3.75,0.275) -- (3.75,0.475);
\draw (3.75,-0.275) -- (3.75,-0.475);
\draw[thick] (5.5,0) -- (6.3,0.4) node[right] {$p_3$};
\draw[thick] (5.5,0) -- (6.3,-0.4) node[right] {$p_4$};

\end{tikzpicture} 

&

\begin{tikzpicture}

\draw [thick] (1,0) .. controls (1,0.5) and (2,0.5) .. (2,0);
\draw [thick] (1,0) .. controls (1,-0.5) and (2,-0.5) .. (2,0);
\draw[ultra thick] (2,0) -- (3,0);

\draw[fill] (2,0) circle(.8mm);

\draw[fill] (3,0) circle(.8mm);

\draw [thick] (3,0) .. controls (3,0.5) and (4.5,0.5) .. (4.5,0);
\draw [thick] (3,0) .. controls (3,-0.5) and (4.5,-0.5) .. (4.5,0);
\draw[fill] (4,0.355) circle(.9mm);
\draw[fill] (4,-0.355) circle(.8mm);

\draw (3.51,0.255) -- (3.49,0.455);
\draw (3.51,-0.255) -- (3.49,-0.455);

\draw[thick] (4,0.325) -- (5,0.6);
\draw[fill] (5,0.6) circle(.8mm);
\draw[thick] (5,0.6) -- (5.8,0.9) node[right] {$p_1$};
\draw[thick] (5,0.6) -- (5.8,0.3) node[right] {$p_3$};
\draw[thick] (4,-0.325) -- (5,-0.6);
\draw[fill] (5,-0.6) circle(.8mm);
\draw[thick] (5,-0.6) -- (5.8,-0.3) node[right] {$p_2$};
\draw[thick] (5,-0.6) -- (5.8,-0.9) node[right] {$p_4$};

\begin{scope}[shift={(4.5,0.4625)}];
\begin{scope}[shift={(0.02,0.0055)}];
\draw (-0.0275,0.1) -- (0.0275,-0.1);
\end{scope};
\begin{scope}[shift={(-0.02,-0.0055)}];
\draw (-0.0275,0.1) -- (0.0275,-0.1);
\end{scope};
\end{scope};

\begin{scope}[shift={(4.5,-0.4625)}];
\begin{scope}[shift={(0.02,-0.0055)}];
\draw (-0.0275,-0.1) -- (0.0275,0.1);
\end{scope};
\begin{scope}[shift={(-0.02,0.0055)}];
\draw (-0.0275,-0.1) -- (0.0275,0.1);
\end{scope};
\end{scope};

\end{tikzpicture}

\\ 
Type (\ref{item:1111type3}) & Type (\ref{item:1111type4}) \\

\begin{tikzpicture}

\draw [thick] (1,0) .. controls (1,0.5) and (2,0.5) .. (2,0);
\draw [thick] (1,0) .. controls (1,-0.5) and (2,-0.5) .. (2,0);
\draw[ultra thick] (2,0) -- (3,0);
\draw [thick] (3,0) .. controls (3,0.5) and (4.5,0.5) .. (4.5,0);
\draw [thick] (3,0) .. controls (3,-0.5) and (4.5,-0.5) .. (4.5,0);
\draw[ultra thick] (4.5,0) -- (5.5,0);
\draw[fill](2,0) circle(.8mm);
\draw[fill](3,0) circle(.8mm);

\draw[fill](5.5,0) circle(.8mm);
\draw[thick] (5.5,0) -- (6.5,0.6);
\draw[thick] (5.5,0) -- (6.5,-0.6);
\draw[thick] (6.5,0.6) -- (7.3,0.9) node[right] {$p_1$};
\draw[thick] (6.5,0.6) -- (7.3,0.3) node[right] {$p_3$};
\draw[fill] (6.5,0.6) circle(.8mm);
\draw[thick] (6.5,-0.6) -- (7.3,-0.9) node[right] {$p_4$};
\draw[thick] (6.5,-0.6) -- (7.3,-0.3) node[right] {$p_2$};
\draw[fill] (6.5,-0.6) circle(.8mm);
\draw (3.75,0.275) -- (3.75,0.475);
\draw (3.75,-0.275) -- (3.75,-0.475);
\draw[fill](4.5,0) circle(.8mm);
\draw (5.2,0);

\begin{scope}[shift={(0.02,0.012)}];
\draw (6.06,0.2) -- (5.94,0.4);
\end{scope};
\begin{scope}[shift={(-0.02,-0.012)}];
\draw (6.06,0.2) -- (5.94,0.4);
\end{scope};

\begin{scope}[shift={(0.02,-0.012)}];
\draw (6.06,-0.2) -- (5.94,-0.4);
\end{scope};
\begin{scope}[shift={(-0.02,0.012)}];
\draw (6.06,-0.2) -- (5.94,-0.4);
\end{scope};

\end{tikzpicture}

& 
\begin{tikzpicture}

\draw [thick] (1,0) .. controls (1,0.5) and (2,0.5) .. (2,0);
\draw [thick] (1,0) .. controls (1,-0.5) and (2,-0.5) .. (2,0);
\draw[ultra thick] (2,0) -- (3,0);
\draw [thick] (3,0) .. controls (3,0.5) and (4.5,0.5) .. (4.5,0);
\draw [thick] (3,0) .. controls (3,-0.5) and (4.5,-0.5) .. (4.5,0);
\draw[ultra thick] (4.5,0) -- (5.5,0);
\draw[fill](2,0) circle(.8mm);
\draw[fill](3,0) circle(.8mm);
\draw[fill](5.5,0) circle(.8mm);
\draw[fill](5,0) circle(.8mm);
\draw[thick] (5,0) -- (4.6,0.8) node[above] {$p_1$};
\draw[thick] (5,0) -- (5.4,0.8) node[above] {$p_2$};
\draw[thick] (5.5,0) -- (6.3,0.4) node[right] {$p_3$};
\draw[thick] (5.5,0) -- (6.3,-0.4) node[right] {$p_4$};
\draw[fill] (4.5,0) circle(.8mm);
\draw (3.75,0.275) -- (3.75,0.475);
\draw (3.75,-0.275) -- (3.75,-0.475);

\end{tikzpicture}

\\

Type (\ref{item:1111type5}) & Type (\ref{item:1111type6}) \\

\end{tabular}
\caption{Examples of marked curves in $DR_{g,(1,1,-1,-1)}$ with the same underlying hyperelliptic curve of Ex.~\ref{ex:hyperellipticcurve}. Thick edges are dilated, pairs of edges with dashes have the same length.}
\label{fig:1111}
\end{figure}
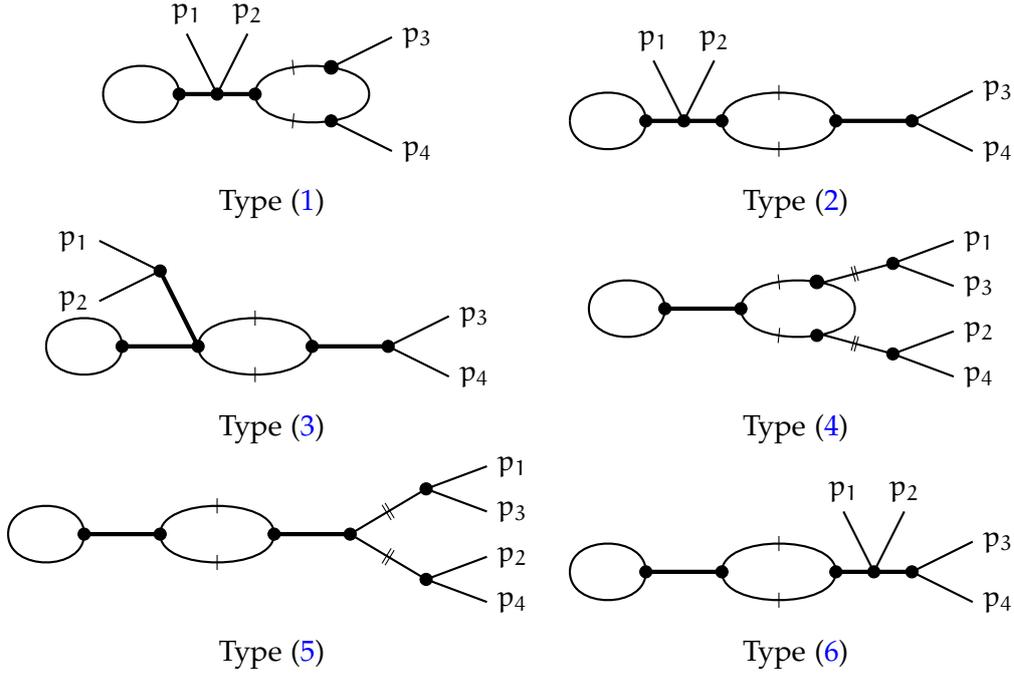

\begin{proof} Let $\tGa$ be one of the curves defined above, it is clear that $\pi_4(\tGa)=\Ga$. We obtain the target tree $\tDe$ from $\De$ as follows:

\begin{enumerate}

\item In case (\ref{item:1111type1}), attach $0$ to $\eta(x_1)=\eta(x_2)$ and $\infty$ to $\eta(x_3)=\eta(x_4)$, respectively.

\item In case (\ref{item:1111type2}), attach an edge of length $l(e)/2$ to $\eta(x_1)$. Attach $0$ to the free endpoint of the edge and $\infty$ to $\eta(x_2)=\eta(x_3)$, or vice versa. 

\item In case (\ref{item:1111type3}), attach edges $f_1$ and $f_2$ of lengths $l(e_1)/2$ and $l(e_2)/2$ to $\eta(x_1)$ and $\eta(x_2)$, respectively. Attach $0$ and $\infty$ to the free endpoints of $f_1$ and $f_2$, respectively.

\item In case (\ref{item:1111type4}), attach an edge of length $l(e_1)=l(e_2)$ to $\eta(x_1)=\eta(x_2)$, and attach $0$ and $\infty$ to the free endpoint of the edge.

\item In case (\ref{item:1111type5}), attach an edge of length $l(e)/2$ to $x_1$, and attach an edge of length $l(f_1)=l(f_2)$. Attach $0$ and $\infty$ to the free endpoint of the second edge. 

\item In case (\ref{item:1111type6}), attach an edge of length $l(e)/2$ with free endpoint $z$ to $x_1$, and attach an edge of length $l(f)/2$ to $z$. Attach $0$ to $z$ and $\infty$ to the free endpoint of the second edge, or vice versa.

\end{enumerate}

It is clear how to extend $\eta:\Ga\to \De$ to a hyperelliptic morphism $\ph:\tGa\to \tDe$, and the restrictions on the points imply that $\ph$ is effective, hence all $\tGa$ constructed in this way lie in $DR_{g,(1,1,-1,-1)}$. We observe that for any $\Ga\in H_g$ constructions (\ref{item:211type3}), (\ref{item:1111type5}), and (\ref{item:1111type6}) give two-dimensional families of curves in $DR_{g,(1,1,-1,-1)}$ with underlying unmarked curve $\Ga$, and so do constructions (\ref{item:1111type1}), (\ref{item:1111type2}), and (\ref{item:1111type4}) unless $\Ga=\bullet_g$. As in Proposition~\ref{prop:211}, there are various common degenerations between the families, for example choosing $x_1=x_2\leq R$ in (\ref{item:1111type4}) is equivalent to choosing an edge $e$ of length zero in (\ref{item:1111type5}).

Conversely, let $\tGa\in DR_{g,(1,1,-1,-1)}$ be a stable curve with undilated legs $p_1$, $p_2$, $p_3$, and $p_4$, such that $\Ga=\pi(\tGa)$. Let $\ph:\tGa\to \tDe$ be the effective hyperelliptic morphism to a tree $\tDe$ with legs $0$ and $\infty$, with $p_1$ and $p_2$ mapping to $0$, and $p_2$ and $p_3$ mapping to $\infty$. Denote $\Ga_0\subset \tGa$ the graph $\tGa$ with the legs $p_i$ removed, so $\Ga\subset \Ga_0$ is its stabilization. 

Let $y_i\in \Ga_0$ be the points at which the legs $p_i$ are attached to $\Ga_0$. Then either $y_1\neq y_2$ and they are both undilated, or $y_1=y_2$ is dilated, and similarly for the pair $y_3$ and $y_4$. In addition, $\si(y_1)=y_2$ and $\si(y_3)=y_4$. The points $y_i$ are stable in $\tGa$, and we consider whether or not they are unstable in $\Ga_0$, which may happen if some of them coincide. We have the following possibilities.

\begin{enumerate}

\item The points $y_i$ are all distinct and undilated. In this case all $y_i$ are semistable, hence $\Ga_0=\Ga$. Denoting $x_i=y_i$ for $i=1,2,3,4$, we see that $\si(x_1)=x_2$ and $\si(x_3)=x_4$, so we obtain a curve of type (\ref{item:1111type1}).

\item The points $y_1$, $y_3$, and $y_4$ are distinct, while $y_1=y_2$. The points $y_3$ and $y_4$ are undilated and semistable, while $y_1=y_2$ is dilated. If $y_1$ is semistable, then $\Ga_0=\Ga$ is already stable. Denoting $x_i=y_i$ for $i=1,2,3,4$, we see that $\si(x_1)=x_2$ and $\si(x_3)=x_4$, so we obtain a curve of type (\ref{item:1111type1}).

If $\chi_{\tGa}(y_1)=-1$, then $y_1$ is an unstable point of $\Ga_0$, and is the free endpoint of dilated edge $e$ attached a point $x_1\in \Ga_0$. Since $x_1$ is dilated it cannot be equal to $y_3$ or $y_4$, therefore removing $e$ does not destabilize $y_3$ and $y_4$. It follows that $x_1\in \Ga_0$. Denoting $x_2=y_3$ and $x_3=y_4$, we see that $R=\Ram \ph+x_1$, hence $x_1\leq R$ and we have a curve of type (\ref{item:1111type2}).

\item The points $y_1$, $y_2$, and $y_4$ are distinct, while $y_3=y_4$. This case is symmetric to the one above and yields curves of types (\ref{item:1111type1}) and (\ref{item:1111type2}).

\item $y_1=y_2\neq y_3=y_4$, and both points are dilated. If $y_1$ and $y_3$ are semistable, then $\Ga_0=\Ga$, and denoting $x_i=y_i$ for $i=1,2,3,4$ we obtain a curve of type (\ref{item:1111type1}).

If $y_1$ is semistable and $y_3$ is unstable, then $y_3$ is the free endpoint of a dilated edge $e$ attached to the rest of $\Ga_0$ at a point $z$, and there are two possibilities. If $z\neq y_1$, then $y_1\in \Ga$, since removing $e$ does not destabilize $y_1$. Denoting $x_1=z$ and $x_2=x_3=y_1$, we see that $R=\Ram \ph+x_1$, hence $x_1\leq R$ and we have a curve of type (\ref{item:1111type2}). If $z=y_1$, it may still happen that $y_1\in \Ga$, and we again obtain a curve of type (\ref{item:1111type2}) with $x_1=x_2=x_3=z$. If $z=y_1$ and $\chi_{\Ga_0}(y_1)=2$, then removing $e$ destabilizes $y_1$, hence $y_1$ is connected to $\Ga$ by a dilated edge $f$. Denoting $x_1\in \Ga$ the attaching point of $f$, we see that $R=\Ram \ph+x_1$, and our curve is of type (\ref{item:1111type6}). The case when $y_1$ is unstable and $y_2$ is semistable is treated similarly and yields curves of types (\ref{item:1111type2}) and (\ref{item:1111type6}).

Finally, suppose that $y_1$ and $y_2$ are both unstable, then they are connected to $\Ga_0$ by dilated edges $e_1$ and $e_2$ rooted at $x_1,x_2\in \Ga$. We see that $R=\Ram \ph+x_1+x_2$, and our curve is of type (\ref{item:1111type3}).

\item $y_1=y_3\neq y_2=y_4$, and both points are undilated. The hyperelliptic involution restricts to an isomorphism of the neighborhoods of $y_1$ and $y_2$ in $\Ga_0$, which implies that $y_1$ and $y_2$ are both either semistable or unstable. If $y_1$ and $y_2$ are semistable, then $\Ga_0=\Ga$, and denoting $x_i=y_i$ for $i=1,2,3,4$ we obtain a curve of type (\ref{item:1111type1}).

If $y_1$ and $y_2$ are unstable, then they are the endpoints of two undilated edges $e_1$ and $e_2$ attached to the rest of $\Ga_0$ at points $z_1$ and $z_2$. The hyperelliptic involution exchanges $e_1$ and $e_2$, so $l(e_1)=l(e_2)$ and $\si(z_1)=z_2$. If $z_1\neq z_2$, then removing $e_1$ and $e_2$ does not destabilize $z_1$ and $z_2$, hence $z_1,z_2\in \Ga$. Denoting $x_1=z_1$ and $x_2=z_2$, we obtain a curve of type (\ref{item:1111type4}). Similarly, if $z_1=z_2$ and $\chi_{\Ga_0}(z_1)\leq -2$, then $z_1\in \Ga$, and we again obtain a curve of type (\ref{item:1111type4}). Finally, if $z_1=z_2$ and $\chi_{\Ga_0}(z_1)=-1$, then $z_1$ is the endpoint of an edge $f$ attached attached to $\Ga$ at a point $x_1$. Denoting $e=f$ and $f_i=e_i$, we see that $R=\Ram \ph+x_1$, so we have a curve of type (\ref{item:1111type5}).

\item $y_1=y_4\neq y_2=y_3$, and both points are undilated. This case is symmetric to the one above and yields curves of types (\ref{item:1111type4}) and (\ref{item:1111type5}).

\item $y_1=y_2=y_3=y_4$ is a dilated point. If $\chi_{\Ga_0}(y_1)\leq -4$, then $\Ga_0=\Ga$ and denoting $x_i=y_i\in \Ga$ for $i=1,2,3,4$, we obtain a curve of type (\ref{item:1111type1}). If $\chi_{\Ga_0}(y_1)=-3$, then $y_1$ is an unstable point of $\Ga_0$, connected by a dilated edge $e$ to a point $x_1\in \Ga$. This is the degenerate case of type (\ref{item:1111type6}) when $l(f)=0$, or type (\ref{item:1111type5}) when $l(f_1)=l(f_2)=0$.
\end{enumerate}
\end{proof}



\section{Tropicalization, specialization, and realizability}

In this section we recall the moduli-theoretic and non-Archimedean approach to the process of tropicalization for marked curves and admissible cover, as pioneered in \cite{ACP} and \cite{CavalieriMarkwigRanganathan_tropadmissiblecovers}, and generalize this approach to divisors. In Section \ref{section_realizability} we then proceed to prove Theorem  \ref{thm_realizabilityofprincipaldivisors} and \ref{thm_realizabilityDR} from the introduction. From now we work over an algebraically closed field $k$ of characteristic zero carrying the trivial absolute value $\vert\cdot\vert$.


\subsection{Tropicalizing algebraic curves and their moduli spaces}

Let $X$ be a smooth projective curve over a non-Archimedean field $K$ extending $k$ and let $p_1,\ldots, p_n\in  X(K)$ be $n$ marked points. The stable reduction theorem implies that, possibly after replacing $K$ by a finite extension, there is a unique flat and proper nodal model $\calX$ of $X$ over the valuation ring $R$ of $K$ together with sections $s_1,\ldots, s_n$ of $\calX$ extending $p_1,\ldots,p_n$ such that the special fiber $(\calX_0,s_1,\ldots,s_n)$ is a stable $n$-marked curve. We define a tropical curve $\Ga_X$ associated to this datum as follows: 
\begin{itemize}
\item The underlying graph $G_X$ is the dual graph of $\calX_0$: there is a vertex $v\in V(G_X)$ for every irreducible component $\calX_v$ of $\calX_0$, there is an edge $e\in E(G_X)$ connecting vertices $u$ and $v$ for every node connecting the two components $\calX_u$ and $\calX_v$, and there is a marked leg $p_i\in L(G_X)$ attached to the vertex $v$ for each section $s_i$ intersecting $\calX_v$.
\item The weight $g(v)$ of a vertex $v\in V(G_X)$ is the genus of the normalization of $\calX_v$. 
\item The length of an edge $e\in E(G_X)$ is equal to $\val(t_e)$, where $t_e\in R$ is an element such that the node corresponding to $e$ is locally given by the equation $xy=t_e$.
\end{itemize}

A point in $\calM_{g,n}^{an}$ is given by a smooth projective algebraic curve $X/K$ of genus $g$ over a non-Archimedean extension of $k$ together with $n$ marked rational points. So we obtain a natural tropicalization map
\begin{equation*}\begin{split}
\trop_{g,n}\colon \calM_{g,n}^{an}&\longrightarrow M_{g,n}^{trop} \\ 
\big[X/K\big] & \longmapsto \big[\Ga_X\big] \ .
\end{split}\end{equation*}

The Deligne-Mumford boundary of $\calMbar_{g,n}$ has (stack-theoretically) normal crossings. Thus the open immersion $\calM_{g,n}\hookrightarrow \calMbar_{g,n}$ is a toroidal embedding. So, by \cite{Thuillier_toroidal} and \cite{ACP}, there is natural strong deformation retraction 
\begin{equation*}
\rho_{g,n}\colon\calM_{g,n}^{an}\longrightarrow \calM_{g,n}^{an}
\end{equation*}
onto a closed subset known as the \emph{non-Archimedean skeleton} $\Sigma_{g,n}$ of $\calM_{g,n}$; it naturally has the structure of a generalized cone complex. 

\begin{theorem}[\cite{ACP} Theorem 1.1]\label{thm_ACP} The tropicalization map $\trop_{g,n}\colon \calM_{g,n}^{an}\rightarrow M_{g,n}^{trop}$ has a natural section $\Phi_{g,n}\colon M_{g,n}^{trop}\rightarrow \calM_{g,n}^{an}$ that induces an isomorphism $M_{g,n}^{trop}\xrightarrow{\simeq}\Sigma_{g,n}$ and identifies $\trop_{g,n}$ with the deformation retraction $\rho_{g,n}$. 
\end{theorem}

Theorem \ref{thm_ACP}, a posteriori, shows that $\trop_{g,n}$ is well-defined, continuous, proper and surjective. 

\begin{remark}\label{remark_nonArchstack}
The space $\calM_{g,n}^{an}$ is non-Archimedean analytic stack, i.e. a category fibered in groupoids over the category of non-Archimedean analytic spaces, for which topological notions such as deformation retractions do not a priori make sense; as a remedy we implicitly work with its underlying topological space, which (topologically) agrees with the non-Archimedean analytification of the coarse moduli space $M_{g,n}$ of $\calM_{g,n}$ (see \cite[Section 3]{Ulirsch_tropisquot} for details).
\end{remark}


\subsection{Compactifying the moduli space of divisors}\label{section_moduliofdivisors}

Before we can study the specialization of divisors from algebraic to tropical curves from a moduli-theoretic point of view in Section \ref{section_specialization}, we first need to find a natural toroidal compactification of the moduli space of divisors. 

\begin{definition}
Let $g\geq 1$ and $d_+, d_-\geq 0$, and assume that either $d_+>0$ or $d_->0$ if $g=1$. Define $\calDivbar_{g,(d_+, d_-)}$ to be the fibered category over $\calMbar_g$ whose fiber over a stable curve $X\rightarrow S$ is the set of pairs $(X', D_+, D_-)$ consisting of a semistable model $X'$ of $X$ as well as two relative effective Cartier divisors $D_+$ and $D_-$ of degrees $d_+$ and $d_-$, respectively, subject to the following conditions:
\begin{enumerate}[(i)]
\item The supports of $D_+$ and $D_-$ are contained in the non-singular part of $X'$ in each fiber of $\pi\colon X'\rightarrow S$.
\item The twisted canonical divisor $K_{X'}+D_++D_-$ is $\pi$-relatively ample. 
\end{enumerate}
\end{definition}

Write $\calDiv_{g,(d_+,d_-)}$ for the restriction of $\calDivbar_{g,(d_+,d_-)}$ to the locus $\calM_g$ of smooth curves in $\calMbar_g$. It parametrizes smooth curves $X$ together with a pair of effective divisors $(D_+, D_-)$ of degrees $d_+$ and $d_-$, and can be identified with the fibered product 
$\calX_{g}^{(d_+)}\times_{\calM_g}\calX_g^{(d_-)}$ of two symmetric powers of the universal curve $\calX_g$ over $\calM_g$. 

\begin{theorem}\label{thm_Divgdd}
The fibered category $\calDivbar_{g,(d_+,d_-)}$ defines a Deligne-Mumford stack that is smooth and proper over $k$. Its coarse moduli space is projective and the complement of $\calDiv_{g,(d_+,d_-)}$ in $\calDivbar_{g,(d_+,d_-)}$ is a divisor with (stack-theoretically) normal crossings. 
\end{theorem}

\begin{proof}
Let $\epsilon=\frac{1}{\max({d_+,d_-})}$ and consider the Hassett moduli space $\calMbar_{g,\epsilon^{d_++d_-}}$ of weighted stable curves of type $(g,\epsilon,\ldots, \epsilon)$. As in \cite[Section 1]{MoellerUlirschWerner_realizability}, the moduli space $\calDivbar_{g,(d_+,d_-)}$ is the relative coarse moduli space (in the sense of \cite[Definition 3.2]{AbramovichOlssonVistoli_twistedstablemaps}) of the forgetful morphism 
\begin{equation*}
\big[\calMbar_{g,\epsilon^{d_++d_-}}/S_{d_+}\times S_{d_-}\big]\longrightarrow \calMbar_g \ .
\end{equation*}
This shows that $\calDivbar_{g,(d_+,d_-)}$ is a proper Deligne-Mumford stack with projective coarse moduli space.
The rest of the proof consists of a standard adaption of the deformation-theoretic arguments for Hassett spaces, as in \cite[Section 3.3]{Hassett}; it is left to the avid reader. 
\end{proof}

\begin{remark}
If $d_-=0$, then the moduli space $\calDivbar_{g,d}$ constructed in \cite[Section 2]{MoellerUlirschWerner_realizability} that provides us with a compactification of the moduli space of effective divisors of degree $d$ on smooth curves of genus $g$, i.e. of the $d$-th symmetric power of the universal curve over $\calM_g$. One can view this space also as a moduli space of stable quotients constructed by Marian, Oprea, and Pandharipande (see \cite[Section 4.1]{MarianOpreaPandharipande_stablequotients}).
\end{remark}

\begin{remark}
The moduli space $\calDivbar_{g,(d_+,d_-)}$ is neither equal to the fibered product $\calXbar_{g}^{(d_+)}\times_{\calMbar_g}\calXbar_g^{(d_-)}$ of the universal curves $\calXbar_g$ over $\calMbar_g$ nor to $\calDivbar_{g,d_+}\times_{\calMbar_g}\calDivbar_{g,d_-}$. 
\end{remark}


\subsection{Specialization of divisors}\label{section_specialization}

In this section, we describe a universal version of the well-known procedure for specializing divisors from algebraic to tropical curves (see~\cite{Baker_specialization}).

Let $X$ be a smooth curve over a non-Archimedean field $K$ extending $k$, and let $D$ be a divisor on $X$. Choose two effective divisors $D_+$ and $D_-$ on $X$ such that $D=D_+-D_-$. By the valuative criterion of properness for $\calDivbar_{g,(d_+,d_-)}$, there is a finite extension $K'$ of $K$ as well as a unique semistable model $\calX'$ of $X_{K'}$ such that the following conditions hold:
\begin{enumerate}[(i)]
\item The closures $\calD_{+}$ and $\calD_{-}$ of $D_+$ and $D_-$ in $\calX'$ do not meet the singularities of the special fiber $\calX'_0$ of $\calX'$.
\item The twisted canonical divisor $K_{\calX'}+\calD_{+}+\calD_{-}$ is ample on $\calX'_0$. 
\end{enumerate}
From this datum we define a divisor $\sp(D)$ on $\Ga_X$, called the \emph{specialization} of $D$ to $\Ga_X$, as follows. The tropical curve associated to $\calX'$ is $\Ga_X$, and we pick a model $G'_X$ for $\Ga_X$ having a vertex $v\in V(G'_X)$ for each irreducible component $X'_v$ of $\calX'_0$. We then set
\begin{equation*}
\sp(D)=\mdeg(\calD_{+}\vert_{\calX'_{0}})-\mdeg(\calD_{-}\vert_{\calX'_0}).
\end{equation*}
Here $\mdeg (\calD_{\pm}|_{\calX'_0})$ is the multidegree of $\calD_{\pm}$ on the irreducible components of $\calX'_0$, so the coefficient of $v\in V(\Ga'_X)$ in $\sp(D)$ is the difference of the intersection numbers of $\calD_+$ and $\calD_-$ with $\calX'_v$. We note that $\sp(D)$ does not depend on the choice of decomposition $D=D_+-D_-$.

We have a natural tropicalization map
\begin{equation*}\begin{split}
\trop_{g,(d_+,d_-)}\colon \calDiv_{g,(d_+,d_-)}^{an}&\longrightarrow \Div_{g,(d_+,d_-)}^{trop} \\ 
\big(X/K,D_+, D_-\big) & \longmapsto \big(\Ga_X, \mdeg(\calD_+\vert_{\calX_0}), \mdeg(\calD_-\vert_{\calX_0})\big) \ .
\end{split}\end{equation*}
By Theorem \ref{thm_Divgdd} the open immersion $\calDiv_{g,(d_+,d_-)}\hookrightarrow \calDivbar_{g,(d_+,d_-)}$ is stack-theoretically a toroidal embedding. So, again by \cite{Thuillier_toroidal} and \cite{ACP}, there is natural strong deformation retraction 
\begin{equation*}
\rho_{g,(d_+,d_-)}\colon\calDiv_{g,(d_+,d_-)}^{an}\longrightarrow \calDiv_{g,(d_+,d_-)}^{an}
\end{equation*}
whose image is the \emph{non-Archimedean skeleton} $\Sigma_{g,(d_+,d_-)}$ of $\calDiv_{g,(d_+,d_-)}$. 

\begin{theorem}\label{thm_Divtrop=Divskel}
The tropicalization map $\trop_{g,(d_+,d_-)}\colon \calDiv_{g,(d_+,d_-)}^{an}\rightarrow \Div_{g,(d_+,d_-)}^{trop}$ has a natural section $\Phi_{g,(d_+,d_-)}\colon \Div_{g,(d_+,d_-)}^{trop}\rightarrow \calDiv_{g,(d_+,d_-)}^{an}$ that induces an isomorphism \begin{equation*}
\Div_{g,(d_+,d_-)}^{trop}\xlongrightarrow{\simeq}\Sigma_{g,(d_+,d_-)} 
\end{equation*}
and identifies $\trop_{g,d}$ with the deformation retraction $\rho_{g,(d_+,d_-)}$. 
\end{theorem}

As above, Theorem \ref{thm_Divtrop=Divskel} implies that the tropicalization map $\trop_{g,(d_+,d_-)}$ is well-defined, continuous, proper, and surjective. 

\begin{proof}[Proof of Theorem \ref{thm_Divtrop=Divskel}]
This proof follows the already well-paved road that has been built in the proof of \cite[Theorem 1.2.1]{ACP}. The central point is that both the cones of $Div_{g,(d_+,d_-)}$ and the toroidal boundary strata of $\calDivbar_{g,(d_+,d_-)}$ are parameterized by the objects of the category $I_{g,(d_+,d_-)}$ of stable triples $(G,D_+,D_-)$, as in Section \ref{section_modulitropicaldivisors}. We leave the details of this argument to the avid reader. 
\end{proof}

\begin{proposition}\label{prop_tropdiv=func}
Let $(a^+_1,\ldots, a^+_{n_+})$ and $(a^-_1,\ldots, a_{n_-}^-)$ be partitions of $d_+$ and $d_-$ respectively. Then the natural diagram
\begin{center}\begin{tikzcd}
\calM_{g,n_++n_-}^{an} \arrow[d]\arrow[rr,"\trop_{g,n_++n_-}"]&&M_{g,n_++n_-}^{trop}\arrow[d]\\
\calDiv_{g,d_+,d_-}^{an}\arrow[rr,"\trop_{g,(d_+,d_-)}"] &&\Div_{g,d_+,d_-}^{trop}
\end{tikzcd}\end{center}
is commutative, where the vertical maps send the marked points (or legs) $p_1^+,\ldots, p_{n_+}^+$ and $p_1^-,\ldots, p_{n_-}^-$ to the pair of divisors $\sum_{i=1}^{n_+} a_i^+p_i^+$ and $\sum_{i=1}^{n_-}a_i^- p_i^-$.
\end{proposition}

\begin{proof}
Let $\epsilon=\frac{1}{d_++d_-}$. By \cite[Proposition 5.1]{Ulirsch_tropHassett}, the natural diagram
\begin{center}\begin{tikzcd}
\calM_{g,n_++n_-}^{an}\arrow[rr,"\trop_{g,n_++n_-}"]\arrow[d]&& M_{g,n_++n_-}^{trop}\arrow[d]\\
\calM_{g,\epsilon^{n_++n_-}}^{an}\arrow[rr,"\trop_{g,\epsilon^{n_++n_-}}"]&& M_{g,\epsilon^{n_++n_-}}^{trop} 
\end{tikzcd}\end{center}
is commutative. Consider the natural closed immersion $\calM_{g,\epsilon^n}\hookrightarrow \calM_{g,\epsilon^{d_++d_-}}$ that sends a pointed curve $(X,p_1^+,\ldots, p^-_{n_-})$ to the pointed curve $(X,p^+_1,\ldots,p^+_1,\ldots, p^-_{n_-},\ldots, p^-_{n_-})$ where each point $p_i^+$ is marked $a_i^+$ many times, and each $p_j^-$ is marked $a_j^-$ many times. On the tropical side, the analogous map identifies $M_{g,\epsilon^{n_++n_-}}^{trop}$ with the subcomplex of $M_{g,\epsilon^{d_++d_-}}^{trop}$ consisting of curves whose legs are split into groups of sizes $a_i^{\pm}$, with all legs in each group attached at a single point. It is clear that both inclusions naturally commute with tropicalization. In Theorem \ref{thm_Divgdd} we have seen that $\calDiv_{g,(d_+,d_-)}$ is a quotient of $M_{g,\epsilon^{d_++d_-}}$ by $S_{d_+}\times S_{d_-}$ (quotients commute with analytification by \cite[Theorem 1.2.2.]{ConradTemkin}) and in Proposition \ref{prop_Divgddtrop=quot} that $\Div_{g,(d_+,d_-)}^{trop}$ is a quotient of $M_{g,\epsilon^{d_++d_-}}^{trop}$ by $S_{d_+}\times S_{d_-}$. We prove our claim by observing that $\trop_{g,\epsilon^{d_++d_-}}$ is $\left(S_{d_+}\times S_{d_-}\right)$-invariant.
\end{proof}

 
\subsection{Tropicalization of admissible covers}\label{section_tropadmissiblecovers}

We recall the tropicalization of admissible covers introduced in~\cite{CavalieriMarkwigRanganathan_tropadmissiblecovers}. Fix $g,h\geq 0$, and let $\mu=(\mu^1,\ldots, \mu^r)$ be a vector of partitions of an integer $d>0$. Denote by $\calH_{g\rightarrow h,d}(\mu)$ the moduli space of degree $d$ Hurwitz covers with ramification profile $\mu$, parametrizing morphisms $f:X'\rightarrow X$ of the following kind:

\begin{itemize}

\item $X'$ is a smooth projective curve of genus $g$ with marked points $p'_{ij}$ for $i=1,\ldots,r$ and $j=1,\ldots,|\mu^i|$, and additional marked points $q'_{ij}$ for $i=1,\ldots,s$ and $j=1,\ldots,d-1$. 

\item $X$ is a smooth projective curve of genus $h$ with marked points $p_i$ for $i=1,\ldots,r$, and additional marked points $q_i$ for $i=1,\ldots,s$.

\item $f$ is a degree $d$ morphism mapping $p'_{ij}$ to $p_i$ with ramification profile $\mu^i$ and $q'_{ij}$ to $q_i$ with simple ramification profile $(2,1,\ldots,1)$, and no other ramification points. 

\end{itemize}

As in Chapter~\ref{sec:tropadmissible}, the number $s$ of simple ramification points is uniquely determined by Eq.~\eqref{eq:explicits}. We write $\calHbar_{g\rightarrow h,d}(\mu)$ for the compactification of $\calH_{g\rightarrow,d}(\mu)$ as a space of admissible covers as in \cite{HarrisMumford}\footnote{We secretly work with the normalization of $\calH_{g\rightarrow h,d}(\mu)$, constructed in \cite{AbramovichCortiVistoli} as a moduli space of twisted stable maps to $\mathbf{B}S_n$, or rather a cover thereof taking into account markings on the source curve (as in \cite{JarvisKaufmannKimura, SchmittvanZelm}). For the purpose of this paper it is safe to ignore this extra complication.}.

We define a natural tropicalization map 
\begin{equation*}
\trop_{g\rightarrow h,d}(\mu)\colon \calH_{g\rightarrow h,d}^{an}(\mu)\longrightarrow H_{g\rightarrow h,d}^{trop}(\mu) 
\end{equation*} 
as follows. A point in $\calH_{g\rightarrow h,d}^{an}(\mu)$ is given by a degree $d$ Hurwitz cover $f\colon X'\rightarrow X$ defined over a non-Archimedean extension $K$ of $k$, with ramification points and local degrees as described above. The valuative criterion for properness, applied to $\calHbar_{g\rightarrow h,d}(\mu)$, yields (possibly after a finite extension of $K$) a unique extension of $X'\rightarrow X$ to a family of admissible covers $\calX'\rightarrow\calX$ over the valuation ring $R$ of $K$. This family defines an unramified harmonic morphism $\ph:\Ga_{X'}\to \Ga_X$ of tropical curves as follows.
\begin{itemize}
\item The family restricts to a map of the special fibers $\calX'_0\to \calX_0$, which sends components to components, nodes to nodes, and marked points to marked points. Hence there is an induced map $\ph\colon G_{X'}\rightarrow G_X$ on the dual graphs of the special fibers.
 
\item Consider a node in $\calX_0$, corresponding to an edge $e\in E(G_X)$ and having local equation $xy=t_e$. A point over this node is a node in $\calX'_0$, corresponding to an edge $e'$ with local equation $x'y'=t_{e'}$, and furthermore there exists an integer $r_{e'}\geq 1$ such that $t_e=(t_{e'})^{r_{e'}}$ and the morphism is locally given by $x=(x')^{r_{e'}}$ and $y=(y')^{r_{e'}}$. We set the degree of $\ph$ on $e'$ to be $d_{\ph}(e')=r_{e'}$. For a vertex $v'\in V(G_{X'})$ the degree $d_{\ph}(v')$ is the degree of the map $X'_{v'}\to X_{\ph(v')}$, while on the legs the degrees of $\ph$ are $\mu_{ij}$ on $p'_{ij}$ and $(2,1,\ldots,1)$ on $q'_{ij}$.

\end{itemize}

Identifying $\ph$ with the restriction of $f^{an}\colon X^{an}\rightarrow Y^{an}$ to their minimal non-Archimedean skeletons, we find that $\ph$ is a finite harmonic morphism by \cite[Theorem A]{ABBRI}. This can also be deduced from the fact that away from nodes and marked points the reduction of $\ph$ to the special fiber is a covering space. Finally, applying the Riemann-Hurwitz formula to each component, we find that $\ph\colon \Ga_{X'}\rightarrow \Ga_X$ is unramified, and therefore defines an element 
\begin{equation*}
\trop_{g\rightarrow h,d}(\mu)\big[f:X'\to X\big]=\big[\ph\colon\Ga_{X'}\rightarrow\Ga_X\big]
\end{equation*}
in the tropical Hurwitz space $H_{g\rightarrow h,d}^{trop}(\mu)$. 

The boundary of $\calH_{g\rightarrow h,d}(\mu)$ in $\calHbar_{g\rightarrow h,d}(\mu)$ has normal crossings. So there is a strong deformation retraction $\rho_{g\rightarrow h,d}(\mu)$ from $\calH_{g\rightarrow h,d}^{an}(\mu)$ onto a closed subset $\Sigma_{g\rightarrow h,d}(\mu)$, the \emph{non-Archimedean skeleton} of $\calH_{g\rightarrow h,d}^{an}(\mu)$, as defined in \cite{Thuillier_toroidal, ACP}. The following Theorem \ref{thm_tropHurwitz} from \cite{CavalieriMarkwigRanganathan_tropadmissiblecovers} shows that these two maps are compatible.

\begin{theorem}[\cite{CavalieriMarkwigRanganathan_tropadmissiblecovers} Theorem 1]\label{thm_tropHurwitz}
The tropicalization map 
\begin{equation*}
\trop_{g\rightarrow h,d}(\mu)\colon \calH^{an}_{g\longrightarrow h,d}(\mu)\rightarrow \calH_{g\rightarrow h,d}^{trop}(\mu)
\end{equation*} 
naturally factors through the retraction $\rho_{g\rightarrow h,d}(\mu)\colon \calH^{an}_{g\rightarrow h,d}(\mu)\rightarrow \Sigma_{g\rightarrow h,d}(\mu)$ and the induced map $\Sigma_{g\rightarrow h,d}(\mu)\rightarrow\calH^{trop}_{g\rightarrow h}(\mu)$ is a strict morphism of generalized cone complexes, i.e. the restriction of it to every cone in $\Sigma_{g\rightarrow h,d}(\mu)$ maps isomorphically onto a cone in $\calH_{g\rightarrow h,d}^{trop}(\mu)$. 
\end{theorem}

Again, we a posteriori see that the tropicalization map $\trop_{g\rightarrow h,d}(\mu)$ is well-defined, continuous, and proper. It is, however, not surjective.
 
\begin{example} \label{example_Hurwitzrealizability} Let $\phi\colon \bullet_{0,6}\rightarrow \bullet_{0,3}$ be the harmonic morphism of degree $4$ from a tropical curve consisting of one genus zero vertex and $6$ legs to a tropical curve consisting of one genus zero vertex and $3$ legs, with dilation profiles $(2,2)$, $(2,2)$, $(3,1)$ along the legs. This morphism is unramified, but it cannot be the tropicalization of a ramified cover $\PP^1\rightarrow\PP^1$ with multiplicity profile $(2,2)$, $(2,2)$, $(3,1)$, since such a cover does not exist algebraically (see Example~3.4 in~\cite{ABBRII})
\end{example}

Example \ref{example_Hurwitzrealizability} is the simplest negative example of the so-called \emph{Hurwitz existence problem}. It asks when, given two genera $g,h\geq 0$ and a vector of partitions $\mu=(\mu_1,\ldots,\mu_n)$ of $d>0$ that fulfill the Riemann-Hurwitz condition, there is a ramified degree $d$ cover $X\rightarrow Y$ of compact Riemann surfaces of genera $g$ and $h$ respectively with ramification profile $\mu$. We refer the reader to \cite{PervovaPetronio_HurwitzexistenceI} for an extensive treatment of this still largely unsolved problem. 

Let $\ph:G'\to G$ be an unramified harmonic morphism of weighted graphs. For a vertex $v'\in V(G')$, we consider the degrees of $f$ on the half-edges $T_{v'}G'$. Specifically, we let $\mu_{\ph}(v')$ be the $\val_{\ph(v')}$-tuple of partitions of $\deg_{\ph}(v')$, indexed by the half-edges $T_{\ph(v')}G$. The local Hurwitz number $h_{g(v')\to g(v)}\big(\mu_{\ph}(v')\big)$ is the number of ramified covers of an algebraic curve of genus $g(v)$ by an algebraic curve of genus $g(v')$ with ramification profile $\mu_{\ph}(v')$ (weighted by automorphism). We say that $\ph$ is a cover of {\it Hurwitz type} if the local Hurwitz numbers $h_{g(v')\to g(v)}\big(\mu_{\ph}(v')\big)$ are not equal to zero for any $v'\in V(G')$.

\begin{theorem}[\cite{CavalieriMarkwigRanganathan_tropadmissiblecovers} Theorem 2]\label{thm_Hurwitzrealizability}
The image of $\trop_{g\rightarrow h,d}(\mu)$ are precisely the cones parametrizing admissible covers of Hurwitz type.
\end{theorem}

In the proof of our result we will use that the tropicalization map naturally commutes with the tautological morphisms.

\begin{theorem}[\cite{CavalieriMarkwigRanganathan_tropadmissiblecovers} Theorem 4]\label{thm_tropadmcovfunc}
Let $N=\vert \mu_1\vert+\cdots +\vert \mu_r\vert +s (d-1)$ to be the number of points in the preimage of the branch points. The tropicalization map $\trop_{g\rightarrow h,d}(\mu)\colon \calH^{an}_{g\rightarrow h,d}(\mu)\rightarrow \calH_{g\rightarrow h,d}^{trop}(\mu)$ naturally commutes with source and branch target morphisms. In other words, the diagram
\begin{center}\begin{tikzcd}
\calH^{an}_{g\rightarrow h,d}(\mu) \arrow[rrrr,"\br^{an}"]\arrow[dd,"\src^{an}"']\arrow[rrd,"\trop_{g\rightarrow h,d}(\mu)"]&&&& \calM_{h,r+s}^{an}\arrow[d,"\trop_{h,r+s}"]\\
&& H_{g\rightarrow h,d}^{trop}(\mu) \arrow[rr,"\br^{trop}"]\arrow[d,"\src^{trop}"'] && M_{h,r+s}^{trop}\\
\calM_{g,N}^{an}\arrow[rr,"\trop_{g,N}"] && M_{g,N}^{trop}
\end{tikzcd}\end{center}
is commutative.
\end{theorem}

\subsection{The realizability problem}\label{section_realizability}

Finally, we now prove Theorem \ref{thm_realizabilityofprincipaldivisors} and \ref{thm_realizabilityDR} from the introduction. We first need to following auxiliary result, which, together with Corollary \ref{cor_tree->linearequivalence}, also implies that the tropicalization of $\calDR_{g,a}$ is a semilinear subset of $\PD_{g,a}^{trop}$. 

\begin{proposition}\label{prop_factorization}
The restriction of the tropicalization map $\trop_{g,n}\colon \calM_{g,n}^{an}\longrightarrow M_{g,n}^{trop}$ to $\calDR_{g,a}^{an}$ naturally factors through $\DR_{g,a}^{trop}$. 
\end{proposition}

\begin{proof}
Let $d=\sum_{i:a_i>0}a_i=-\sum_{i:a_i<0}a_i$ be the degree of $a$ and $\mu=(\mu^1,\ldots, \mu^r)$ be a vector of partitions of $d$ such that $\mu^1=(a_i)_{i:a_i>0}$ and $\mu^2=(-a_i)_{i:a_i<0}$. Let $N=\vert\mu_1\vert+\cdots+\vert\mu_r\vert+s(d-1)$. By Theorem \ref{thm_tropadmcovfunc} and \cite[Theorem 1.2.2]{ACP} we have a natural commutative diagram
\begin{equation}\label{eq_Hurwitztropfunc}\begin{tikzcd}
\calH^{an}_{g\rightarrow h,d}(\mu)\arrow[d,"\src^{an}"] \arrow[rr,"\trop_{g\rightarrow h,d}(\mu)"] && H_{g\rightarrow h,d}^{trop}(\mu)\arrow[d,"\src^{trop}"]\\
\calM_{g,N}^{an} \arrow[d,"\textrm{forget}^{an}"] \arrow[rr,"\trop_{g,N}"]&& M_{g,N}^{trop}\arrow[d,"\textrm{forget}"]\\
\calM_{g,n}^{an} \arrow[rr,"\trop_{g,n}"]&& M_{g,n}^{trop} 
\end{tikzcd}\end{equation}
where $n=\vert \mu_1\vert +\vert \mu_2\vert$. Given a point in $\calDR_{g,a}^{an}\subseteq \calM_{g,n}^{an}$, we find that it naturally lifts to a point in $\calH_{g\rightarrow h,d}^{an}(\mu)$ for some choice of $\mu$ as above. Therefore, by the commutativity of \eqref{eq_Hurwitztropfunc}, its image under the tropicalization map $\trop_{g,n}$ also lies in $\DR_{g,a}^{trop}$. 
\end{proof}

\begin{proof}[Proof of Theorem \ref{thm_realizabilityofprincipaldivisors} and \ref{thm_realizabilityDR}]
By Theorem \ref{thm_Hurwitzrealizability}, the image of $\trop_{g\rightarrow h,d}(\mu)$ in \eqref{eq_Hurwitztropfunc} are precisely those cones in $H_{g\rightarrow h,d}^{trop}(\mu)$ that parametrize unramified covers of Hurwitz type. This immediately shows Theorem \ref{thm_realizabilityDR}. Since the diagram
\begin{equation*}\begin{tikzcd}
\calM_{g,n}^{an}\arrow[d,""] \arrow[rr,"\trop_{g,n}"]&& M_{g,n}^{trop}\arrow[d,""]\\
\calDiv_{g,d,d}^{an}\arrow[rr,"\trop_{g,d,d}"]&&\Div_{g,(d,d)}^{trop}
\end{tikzcd}\end{equation*}
commutes, by Proposition \ref{prop_tropdiv=func}, this also implies Theorem \ref{thm_realizabilityofprincipaldivisors}.
\end{proof}




\addcontentsline{toc}{section}{References}
\bibliographystyle{amsalpha}
\bibliography{biblio}{}

\addresses

\end{document}